\documentclass[11pt]{article}

\usepackage[latin1]{inputenc}
\usepackage[T1]{fontenc}
\usepackage[english]{babel}
\usepackage{amsmath,amssymb,amsthm}
\usepackage{mathrsfs} 
\usepackage{graphicx}
\usepackage{bbm}
\usepackage{color}
\usepackage{comment}
\usepackage[textsize=small]{todonotes}       % includes TO DO LIST
\usepackage{color}
\usepackage{enumitem}

\newtheorem{theorem}{Theorem}[section]

\newtheorem{definition}{Definition}[section]
\newtheorem{lemma}{Lemma}[section]

\newtheorem{proposition}[theorem]{Proposition}

\newtheorem{remark}{Remark}[section]
\newtheorem{cond}{Condition}[section]

\newcommand{\be}{\begin{equation}}
\newcommand{\ee}{\end{equation}}
\newcommand{\nn}{\nonumber}

\newcommand{\prob}{\mathbb P}
\newcommand{\expec}{\mathbb E}

\newcommand{\atanh}{{\rm atanh}}

\newcommand{\asinh}{{\rm asinh}}

\newcommand{\bbR}{\mathbb{R}}

\newcommand{\sss}{\scriptscriptstyle}

\numberwithin{equation}{section}
\newcommand{\vep}{\varepsilon}
\newcommand{\e}{{\rm e}}
\newcommand{\gs}{\left(G_{\sss N}\right)_{N\geq1}}
\newcommand{\rem}[1]{}

\def\1{{\mathchoice {1\mskip-4mu\mathrm l}      % Blackboard bold 1
{1\mskip-4mu\mathrm l}
{1\mskip-4.5mu\mathrm l} {1\mskip-5mu\mathrm l}}}
\newcommand{\indic}[1]{\1_{\{#1\}}}

\newcommand{\q}{Q_{\sss N}} %misura sui grafi
\newcommand{\m}{\mu_{G_{\sss N}}} % misura di gibbs
\newcommand{\p}{P_{\sss N}} % misura prodotto
\newcommand{\widetildep}{\widetilde{P}_{\sss N}}
\newcommand{\CMNd}{\mathrm{CM}_{\sss N}(\boldsymbol{d})}

\newcommand{\CMNtwo}{\mathrm{CM}_{\sss N}(\mathrm{\textbf{2}})}
\newcommand{\CMNonetwo}{\mathrm{CM}_{\sss N}(\textbf{1},\textbf{2})}
\newcommand{\GRGw}{\mathrm{GRG}_{\sss N}(\boldsymbol{w})}

\newcommand{\eqn}[1]{\begin{equation}#1\end{equation}}
\newcommand{\eqan}[1]{\begin{align}#1\end{align}}

\newcommand{\convd}{\stackrel{\sss {\mathcal D}}{\longrightarrow}}

\newcommand{\Poi}{{\sf Poi}}
\newcommand{\CMtwo}[1]{\mathrm{CM}_{#1}(\mathrm{\textbf{2}})}

\begin{document}

\title{Annealed central limit theorems \\
for the Ising model on random graphs}
\author{
Cristian Giardin\`a$^{\textup{{\tiny(a)}}}$,
Claudio Giberti$^{\textup{{\tiny(b)}}}$,\\
Remco van der Hofstad$^{\textup{{\tiny(c)}}}$,
Maria Luisa Prioriello$^{\textup{{\tiny(a,c)}}}$\;.
\\
{\small $^{\textup{(a)}}$
University of Modena and Reggio Emilia},
{\small via G. Campi 213/b, 41125 Modena, Italy}
\\
{\small $^{\textup{(b)}}$
University of Modena and Reggio Emilia},
{\small Via Amendola 2, 42122 Reggio Emilia, Italy}
\\
{\small $^{\textup{(c)}}$ 
Eindhoven University of Technology,}
{\small P.O. Box 513, 5600 MB Eindhoven, The Netherlands}
}
\maketitle

\pagenumbering{arabic}
%\keywords{Random graphs, Ising model, Central limit theorem, Annealed measure}

\begin{abstract}
The aim of this paper is to prove central limit theorems with respect to the annealed measure for the magnetization rescaled 
by $\sqrt{N}$ of Ising models on random graphs.  More precisely, we consider the general rank-1 inhomogeneous random graph (or generalized random graph), the 2-regular configuration model and the configuration model with degrees 1 and 2. For the generalized random graph, we first show the existence of a finite annealed inverse critical temperature $0\le \beta^{\mathrm{an}}_c < \infty$ and then prove our results in the uniqueness regime, i.e., the values of inverse temperature $\beta$ and external magnetic field $B$ for which either $\beta<\beta^{\mathrm{an}}_c$ and $B=0$, or $\beta>0$ and $B\neq 0$.
% $\beta^{\mathrm{an}}_c$ is the finite inverse annealed critical temperature  and 

In the case of the configuration model, the central limit theorem holds in the whole region of the parameters $\beta$ and $B$, because phase transitions do not exist for these systems as they are closely related to one-dimensional Ising models. Our proofs are based on explicit computations that are possible since the Ising model on the generalized random graph in the annealed setting is reduced to an inhomogeneous Curie-Weiss model, while the analysis of the configuration model with degrees only taking values 1 and 2 relies on that of the classical one-dimensional Ising model.
%The main goal of the paper is to prove central limit theorems for
%the magnetization rescaled by $\sqrt{N}$ for the Ising model on 
%random graphs. Both quenched and annealed versions are considered.
%We work in the uniqueness regime $\beta>\beta_c$ or $\beta>0$ and $B\neq0$,
%where $\beta$ is the inverse temperature, $\beta_c$ is the quenched or annealed
%critical inverse temperature and $B$ is the external magnetic field.
%In the random quenched setting our results apply to general tree-like random graphs
%(as introduced by Dembo, Montanari and further studied by Dommers and the first and 
%third author) and our proof follows that of Ellis in $\mathbb{Z}^d$. 
%Both for the averaged quenched and annealed setting, we specialize
%to three random graph models, namely the 2-regular Configuration model, 
%the Configuration model with degrees 1 and 2, and the general rank-1 
%inhomogeneous random graph. Our proofs are based on explicit computations
%for one dimensional Ising models for the configuration model examples
%and Curie-Weiss type models for the rank-1 models.
\end{abstract}

\maketitle

\section{Introduction and main results}

\subsection{Motivation}

\subsubsection{Ising models on random graphs}
The ferromagnetic {\em Ising model} is the most well-known example of statistical mechanics system describing cooperative behavior.  
Its probabilistic formulation \cite{El} amounts to an infinite family of random variables taking values in $\{-1,1\}$ (so-called spins) whose joint law is given by the Boltzmann-Gibbs  distribution. The properties of such families of random variables are crucially determined by the spatial structure where the spin variables are sitting. For instance, for the Ising model on $\mathbb{Z}^d$ with nearest-neighbor interactions, the model displays a second-order phase transition for $d\ge 2$. Furthermore, the universality prediction states 
that the precise details of the interactions are not relevant for the near-critical behavior, so that around the critical temperature each universality class is described by a single set of critical exponents.

Besides regular lattices, in recent years much attention has been devoted to the setting in which the spin variables are placed on the
vertices of {\em random graphs} \cite{AB, DM,DMSS,DMS,DG,DGH,DGH2,DGM,DGM2,LVVZ,MMS}. Such random graphs aim to model emergent properties of complex systems consisting of many interacting agents described by a network. Several studies on empirical networks have found that two random elements of the network are typically within relatively short graph distance (the so-called {\em small-world} paradigm), whereas there does not exist a typical scale for the number of neighbors that a random element has (the so-called {\em scale-free} paradigm where degrees in the network are proposed to have a power-law distribution) \cite{New, NewBook, vdH, vdHII}. 
As a consequence, there are vertices with very high degree that often play an important role in the functionality of the network.

Thus, the combination of the ferromagnetic Ising model on a random graph describes situations in which single units establish macroscopic cooperative behavior in the presence of the random and complex connectivity structure described by a network. Here the playing field has two levels of randomness: firstly, the probabilistic law of the spins and, secondly, the probability distribution 
of the graph.  
%\col
{So far, most of the studies have focused on the so-called {\em (random) quenched} state, in which the random graph is considered to be fixed once and for all. In this paper, we instead consider the {\em annealed} state: 
%that arises 
%in the presence of a dynamical evolution in which the dynamics of the graph is much faster than the dynamics of the spins and 
the Ising model at every time 
%only 
sees an average of the possible random graphs
\cite{Bia, KBHK}, rather than one realization of the graph.
% (as is more customary). 
%
%
%
%
%This means that the effective Glauber dynamics of a spin flip from $\sigma$ to $\sigma'$ in a spin model with Hamiltonian $H$ occurs with probability
%	\eqn{
%	\frac{\expec[\e^{-\beta H(\sigma')}]}{\expec[\e^{-\beta H(\sigma)}]}\wedge 1.
%	}
%The above dynamics corresponds to an extremely fast random graph dynamics in which we do not even {\em observe} the graph at any time, but merely see it averaged over the random graph distribution. This gives rise to a Glauber dynamics with (annealed) Hamiltonian equal to
%	\eqn{
%	H^{\rm an}(\sigma)=-\frac{1}{\beta}\log(\expec[\e^{-\beta H(\sigma')}]).
%	}
The annealed measure is particularly relevant for applications in socio-economic systems, in which the graph dynamic models the evolution of social acquaintances, or the brain, in which graph edge rearrangements represent the evolution of synaptic connections. We explain the role of the annealed and the quenched laws in more detail in the following section.}

%\RvdH{Is this the right place for this explanation? Claudio: I have moved the discussion to the next subsection}

\subsubsection{Annealing}
To understand the difference between the quenched and annealed settings, it is convenient to think of a microscopic dynamics yielding
the equilibrium state. For instance, one could imagine that the spins are subject to a Glauber dynamics with a reversible Boltzmann-Gibbs distribution and the graph also has its own dynamical evolution approaching the graph's stationary distribution. In general, these two dynamics are intertwined and both concur to determine the equilibrium state, i.e., the asymptotic value of an ergodic dynamical time average. The quenched and annealed state arise as follows:

\begin{itemize}
\item[(a)] In the \emph{quenched state}, the changes of the graph happen on a time-scale that is infinitely longer than the time-scale over 
which the changes of the spin variables occur. Thus in the quenched state the graph viewed by the evolving spins is frozen. One distinguishes between the {\em random quenched measure}, i.e., the random Boltzmann-Gibbs distribution of a given realization of the graph, and the {\em averaged quenched  measure}, i.e., the average of the Boltzmann-Gibbs distribution over the graph ensemble.
Several thermodynamic observables (e.g., the free energy per particle, the internal energy per particle, etc.) are self-averaging, and therefore  the random quenched values and their averaged quenched expectations do coincide in the thermodynamic limit. In the study of the {\em fluctuations} of the properly rescaled magnetization one finds a Gaussian limiting law. Interestingly, the asymptotic variances of the random quenched and averaged quenched state might be different \cite{GGPvdH1} due to local Gaussian fluctuations of graph properties.

\item[(b)] In the \emph{annealed state}, 
%the random environment evolves much faster than 
%the spin variables. As a consequence, 
the environment seen by the spins includes all possible arrangements of the random graph. 
The annealed measure  (defined later in \eqref{annealing}) is given by the stationary 
reversible measure of a Glauber spin dynamics in which the transition from a configuration $\sigma$ to 
another configuration $\sigma'$ occurs with probability
	\eqn{
	\label{ann-Glauber}
	\frac{\expec[\e^{-\beta H(\sigma')}]}{\expec[\e^{-\beta H(\sigma)}]}\wedge 1,
	}
	where $H$ is the Hamiltonian and $\expec[\cdot]$ represents the average over the graph ensemble.
The above dynamics corresponds to an extremely fast random graph dynamics in which we do not even {\em observe} the graph at any time, but merely see it averaged over the random graph distribution. This is equivalent to an {\em effective} Glauber dynamics with (annealed) Hamiltonian equal to
	\eqn{
	H^{\rm an}(\sigma)=-\frac{1}{\beta}\log(\expec[\e^{-\beta H(\sigma)}]).
	}
Thus, by construction, the annealed measure is necessarily non-random. We will be interested in the properties of the Gibbs measure corresponding to the dynamics in \eqref{ann-Glauber}, which corresponds to the stationary or infinite-time distribution of the spins under the dynamics. While the Glauber dynamics \eqref{ann-Glauber} corresponds to infinitely fast graph dynamics compared to the spin dynamics, the stationary distribution can equally well be viewed as a dynamics where the graph and the spin evolve at equal speeds, as is the more usual viewpoint in statistical mechanics. Note that, in the definition of the annealed pressure (see \eqref{def-annealed-press}), the averages taken w.r.t.\ the spins and the graph are completely symmetric, which can be seen as another argument in favor of the view that the corresponding dynamics run equally fast and that the limiting measure corresponds to the average w.r.t.\ graph and spins alike.
%and its normalization includes the average over the graphs ensemble. 
In this paper, we will study annealed central limit theorems for the ferromagnetic Ising model on random graphs, in order to deduce what the effect of annealing on the macroscopic properties of the Ising model is. 
\end{itemize}

The definition of the annealed measure in the context of Ising models on  random graphs is thus different than in other class of problems  with disorder, such as random walks in random environment \cite{CGZ}. In that context, annealing is rather similar to what here we have called the averaged quenched measure.

In disordered systems (such as spin glasses \cite{MPV, CG}), annealed disorder is usually considered to be easier to deal with mathematically,  since the average on the disorder and the thermal average are treated on the same footing. This is true whenever the edges of the graph are independent, due to the form of the Hamiltonian that allows a factorization of expectations w.r.t.\ the bond variables.  If instead the edge distribution in the graph does not have a product structure, the annealed case can actually be more difficult than the quenched case. Indeed, whereas the random-quenched case is dominated by the typical realization of the graph (often having the local structure of a random tree),  in the annealed case (as in the averaged-quenched case) the rare graph samples actually give a contribution that can not be ignored. This is due to the fact that the Ising model gives rise to {\em exponential functionals} on the random graph, and expectations of exponential functionals tend to be dominated by rare events in which the exponential functional is larger than it would be under the quenched law. Deriving such statement rigorously requires a deep understanding of the large deviation properties of random graphs, a highly interesting but also challenging topic.

In this paper, we consider graph ensembles of both types, i.e., random graphs with independent edges (these are generalized random graphs) or dependent edges (in this case, we study the configuration model). These are described in the following section.

\vspace{.3cm}
\subsection{Random graph models}
%In this sub-section we introduce the random graphs on which our spin models are defined. 
%The probability  distribution of the  spins  will  be described in the next sub-section. 
We denote by  $G_{\sss N}=(V_{\sss N},E_{\sss N})$ a random graph with vertex set $V_{\sss N} = [N]$ and edge set $E_{\sss N} \subset V_{\sss N} \times V_{\sss N}$. %and by $\mathcal G_{\sss N}$ the set of all possible graphs with $N$ vertices.
Here and in the rest of this paper, we write $[N]= \{1,\ldots,N \}$ for the vertex set of $G_{\sss N}$. For any $N\in {\mathbb N}$, we denote by  $\q$ the probability law of the random graph $G_{\sss N}$. In this work, we consider two classes of random graphs: the configuration model and the generalized random graph. We next introduce these models.

\vspace{0.3cm}
\noindent
\noindent
\subsubsection{The Generalized Random Graph}
In the generalized random graph, each vertex $i \in [N]$ receives a weight $w_i>0$. Given the weights, edges are present independently, but the occupation probabilities for different edges are {\em not} identical, 
instead, they are moderated by the vertex weights. 
%
%In some cases, we assume that the weights $\boldsymbol{w} = (w_i)_{i \in [N]}$ are a sequence of random variables, and in this case the graph is denoted by $\GRGw$. 
For a given sequence of weights $\boldsymbol{w} = (w_i)_{i \in [N]}$, the graph is denoted by $\GRGw$.
%
%\RvdH{Do we really? This changes the nature of annealed asymptotics... I think we just take a determinstic sequence of weights that satisfies Condition \ref{cond-WR-GRG} below...}
%When the weights are themselves random variables, they introduce
%a double randomness: firstly there is the randomness introduced by the weights, and secondly there is the randomness introduced %by the edge occupations, which are conditionally
%independent given the weights.
We call $I_{ij}$ the Bernoulli indicator that the edge between vertex $i$ and vertex $j$ is present and $p_{ij} = \mathbb{P}\left(I_{ij} = 1\right)$ is equal to
	\be
	p_{ij} \, = \, \frac{w_i w_j}{\ell_{\sss N} + w_i w_j},
	\ee
where $\ell_{\sss N}$ is the total vertex weight given by 
	\be
	\ell_{\sss N} = \sum_{i=1}^N w_i\;.
	\ee
Denote by $W_{\sss N} = w_{\sss I_{\sss N}}$ the weight of a uniformly chosen vertex $I_{\sss N} \in [N]$. The weight sequence  of the  generalized random graph $\GRGw$ is often assumed to satisfy a {\em regularity condition}, which is expressed as follows:

\begin{cond}[Weight regularity]  There exists a random variable $W$ such that, as $N\rightarrow\infty$,
\label{cond-WR-GRG}
\begin{itemize}
\item[$(a)$] $W_{\sss N} \stackrel{\mathcal D}{\longrightarrow} W$,
\item[$(b)$] $\mathbb{E}[W_{\sss N}] = \frac{1}{N}\sum_{i\in[N]} w_i\longrightarrow \mathbb{E}[W]< \infty$,
\item[$(c)$] $\mathbb{E}[W_{\sss N}^2] = \frac{1}{N}\sum_{i\in[N]} w_i^2 \longrightarrow \mathbb{E}[W^2]< \infty$,
\end{itemize}
where $\,\stackrel{\mathcal D}{\longrightarrow}\,$ denotes convergence in distribution. Further, we assume that $\mathbb{E}[W]>0$.
\end{cond}
In the following, we will consider deterministic sequences of weights that satisfy Condition \ref{cond-WR-GRG}. 
In many cases, one could also work with weights $\boldsymbol{w} = (w_i)_{i \in [N]}$ that are i.i.d.\ random variables. 
For the annealed setting, however, 
%this is {\em not} the case, 
one has to be careful, as we will argue in more detail in Section \ref{sec-ann-prop} below.
{Indeed, when the weights are themselves random variables, they introduce a double randomness in the random 
graphs: firstly there is the randomness introduced by the weights, and secondly there is the randomness 
introduced by the edge occupation statuses, which are conditionally independent given the weights. 
Whereas the thermodynamic properties (pressure, magnetization, etc.) in the quenched measures 
of the Ising model on $\GRGw$ are not affected by the choice of deterministic or random weight sequences,
the pressure of the annealed Ising model becomes {\em infinite} when the weights have sufficiently heavy tails.}

\subsubsection{The Configuration Model}
The configuration model is a {\em multigraph}, that is, a graph possibly having self-loops and multiple edges between pairs of vertices.
Fix an integer $N$ and consider a sequence of integers $\boldsymbol{d}=(d_i)_{i \in [N]}$. The aim is to construct an undirected multigraph with $N$ vertices, where vertex $j$ has degree $d_j$. We assume that $d_j \geq 1$ for all $j \in [N]$ and we denote the total degree in the graph $\ell_{\sss N}$ by
	\be
	\ell_{\sss N} \;:=\; \sum_{i \in [N]} d_{i}.
	\ee
We assume $\ell_{\sss N}$ to be even in order to be able to construct the graph. 

Assuming that initially $d_{j}$ half-edges are attached to each vertex $j\in [N]$, one way of obtaining a uniform multigraph with the given degree sequence is to pair the half-edges belonging to the different vertices in a uniform way. Two half-edges together form an edge, thus creating the edges in the graph. 
%To construct the multigraph where vertex $j$ has degree $d_j$ for all $j \in [N]$, we have $N$ separate vertices and incident to vertex $j$, we have $d_j$ half-edges. Every half-edge needs
%to be connected to another half-edge to build the graph. 
To construct the multigraph with degree sequence $\boldsymbol{d}$, the half-edges are numbered in an arbitrary order from 1 to $\ell_{\sss N}$. Then we start by randomly connecting the first half-edge with one of the $\ell_{\sss N} -1$ remaining half-edges. Once paired, two half-edges form a single edge of the multigraph. We continue the procedure of randomly choosing and pairing the half-edges until all half-edges are connected, and call the resulting graph the \emph{configuration model with degree sequence $\boldsymbol{d}$}, abbreviated as $\CMNd$.

%We will consider first a $\CMNd$ with $d_i = 2$ for all $i \in [N]$, called \emph{Random 2-regular graph}, and then a $\CMNd$ with $d_i \in \left\{1,2\right\}$ for all $i \in [N]$.\\

We will consider, in particular, the following models:
	\begin{itemize}
		\item[(1)] The \emph{2-regular random graph}, i.e., the configuration model with $d_i = 2$ for all $i \in [N]$, which we denote  by $\CMNtwo$.
		\item[(2)] The configuration model with $d_i \in \{1,2\}$ for all $i \in [N]$, which we denote by $\CMNonetwo$. In $\CMNonetwo$, for a given $p \in [0,1]$, we have $N -\lfloor pN\rfloor$ vertices of degree 1 and $\lfloor pN \rfloor$ vertices of degree 2. 
\end{itemize}

%\begin{remark}
%The  configuration model $\CMNonetwo$ can be implemented by assigning to each vertex degree 2 with probability $p$ or
%degree 1 with probability $1-p$, conditioned to having $\lfloor pN \rfloor$ vertices of degree 2. Another configuration model
%would be obtained by considering  the independent Bernoulli assignment  to each vertex. This  yields a random graph that 
%only on average has  $\lfloor pN \rfloor$ vertices of degree 2.  
%\end{remark}

%For both models $(1)$ and $(2)$ we see that there exists no phase transition point, i.e. the models
%are always in the one phase region.\\ 

% Remco removed this, since we only use the above setting of the CM...
%It is often required that the  degree sequence   $\boldsymbol{d}$ satisfies a {\em regularity condition}, which is defined as follows. 
%The degree sequence of the configuration model $\CMNd$  is assumed as well to satisfy a  {\em regularity condition}, which is expressed as follows.
%Denoting by $D_{\sss N} = d_{V_{\sss N}}$,  the degree of an uniformly chosen vertex $V_{\sss N} \in [N]$,  we assume that the following property is satisfied.
%\begin{cond}[Degree regularity] There exists a random variable $D$ such that, as $N\rightarrow\infty$, 
%\label{cond-DR-CM}
%\begin{itemize}
%\item[$(a)$] $D_{\sss N} \stackrel{\cal D}{\longrightarrow} D$,
%\item[$(b)$] $\mathbb{E}[D_{\sss N}] \rightarrow \mathbb{E}[D] < \infty$,
%\item[$(c)$] $\mathbb{E}[D_{\sss N}^2] \rightarrow \mathbb{E}[D^2] < \infty$,
%\end{itemize}
%Further, we assume that $\mathbb{P}(D\ge 1)=1$.
%\end{cond}
%\noindent

\subsubsection{Properties  of  $\GRGw$  and $\CMNd$}
%A large class of models for which a fairly detailed picture emerged is that of random graphs that are locally tree like  \cite{GGPvdH1, DM,DGH, DGH2}. Refraining form giving the mathematical definition of the relevant properties
%(since they are not used in this paper)   \cite{DM},  we only remark that  $\CMNd$ and $\GRGw$ are {\em locally tree-like graphs uniformly sparse} when properties $(a)$ and $(b)$ of the Conditions  \ref{cond-WR-GRG}   and  \ref{cond-DR-CM} hold \cite{vdH}. When also property  $(c)$  holds, the graphs have  {\em strongly finite mean}. 
%The properties  of local tree-likeness, uniform sparsity and strongly finite mean guarantee that our models posses a well behaved thermodynamic limit 
%(existence of thermodynamic quantities) \cite{DM,DGH} and satisfy a Central Limit  Theorem in the regime in which the equilibrium measure is unique \cite{GGPvdH1}. However, these results are of no use 
%in the present paper, since they  hold with respect to the random Boltzmann-Gibbs measure (called random quenched measure in \cite{GGPvdH1}), while the annealed measure is considered here.   

%It is straightforward to prove that the degree sequences of   $\CMNtwo$ and  $\CMNonetwo$  satisfy Condition \ref{cond-DR-CM}. For example, for $\CMNonetwo$, we have 
%$\prob(D_{\sss N}=2)=1-\prob(D_{\sss N}=1)=\lfloor pN \rfloor/N\rightarrow p$, so that  the limiting degree $D$ has distribution  $\prob(D=2)=1-\prob(D=1)=p$.

The existence of a phase transition in the structural properties of the graph depends on the asymptotic degree $D$, {i.e. the weak limit, provided it exists,  of  the sequence $(D_{N})_{N\ge 1}$ where  $D_{N}$ is the degree  a uniformly chosen vertex  $I_{N}\in [N]$ in the graph.}
In order to state this result, we introduce some notation that we will frequently rely upon.  Let the integer-valued random variable $D$ have distribution $P = (p_k)_{k\geq1}$, i.e., $\mathbb{P}(D=k) = p_k$, for $k\geq 1$. We define the \emph{size-biased law} $\rho =(\rho_k)_{k \geq 0}$ of $D$  by
\begin{equation*}
	\rho_k = \frac{\left(k+1\right)p_{k+1}}{\mathbb{E}[D]},
\end{equation*}
where the expected value of $D$ is supposed to be finite, and introduce the average value of $\rho$ by
	\be\label{nu_d}
	\nu := \sum_{k\geq0} k \rho_k = \frac{\mathbb{E}[D(D-1)]}{\mathbb{E}[D]}.
	\ee 
For $\CMNtwo$, the asymptotic degree distribution equals $\prob(D=2)=1,$ while for $\CMNonetwo$, the asymptotic degree distribution equals $\prob(D=2)=p, \prob(D=1)=1-p.$ For $\GRGw$ with asymptotic weight distribution $W$, the asymptotic degree $D$ is a mixed Poisson random variable $\Poi(W)$ where $W$ appears in Condition \ref{cond-WR-GRG}, i.e., 	
	\begin{equation*}
	\prob(D=k)=\mathbb{E}\left[\e^{-W}\frac{W^k}{k!}\right],
	\end{equation*}
see e.g., \cite[Chapter 6]{vdH}.

It is well known \cite{BJR,JL} that the above random graphs have a phase transition in their maximal component. Indeed, when $\nu>1$ a giant component exists, while for $\nu\leq 1$ the maximal component has $o(N)$ vertices. Here, since the degree distribution for $\GRGw$ is $D=\Poi(W)$, we have that $\nu = \frac{\mathbb{E}[W^2]}{\mathbb{E}[W]}$, because $\mathbb{E}[D]= \mathbb{E}[W]$ and $\mathbb{E}[D(D-1)] = \mathbb{E}[W^2]$. Thus, depending on $W$,  a giant component for $\GRGw$ may exist, while it does not exist for $\CMNtwo$ and $\CMNonetwo$ since, in these cases, $\nu\le 1$. In fact, for $\CMNtwo,$ the connectivity structure is quite interesting and explained in more detail in \cite{JL}.

\subsection{Annealed measure and thermodynamic quantities}\label{iniz_def}
We continue by introducing the ferromagnetic Ising model and the {\em annealed measure}. We define them on finite graphs with $N$ vertices and then study asymptotic results  in the limit $N\to\infty$. We denote a configuration of $N$ spins by $\sigma$, where $\sigma$ is defined on the vertices of the random graph $G_{\sss N}$ whose law is $\q$.

In our previous work \cite{GGPvdH1}, we have considered two Ising models. The \emph{random-quenched measure} $\m(\sigma)$ coincides with the random Boltzmann--Gibbs distribution, where the randomness is given by the graph $ G_{\sss N}$.The \emph{averaged-quenched measure} $\p(\sigma)$ is obtained by averaging the random Boltzmann--Gibbs distribution over all possible random graphs, i.e., $\p(\sigma)=\q(\m(\sigma))$.  

In defining the {\em annealed measure}, the numerator and denominator of the Boltzmann--Gibbs distribution $\m$ are averaged separately with respect to $\q$, as formalized in the following definition:
\begin{definition}[Annealed measure]
\label{measure}
For spin variables $\sigma = (\sigma_1,...,\sigma_{\sss N})$ taking values on the space of spin configurations $\Omega_{\sss N}=\{-1,1\}^N$, we define the \emph{annealed measure} by
%\begin{description}
%\item[(i) Random quenched measure.] 
%For a given realization $G_{\sss N}\in \mathcal G_{\sss N}$ 
%the random quenched measure coincides with the random Boltzmann--Gibbs distribution
%\begin{equation}
%\label{bg}
%\m(\sigma)=\frac{\exp \left[ \beta \sum_{(i,j)\in E_{\sss N}}{\sigma_i \sigma_j} + B \sum_{i \in[N]} {\sigma_i}  \right] }{Z_{G_{\sss N}} \left( \beta, B \right)}
%\end{equation}
%where
%\begin{equation}
%Z_{G_{\sss N}} \left( \beta, B \right)= \sum_{\sigma \in \Omega_{\sss N}} \exp \left[ \beta \sum_{(i,j)\in E_{\sss N}}{\sigma_i \sigma_j} + B \sum_{i \in[N]} {\sigma_i} \right] 
%\end{equation}
%is the partition function. Here $\beta \ge 0$ is the inverse temperature and $B\in\mathbb{R}$ is the uniform external magnetic field.
%%
%\MLP{Cut?}
%\item[(ii) Averaged quenched measure.] 
%This law is obtained by averaging the random Boltzmann--Gibbs distribution over all possible
%random graphs
%\begin{equation}
%\p(\sigma) = \q(\m(\sigma))=\q\left(\frac{\exp \left[ \beta \sum_{(i,j)\in E_{\sss N}}{\sigma_i \sigma_j} + B \sum_{i \in[N]} {\sigma_i}  \right] }{Z_{G_{\sss N}} \left( \beta, B \right)}\right).
%\end{equation}
%\MLP{Cut?}
%\item[Annealed measure.]
	\begin{equation}
		\label{annealing}
		\widetildep(\sigma) = \frac{\q\left(\exp \left[ \beta \sum_{(i,j)\in E_{\sss N}}{\sigma_i \sigma_j} 
		+ B \sum_{i \in[N]} {\sigma_i}  \right] \right)}{\q(Z_{\sss N} \left( \beta, B \right))},
	\end{equation}
where
	\begin{equation*}
		Z_{{\sss N}}(\beta, B)= 
		\sum_{\sigma \in \Omega_{\sss N}} 
		\exp \Big[\beta \sum_{(i,j)\in E_{\sss N}}{\sigma_i \sigma_j} + B \sum_{i \in[N]} {\sigma_i} \Big] 
	\end{equation*}
is the partition function. Here $\beta \ge 0$ is the inverse temperature and $B\in\mathbb{R}$ is the uniform external magnetic field.
\end{definition}
\medskip

%In the following , with a slight abuse of notation, we use the same symbol $\widetildep(\cdot)$ to denote both the annealed measure and the corresponding expectation. 
In this paper,  with a slight abuse of notation, we use the same symbols to denote both a measure and the corresponding expectation.
Moreover, we remark that the measure defined above depends sensitively on the two parameters $(\beta,B)$. However, for the sake of notation, we will drop the dependence of the measure on these parameters.  Sometimes we will use  $\text{Var}_{\mu}(X)$ to denote the variance of a random variable $X$ with law $\mu$.
\medskip

We now define the thermodynamic quantities with respect to the annealed measure:

\begin{definition}[Thermodynamic quantities \cite{DM,DGH}]
\noindent For a given $N\in\mathbb{N}$, we introduce the following thermodynamics quantities
in finite volume:
\begin{itemize}

\item[(i)]  The  {\em annealed pressure} is given by
	\be
	\label{def-annealed-press}
	\widetilde{\psi}_{\sss N} (\beta, B)  \,=\, \frac{1}{N} \log \left(\q \left(Z_{\sss N} \left( \beta, B \right)\right)\right).
	\ee

\item[(ii)] The  \emph{annealed magnetization} is given by
	\begin{equation*}
	\widetilde{M}_{\sss N} (\beta, B) \,=\, \widetildep\left(\frac{S_{\sss N}}{N}\right),
	\end{equation*}
where the \emph{total spin} is defined as
	\begin{equation*}
	S_{\sss N}= \sum_{i \in [N]} \sigma_i \;.
	\end{equation*}

%\item[(v)] The  \emph{annealed magnetization}:
%\be\label{ann_magn_{\sss N}}
%\widetilde{M}_{\sss N} (\beta, B) \,=\, \widetildep\left(\frac{S_{\sss N}}{N}\right) 
%\ee

\item[(iii)] The \emph{annealed susceptibility} equals
	\begin{equation*}
	%\label{susc_N} 
	\quad \widetilde{\chi}_{\sss N}(\beta,B) \,:=\, \frac{\partial}{\partial B} \widetilde{M}_{\sss N} (\beta, B) = 	\text{Var}_{\widetildep} \left(\frac{S_{\sss N}}{ \sqrt{N}}\right). 
	\end{equation*}

%\item[(vii)] The  \emph{annealed susceptibility}:
%\be\label{ann_susc_{\sss N}}
%\widetilde{\chi}_{\sss N}(\beta,B) \,=\,
%%\frac{1}{N} \sum_{{i \in [N]}, {j \in [N]}} \left[\m(\sigma_i \sigma_j)- \m(\sigma_i)\m(\sigma_j)\right] \,=\, 
%\frac{\partial}{\partial B} \widetilde{M}_{\sss N} (\beta, B).  
%\ee
\end{itemize}
\end{definition}

We are interested in the thermodynamic limit of these quantities, i.e., their limits as $N\to \infty$. In this limit, critical phenomena  may appear.  If ${\mathcal M}(\beta, B):=\lim_{N\to \infty} {\mathcal M}_{\sss N}(\beta, B)$, where $ {\mathcal M}_{\sss N}(\beta, B)$ 
is the average of $S_{\sss N}/N$ with respect to $\m(\cdot)$, $\p(\cdot)$ or $\widetildep(\cdot)$ and provided this limit exists,  criticality manifests itself in the behavior of the {\em spontaneous magnetization}  defined as ${\cal M}(\beta, 0^{+})=\lim_{B \to 0^{+}} {\cal M}(\beta, B)$. In more detail, the {\em critical inverse temperature}  is defined as
	\be\label{def_beta_c}
	\beta_{c}:=\inf \{\beta > 0\colon {\cal M}(\beta, 0^{+})>0\}.
	\ee 
and thus, depending on the setting, we can obtain the {\em quenched} and {\em annealed critical points} denoted by $\beta^{\mathrm{qu}}_{c}$ and $\beta^{\mathrm{an}}_{c}$, respectively.  When $0<\beta_{c} <\infty$, we say that the system undergoes a {\em phase transition} at $\beta=\beta_{c}$.

From \cite{DGH2}, we recall that, in the general setting of tree-like random graphs to which our models belong, the quenched critical inverse temperature is given by
	\be
	\label{betac_qu}
	\beta_c^{\mathrm{qu}} = \atanh(1/\nu),
	\ee
where $\nu$ is defined in (\ref{nu_d}). Let us remark that, in the quenched setting, since  $\nu\le 1$ for both  $\CMNtwo$ and $\CMNonetwo$, from \eqref{betac_qu} it follows immediately that $\beta^{\mathrm{qu}}_c= \infty$, which means that there is no quenched phase transition in these models. In the annealed setting, we will prove the absence of phase transition for $\CMNtwo$ and $\CMNonetwo$ below. On the contrary, we will see that a critical inverse temperature appears for $\GRGw$.

\subsection{Results}

We focus first on the study of the generalized random graph under the annealed measure, obtaining the Strong Law of Large Numbers (SLLN) and the Central Limit Theorem (CLT) for the total spin $S_{\sss N}$. Then we present the results in the annealed setting for the configuration models $\CMNtwo$ and $\CMNonetwo$.

\subsubsection{Results for $\GRGw$}
The proofs of the SLLN and CLT for $\GRGw$ require to investigate the uniqueness regime for $\GRGw$. For this, we first investigate the existence of the thermodynamic quantities in the infinite volume limit with respect to the annealed law. These results will be obtained in the next theorem. They show, in particular, that annealing changes the critical inverse temperature. Indeed, the annealed critical inverse temperature $\beta^{\mathrm{an}}_c$ is {\em strictly smaller} than the quenched critical inverse temperature $\beta_c^{\mathrm{qu}}$, when the latter exists. In the statement of the theorem below, we will use the notation $\mathcal{U}^{\mathrm{an}}$ for the annealed uniqueness regime, i.e.,
	\begin{equation*}
	\,\mathcal{U}^{\mathrm{an}}
	:=\, \left\{\left(\beta, B\right)\colon \beta\ge 0, B\neq0 \; 
	\mbox{or} \, \; 0< \beta < \beta_c^{\mathrm{an}}, B=0\right\}.
	\end{equation*}

\begin{theorem}[Thermodynamic limits for the annealed $\GRGw$]\label{term_lim_annealed}
Let  $\gs$ be a sequence of $\GRGw$ satisfying Condition \ref{cond-WR-GRG}. Then the following conclusions hold:
\begin{itemize}

\item[(i)]  For all $0\le \beta < \infty$ and for all $B\in \mathbb{R}$, the {\em annealed pressure} exists in the thermodynamic limit $N\to\infty$ and is given by
	\begin{equation}\label{lim_press_ann}
		\widetilde{\psi} (\beta, B)  \,:=\,  \lim_{N\rightarrow \infty} \widetilde{\psi}_{\sss N} (\beta, B),
	\end{equation}
its value is given in (\ref{annealedpGRG}).
%Moreover, $\psi (\beta, B)$ is a non-random quantity and $\psi (\beta, B)  \,=\,  \lim_{N\rightarrow \infty} \bar\psi_{\sss N} (\beta, B)$.

\item[(ii)]   For all $\left(\beta, B\right) \in \mathcal{U}^{\mathrm{an}}$,
the {\em magnetization per vertex} exists in the limit $N\to\infty$, i.e.,
	\begin{equation}\label{lim_magn_ann}
		\widetilde{M}(\beta, B) \,:=\, \lim_{N\rightarrow \infty} \widetilde{M}_{\sss N} (\beta, B).
	\end{equation}
For $B\neq 0$ the limit value $\widetilde{M}(\beta, B)$ equals $\widetilde{M}(\beta, B) \,=\, \frac{\partial}{\partial B} \widetilde{\psi} (\beta, B)$ and is given by % for $B\neq 0$, whereas $\widetilde{M}=0$ in the region $0< \beta < \beta_c^{\mathrm{an}}$, $B=0$\;. More explicitly,
	\begin{equation*}
	\widetilde{M}(\beta,B)\, = \, 
	\mathbb{E}\left[\tanh \left(\sqrt{\frac{\sinh\left(\beta\right)}{\mathbb{E}\left[W\right]}}W z^* + B\right) \right],
	\end{equation*}
where $z^{*}=z^{*}(\beta,B)$ is the solution of the fixed-point equation
	\begin{equation*}
	z \, = \, \mathbb{E}\left[\tanh \left(\sqrt{\frac{\sinh\left(\beta\right)}{\mathbb{E}\left[W\right]}}W z + B\right) 	\sqrt{\frac{\sinh\left(\beta\right)}{\mathbb{E}\left[W\right]}}\,W\right]
	\end{equation*}
and $W$ is the limiting random variable defined in Condition \ref{cond-WR-GRG}.

\item[(iii)] The {\em spontaneous magnetization} is given by
	\[
	\widetilde{{\cal M}}(\beta) := \lim_{B\rightarrow 0^+} \widetilde{M} (\beta, B) =  \left\{\begin{array}{ll} 0 \quad \; \; \; \, \mbox{if } \; \beta \in \mathcal{U}^{\mathrm{an}} \\ 
	\neq 0  \quad  \mbox{if } \; \beta \notin \mathcal{U}^{\mathrm{an}} \end{array}\right.
	\]
and the {\em annealed critical inverse temperature} is 
	\begin{equation*}
	\beta^{\mathrm{an}}_c = \asinh \left(1/\nu\right)\;,
	\end{equation*}
where $\nu$, defined in (\ref{nu_d}), is given by $\nu = \mathbb{E}[W^2]/\mathbb{E}[W]$ and $W$ is the limiting random variable introduced in Condition \ref{cond-WR-GRG}. In particular, if $\nu >1$, then $\beta_c^{\mathrm{an}} < \beta_c^{\mathrm{qu}}$.

\item[(iv)] For all $\left(\beta, B\right) \in \mathcal{U}^{\mathrm{an}}$, the {\em thermodynamic limit of the susceptibility} exists and is given by
	\begin{equation}\label{lim_susc_ann}
		\widetilde{\chi}(\beta,B) \,:=\, \lim_{N\rightarrow \infty} \widetilde{\chi}_{\sss N} (\beta, B) 
		\,=\, \frac{\partial^2}{\partial B^2} \widetilde{\psi} (\beta, B).
	\end{equation}
\end{itemize}
\end{theorem}
\medskip

Having investigated the phase diagram of the annealed Ising model on the $\GRGw$, we next state the SLLN and CLT for the total spin in the following two theorems:

\begin{theorem}[Annealed SLLN]\label{slln_ann}
Let  $\gs$ be a sequence  of $\GRGw$ graphs satisfying Condition \ref{cond-WR-GRG} then, for all $\left(\beta, B\right) \in \mathcal{U}^{\mathrm{an}}$, for any $\varepsilon > 0$ there exists a number $L=L(\varepsilon) > 0$ such that the total spin is exponentially concentrated in the form
	\be\nn
	\widetildep \left(\left | \frac{S_{\sss N}}{N}  - \widetilde{M}\right| \geq \varepsilon \right) 
	\leq \e^{-NL}  \quad \quad \mbox{for all sufficiently large N,} 
	\ee
%\begin{equation}\nonumber
%\frac{S_{\sss N}}{N} \stackrel{\exp}{\longrightarrow} \widetilde{M} \qquad \mbox{w.r.t.} \;\; \widetildep , \; \quad\qquad\quad \mbox{as}\; \; N \rightarrow \infty,
%\end{equation}
where $\widetilde{M}= \widetilde{M}(\beta,B)$ is the annealed magnetization defined in \eqref{lim_magn_ann}.
\end{theorem}

\vspace{0.3cm}
\begin{theorem}[Annealed CLT]\label{ann_CLT}
Let  $\gs$ be a sequence  of $\GRGw$ graphs satisfying Condition \ref{cond-WR-GRG}. Then, for all $\left(\beta, B\right) \in \mathcal{U}^{\mathrm{an}}$, the total spin satisfies a CLT of the form
	\begin{equation*}
	\frac{S_{\sss N} - \widetildep\left(S_{\sss N}\right) }{\sqrt{N}} \; 
	\stackrel{{\cal D}}{\longrightarrow} \; \mathcal{N}(0, \widetilde{\chi}), 
	\qquad \mbox{w.r.t.} \;\; \widetildep, \; \quad\qquad\quad \mbox{as}\; \; N \rightarrow \infty,
	\end{equation*}
where $\widetilde{\chi}=\widetilde{\chi}(\beta,B)$ is the thermodynamic limit of the annealed susceptibility defined in \eqref{lim_susc_ann} and $\mathcal{N}(0, \sigma^2)$ denotes a centered normal random variable with variance $\sigma^2$.
\end{theorem}
\medskip

The proofs of Theorems \ref{term_lim_annealed}, \ref{slln_ann} and \ref{ann_CLT} all heavily rely on the fact that the annealed $\GRGw$ gives rise to an inhomogeneous Curie-Weiss model, which is interesting in its own right. We continue by studying the annealed measure on $\CMNtwo$.

\subsubsection{Results for $\CMNtwo$}
Our main result for $\CMNtwo$ concerns its thermodynamic limits, a SLLN and a CLT for its total spin, as formulated in the following theorems:

\begin{theorem}[Thermodynamic limits for the annealed $\CMNtwo$]\label{term_lim_ann_CM2}
Let  $\gs$ be a sequence  of $\CMNtwo$ graphs. Then,  for all $\beta > 0$, $B\in\mathbb{R}$, the following hold:
\begin{itemize}

\item[(i)] The {\em annealed pressure} exists in the thermodynamic limit $N\to\infty$ and is given by
	\begin{equation*}
	\widetilde{\psi} (\beta, B)  \,:=\,  
	\lim_{N\rightarrow \infty} \widetilde{\psi}_{\sss N} (\beta, B) \,= \,\log \lambda_+(\beta, B),
	\end{equation*}
where
	\begin{equation*}
	\lambda_+(\beta, B) \,=\, \e^{\beta } \left[ \cosh(B) + \sqrt{\sinh^2(B)+\e^{-4\beta }}\right]\;.
	\end{equation*}

\item[(ii)] The {\em magnetization per vertex} exists in the limit $N\to\infty$
and is given by
	\begin{equation}\label{lim_magn_ann_CM2}
		\widetilde{M}(\beta, B) \,:=\, 
		\lim_{N\rightarrow \infty} \widetilde{M}_{\sss N} (\beta, B) \,=\, \frac{\sinh(B)}{\sqrt{\sinh^2(B) + \e^{-4\beta}}}.
	\end{equation}
%\item[(iii)] There is not an annealed phase transition for the $\CMNtwo$.
\end{itemize}
%\begin{equation}\label{lim_magn_ann_CM2}
%\widetilde{M}(\beta, B) \,:=\, \lim_{N\rightarrow \infty} \widetilde{M}_{\sss N} (\beta, B) \,=\, e^{\beta}\left[\sinh(B) + \frac{\sinh(B)\cosh(B)}{\sqrt{\sinh^2(B) + e^{-4\beta}}}\right].
%\end{equation}
%%\item[(iii)] There is not an annealed phase transition for the $\CMNtwo$.
%\end{itemize}
\end{theorem}
\noindent
{\remark
Since $\lim_{B \rightarrow 0^+} \widetilde{M}(\beta,B)=0$  for all $\beta>0$, by definition \eqref{def_beta_c} we conclude that there is no annealed phase transition for $\CMNtwo$. This is not surprising, since $\CMNtwo$ consists of a collection of disjoint cycles, and the Ising model does not have a phase transition in dimension one.}

Next we state the SLLN for the total spin in $\CMNtwo$:

\begin{theorem}[Annealed SLLN for $\CMNtwo$]\label{slln_ann-CM2}
Let  $\gs$ be a sequence  of $\CMNtwo$ graphs. Then, for all $\beta\geq 0, B \in \mathbb{R}$, for any $\varepsilon > 0$ there exists a number $L=L(\varepsilon) > 0$ such that the total spin is exponentially concentrated in the form
	\be\nn
	\widetildep \left(\left | \frac{S_{\sss N}}{N}  - \widetilde{M}\right| \geq \varepsilon \right) 
	\leq \e^{-NL}  \quad \quad \mbox{for all sufficiently large N,} 
	\ee
%\begin{equation}\nonumber
%\frac{S_{\sss N}}{N} \stackrel{\exp}{\longrightarrow} \widetilde{M} \qquad \mbox{w.r.t.} \;\; \widetildep , \; \quad\qquad\quad \mbox{as}\; \; N \rightarrow \infty,
%\end{equation}
where $\widetilde{M}= \widetilde{M}(\beta,B)$ is the annealed magnetization defined in \eqref{lim_magn_ann_CM2}.
\end{theorem}

Finally, we investigate the CLT for $\CMNtwo$:

\begin{theorem}[Annealed CLT for $\CMNtwo$]
\label{CLT_annealed1}
Let  $\gs$ be a sequence  of $\CMNtwo$ graphs. Then,  for all $\beta \geq 0$, 
$B\in\mathbb{R}$, the total spin satisfies a CLT of the form
%$B\neq0$ and for all $0< \beta < \beta_c$, $B=0$,
	\begin{equation*}
		\frac{S_{\sss N} - \widetildep \left(S_{\sss N}\right) }{\sqrt{N}} \; \stackrel{{\cal D}}{\longrightarrow} 
		\; \mathcal{N}\left(0, \chi \right), \qquad \mbox{w.r.t.} \;\; \widetildep, \; \quad \mbox{as} \; \; N \rightarrow \infty,
	\end{equation*}
where $\chi=\chi(\beta, B)$ is the thermodynamic limit of the quenched susceptibility (see \cite[Theorem 1.1]{GGPvdH1}) of the Ising model on $\CMNtwo$. Moreover, $\chi(\beta, B)$ is also equal to the susceptibility of the one-dimensional Ising model, i.e.,
	\begin{equation*}
	\chi(\beta, B)=\chi^{d=1}(\beta,B) = \frac{\cosh(B) \e^{-4\beta}}{(\sinh(B)+\e^{-4\beta})^{3/2}}\; .
	\end{equation*}
\end{theorem} 
\medskip

Theorems \ref{term_lim_ann_CM2}, \ref{slln_ann-CM2} and \ref{CLT_annealed1} are proved in Section \ref{three}. Their proofs heavily rely on the fact that $\CMNtwo$ consists of a collection of cycles, and the partition function on a cycle can be computed explicitly.

\subsubsection{Results for $\CMNonetwo$}
Our main result for $\CMNonetwo$ again concerns its thermodynamic limits, SLLN and CLT for its total spin. Some of the quantities involved in the statement
of these results are defined in Section \ref{four}.
%, here we give precise references only:

\begin{theorem}[Thermodynamic limits for the annealed  $\CMNonetwo$]\label{ann_press2_CM12}
Let $\gs$ be a sequence  of $\CMNonetwo$ graphs for a given $p \in (0,1)$. Then,  for all $\beta > 0$, $B\in\mathbb{R}$, the following hold:
\begin{itemize}

\item[(i)] The {\em annealed pressure} exists in the thermodynamic limit $N\to\infty$ and is given by
	\begin{align}\label{lim_press_ann_CM12}
		\widetilde{\psi} (\beta, B)  \,&:=\, \lim_{N\rightarrow \infty} \widetilde{\psi}_{\sss N} (\beta, B) \nn\\
		&\,= \, \log{\lambda_+}(\beta, B) + \frac{1-p}{2}\log{A_+}(\beta, B) + H(s^*,t^*),
		\end{align}
where ${A_+}(\beta, B)$ is defined in (\ref{A_piu_meno}) below,
%\be
%{A_+}(\beta, B) \,:= \,  \frac{\e^{-2\beta}\e^{B} + (\lambda_+(\beta, B)-\e^{\beta+B})^2 \e^{-B} 
	%+ 2 \e^{-\beta} (\lambda_+(\beta, B) - \e^{\beta+B})}
	%{[\e^{-2\beta} + (\lambda_+(\beta, B) - \e^{\beta+B})^2 ] \lambda_{+}(\beta, B)},
%\ee
the function $H\colon [0, \frac{1-p}{2}] \times [0,p] \rightarrow \mathbb{R}$ is defined in (\ref{defH}) below, and $(s^*,t^*)$ is the unique maximum point of $(s,t)\mapsto H(s,t)$ on $[0, \frac{1-p}{2}] \times [0,p]$. 

\item[(ii)] The {\em magnetization per vertex} exists in the limit $N\to\infty$, i.e., 
	\begin{equation}\label{lim_magn_ann_CM12}
		\widetilde{M}(\beta, B) \,:=\, \lim_{N\rightarrow \infty} \widetilde{M}_{\sss N} (\beta, B) 
		\,=\, \frac{\partial}{\partial B}\widetilde{\psi} (\beta, B),
	\end{equation}
and is given in (\ref{magn_12_expl}) below.
\end{itemize} 
\end{theorem}

\noindent
{\remark
(a) Since $\lim_{B \rightarrow 0^+} \widetilde{M}(\beta,B)=0$ for all $\beta>0$ (the explicit expression of the
magnetization is given in \eqref{magn_12_expl}), we again conclude that there is not an annealed phase transition 
also for $\CMNonetwo$. Again this is not surprising, since $\CMNonetwo$ consists of a collection of one-dimensional 
lines and cycles, and the one-dimensional Ising model does not have a phase transition.\\
{(b) \cite{dPB} proved a CLT for the number of lines of given lengths in $\CMNonetwo$. 
Leveraging on this result, we proved in \cite{GGPvdH1} the averaged quenched CLT for the total spin of  the Ising model on 
$\CMNonetwo$. We have applied the result of \cite{dPB} to compute also the  annealed pressure \eqref{def-annealed-press}, 
but obtaining a result different form \eqref{lim_press_ann_CM12}. While we are able to see numerically that the two formulas agree, 
we have no {analytic} proof that they coincide.}\\
%(b) In \cite{GGPvdH1}, we use a result by Broutin and de Panafieu \cite{dPB} on the CLT of the number of lines of given lengths in $\CMNonetwo$ to compute the averaged quenched variance in the CLT for the total spin. We could also use this result to identify the pressure in the annealed setting, giving another formula for the pressure than \eqref{lim_press_ann_CM12}. While we are numerically able to see that the two formulas agree, and they must, as they arise as limits of the same sequence, we surprisingly have no direct proof that the formulas for the pressure coincide.

}
\medskip

We next state the SLLN for the total spin in $\CMNonetwo$:

\begin{theorem}[Annealed SLLN for $\CMNonetwo$]\label{slln_ann-CM12}
Let  $\gs$ be a sequence  of $\CMNonetwo$ graphs. Then, for all $\beta\geq 0, B \in \mathbb{R}$, for any $\varepsilon > 0$ there exists a number $L=L(\varepsilon) > 0$ such that the total spin is exponentially concentrated in the form
	\begin{equation*}
	\widetildep \left(\left | \frac{S_{\sss N}}{N}  - \widetilde{M}\right| \geq \varepsilon \right) 
	\leq \e^{-NL}  \quad \quad \mbox{for all sufficiently large N,} 
	\end{equation*}
%\begin{equation}\nonumber
%\frac{S_{\sss N}}{N} \stackrel{\exp}{\longrightarrow} \widetilde{M} \qquad \mbox{w.r.t.} \;\; \widetildep , \; \quad\qquad\quad \mbox{as}\; \; N \rightarrow \infty,
%\end{equation}
where $\widetilde{M}= \widetilde{M}(\beta,B)$ is the annealed magnetization defined in \eqref{lim_magn_ann_CM12}.
\end{theorem}

We finish with the annealed CLT in $\CMNonetwo$:

\begin{theorem}[Annealed CLT for $\CMNonetwo$]
\label{CLT_annealed2}
Let  $\gs$ be a sequence  of $\CMNonetwo$ graphs. Then, for all $\beta > 0$, 
$B\in\mathbb{R}$,
%$B\neq0$ and for all $0< \beta < \beta_c$, $B=0$,
	\begin{equation*}
		\frac{S_{\sss N} - \widetildep \left(S_{\sss N}\right) }{\sqrt{N}} \; \stackrel{{\cal D}}{\longrightarrow} \; 	\mathcal{N}\left(0,\sigma^2_2\right), \qquad \mbox{w.r.t.} \;\; \widetildep, \; \quad \mbox{as} \; \; N \rightarrow \infty,
	\end{equation*}
where $\sigma^2_2$ is defined in (\ref{var_sigma_2}) below.
\end{theorem}
\medskip

Theorems \ref{ann_press2_CM12}, \ref{slln_ann-CM12} and \ref{CLT_annealed2} are proved in Section \ref{four}.

\subsection{Discussion}
\label{discussio}

\subsubsection{Properties of annealing}
\label{sec-ann-prop}

From the results described above, the following general picture emerges on the effect of annealing: 

\begin{itemize}

\item[(i)]
First of all, in the presence of a ferromagnetic phase transition, annealing can change the critical 
temperature, meaning that $\beta_c^{\mathrm{an}} < \beta_c{^{\mathrm{qu}}}$. We proved this for the rank-1 inhomogeneous graph.
For the configuration models with vertex degrees at most two that we have analyzed, it holds $\beta_c^{ \mathrm{an}} = \beta_c{^{\mathrm{qu}}} =\infty$. We conjecture that in the general case when there is a positive proportion of vertices of degree at least 3 and $\nu>1$ (so that there exists a giant component), an annealed positive critical temperature exists. We believe that this annealed critical temperature is strictly larger than the quenched critical temperature whenever the vertex degrees fluctuate and a positive proportion of the vertices have at least degree three.

\item[(ii)]
Furthermore, the annealed state satisfies a central limit theorem for the rescaled magnetization, as the quenched state does as proved in our previous paper \cite{GGPvdH1}. Unfortunately, we can only prove this for certain random graph sequences, but we believe this to be true in general. The variance of the annealed CLT and the variance of the quenched CLT are different whenever the degrees are allowed to fluctuate. We showed this in the case of the generalized random graph, where they can not be ordered because the quenched and annealed  critical temperatures are different
and the quenched and annealed susceptibilities diverge at the critical point. For $\CMNonetwo$ having zero critical temperature and fluctuating degrees
the variances are also different, and we believe the annealed variance to be {\em larger} than the quenched variance. 
%(as in the $\CMNonetwo$ with fluctuating degrees).
%In this last case
Unfortunately, we have not been able to prove this.
%\RvdH{Check this! Cristian: What is the argument?}

\item[(iii)]
From the analysis of the $\CMNtwo$, we see that both the annealed critical temperature and the annealed variance are the same of their quenched counterparts. We conjecture this behavior to occur for all random regular graphs.

{
\item[(iv)] In the $\GRGw$, when the weights $(w_i)_{i\in[N]}$ are i.i.d and such such $\prob(w_1>w)= cw^{-(\tau-1)}(1+o(1))$ for some {$\tau>1$}, the annealed partition function satisfies
	\begin{equation}
	\label{weird-PF}
	\q(Z_{\sss N}(\beta,B))=\e^{\frac{\beta N^2}{2}(1+o(1))}.
	\end{equation}
Thus, the effect of annealing of the weights is dramatic, as the pressure becomes infinite for every $\beta>0$. To see \eqref{weird-PF}, we first note that the upper bound is trivial, as $H(\sigma)\leq N(N-1)/2$. Thus, it suffices to prove a matching lower bound. With $K_{\sss N}$ the complete graph on $N$ vertices {and for $a>0$}, 
	\begin{align*}
	\prob(\GRGw=K_{\sss N})\,=\,\expec\Big[\prod_{ij}p_{ij}\Big]\,\geq\, &\prob({w}_i\in [N^a,2N^a]\forall i\in [N])\\
	&\times \expec\Big[\prod_{ij}p_{ij}\mid {w}_i\in [N^a,2N^a] \forall i\in [N]\Big].
	\end{align*}
We analyze both terms separately. Firstly, since the weights are i.i.d.,
	\begin{equation*}
	\prob({w}_i\in [N^a,2N^a]\forall i\in [N])=\prob({w}_1\in [N^a,2N^a])^N
	\geq \Big(cN^{-a(\tau-1)}\Big)^N=\e^{o(N^2)}.
	\end{equation*}
Secondly, when ${w}_i\in [N^a,2N^a]$ for every $i$, there exists $b>0$ such that
	\begin{equation*}
	p_{ij}\geq 1-bN^{1-a}.
	\end{equation*}
Therefore,
	\begin{equation*}
	\expec\Big[\prod_{ij}p_{ij}\mid {w}_i\in [N^a,2N^a]\forall i\in [N]\Big]
	\geq \Big(1-bN^{1-a}\Big)^{N(N-1)/2}=1-o(1),
	\end{equation*}
when $a>3$. Thus,
	\begin{equation*}
	\q(Z_{\sss N}(\beta,B))\geq Z_{\sss N}^{\sss K_N}(\beta,B)\e^{o(N^2)}\geq \e^{\frac{\beta N^2}{2}(1+o(1))},
	\end{equation*}
where $Z_{\sss N}^{\sss K_N}(\beta,B)$ is the partition function on the complete graph.
This proves the claim.
}

%\RvdH{This means that we have to change the discussion in the paper with Sander! I.i.d. weights do NOT make any sense!}
\end{itemize}
\medskip

We will expand on the analysis of the annealed critical behavior of Ising models on generalized random graphs in a forthcoming paper \cite{GGPvdH3}, where we study critical exponents around the annealed critical temperature and we derive non-classical asymptotic laws {\em at} criticality.

\subsubsection{CLT proof strategy}\label{sec_proof_strategy}
By applying a commonly used strategy \cite{El}, we can prove CLTs for $(S_{\sss N})_{N \ge 1 }$ by showing that the moment generating function of the rescaled total spin $V_{\sss N}=\frac{S_{\sss N} - \mathbb{E}(S_{\sss N})}{\sqrt{N}}$, converges in a neighborhood of $t=0$ to the moment generating function of a centered Gaussian random variable. The convergence can be achieved by considering the
so-called scaled \emph{cumulant generating functions}  of $S_{\sss N}$, defined as
	\begin{equation}
		\label{c_N}
		c_{\sss N}(t) = \frac{1}{N} \log{\mathbb{E} \left[\exp\left(t S_{\sss N}\right)\right]},
	\end{equation}
and by proving the convergence of the sequence  $(c_{\sss N}^{\prime\prime}(t_{\sss N}))_{N\ge 1}$ for $t_{\sss N}=o(1)$  to a finite value  $\chi$, which turns out to be the variance of the normal limit. This strategy has been followed in the quenched setting in \cite{GGPvdH1} where, specializing  $\mathbb{E}$ to the relevant measures, the CLT was proved for the Ising model on the whole class of locally tree-like random graphs in the random quenched setting, and for the $\CMNtwo$ and $\CMNonetwo$ models in the averaged quenched setting.  In the former case, the limit $c(t) := \lim_{N\to\infty} c_{\sss N}(t)$ can be established as a simple consequence of the existence of the random quenched pressure on locally tree-like graphs, while the convergence of $(c_{\sss N}^{\prime\prime}(t_{\sss N}))_{N\ge 1}$ follows from the concavity of the first derivatives of the cumulant generating functions. 
%This follows in the random quenched setting from the GHS inequality, which holds for the ferromagnetic Boltzmann-Gibbsmeasures $\mu_{G_{\sss N}}$. 
In the random quenched setting, this in turn is a consequence of the GHS inequality, which holds for the ferromagnetic Boltzmann-Gibbs measure $\mu_{G_{\sss N}}$. 
%The GHS inequality is lacking in the general case of the averaged quenched measure. 
On the other hand, under the averaged quenched measure this derivative can not be expressed in terms of the averaged quenched magnetization to exploit the GHS inequality.  
Because of that, only the $\CMNtwo$ and $\CMNonetwo$ setting have been treated in \cite{GGPvdH1} explicitly, by exploiting the structure of the graphs and 
connecting these systems to the one-dimensional Ising model.

A similar scenario is found in this paper, where the approach to the proof of the CLT described above is applied to the annealed setting, i.e., with \eqref{c_N} replaced by the {\em annealed} cumulant generating function
	\begin{equation*}
	\tilde{c}_{\sss N}(t) = \frac{1}{N} \log{\widetildep \left[\exp\left(t S_{\sss N}\right)\right]}
	\end{equation*}
that can be connected to the annealed pressure, since  $\tilde{c}_{\sss N}(t) =\widetilde{\psi}_{\sss N}(\beta,B+t)-\widetilde{\psi}_{\sss N}(\beta,B)$,  see \eqref{def-annealed-press}.

We will show by an explicit computation that the annealed pressure of $\GRGw$ coincides with that of an inhomogeneous Curie-Weiss model. From this fact, the thermodynamic limit of the annealed pressure, magnetization and susceptibility can be obtained. This again relies on the GHS inequality that is valid also for this inhomogeneous ferromagnetic system. 
Thus, for the generalized random graph, the annealed CLT can be  proven in a similar way as for the random quenched measure.  

On the other hand, the proofs  of the CLT for the configuration models do {\em not} follow from the abstract argument based on the GHS inequality, since GHS is not available in the general annealed context. Because of that, we have to explicitly control the limit $(\tilde{c}_{\sss N}^{\prime\prime}(t_{\sss N}))_{N\ge 1}$ throughout the computation of the annealed  pressure. It is relatively simple to accomplish this task in the case of the regular $\CMNtwo$ graph consisting of cycles only.  The fluctuating degree of  $\CMNonetwo$ makes the computation of the pressure and of the  limit $(\tilde{c}_{\sss N}^{\prime\prime}(t_{\sss N}))_{N\ge 1}$ much more involved. $\CMNonetwo$ consists of both lines and cycles. While the cycles give a vanishing contribution to the thermodynamic limit, the distribution of the length of the lines has to be carefully analyzed and its Gaussian fluctuations appear in the CLT for the total spin.

\subsubsection{Paper organization}
The rest of the paper is organized as follows. 
%We define the models,  state and discuss our main results in this Section, while
%the rest of the paper is devoted to the proofs of the theorems. 
In Section \ref{two} we deal with $\GRGw$ for which we compute the pressure and magnetization in the thermodynamic limit, identify the critical temperature and then prove the SLLN and CLT.  All of these results rely on the fact that the Ising model on $\GRGw$ in the annealed setting turns into an inhomogeneous Curie-Weiss model. The pressures and CLTs for the 2-regular configuration model are considered in Section \ref{three} and for the configuration model with vertex degrees 1 and 2 in Section \ref{four}. 
In the former case, we show that the variance of the limiting normal variable is the susceptibility of the one-dimensional
 Ising model.  In the latter case, which is much more difficult, the varying degrees of the vertices affect the pressure and the limiting distribution. In fact, the limiting variance is the sum of that of the one-dimensional Ising model and of an extra term emerging from the fluctuations of the connected structures of the graph. 

\section{Proofs for $\mathrm{GRG}_{\sss N}({\bf w})$}{GRG}\label{sec_annealedGRG}\label{two}
In this section, we derive our results for the generalized random graph $\GRGw$ stated in Theorems \ref{term_lim_annealed}, \ref{slln_ann} and \ref{ann_CLT}.
%$\mathrm{GRG}_{\sss N}({\mathbold w})$ 
\subsection{Annealed thermodynamic limits: Proof of Theorem \ref{term_lim_annealed}}

%\begin{itemize}
%\item[(i)]  We want to compute the annealed pressure in the thermodynamic limit:
%\be
%\widetilde{\psi} (\beta, B)  \,:=\, \lim_{N \rightarrow\infty} \, \widetilde{\psi}_{\sss N} (\beta, B)  \,=\, \lim_{N \rightarrow\infty} \, \frac{1}{N} \log \left(\q \left(Z_{\sss N} \left( \beta, B \right)\right)\right).
%\ee
%for the Generalized Random Graphs. \\
The proof is divided into several steps. \\

\paragraph{Annealed partition function.} 
We start by analyzing the average of the partition function for $\GRGw$. {By remembering that in this random graph the edges are independent and  denoting by  $I_{ij}$  the Bernoulli indicator that the edge between vertex $i$ and vertex $j$ is present}, we compute
  \begin{align*}\nonumber
	\q \left(Z_{\sss N} \left( \beta, B \right)\right) \, & = \, \q\Big( \sum_{\sigma \in \Omega_{\sss N}} \exp \Big[ \beta \sum_{i<j}{I_{ij}\sigma_i \sigma_j} + B \sum_{i \in[N]}  {\sigma_i}  \Big] \Big)  \\ \nonumber
	&  = \,  \sum_{\sigma \in \Omega_{\sss N}} \e^{B \sum_{i \in[N]} {\sigma_i}} 
	\q\left( \e^{\beta \sum_{i<j}{I_{ij}\sigma_i \sigma_j}} \right) \\ \nonumber
	&  = \,  \sum_{\sigma \in \Omega_{\sss N}} \e^{B \sum_{i \in[N]} {\sigma_i}}
	\prod_{i<j} \q \big(\e^{\beta I_{ij}\sigma_i \sigma_j}\big)\\ 
	&  = \,  \sum_{\sigma \in \Omega_{\sss N}}\e^{B \sum_{i \in[N]} {\sigma_i}} \prod_{i<j} \left(\e^{\beta\sigma_i \sigma_j}p_{ij} + \left(1 - p_{ij}\right)\right).
	\end{align*}
We rewrite
	\begin{equation*} 
	\e^{\beta\sigma_i \sigma_j}p_{ij} + \left(1 - p_{ij}\right) \, = \, C_{ij}\e^{\beta_{ij}\sigma_i \sigma_j},
	\end{equation*} 
where $\beta_{ij}$ and $C_{ij}$ are  chosen such that
%where $C_{ij}$ is a deterministic factor independent of $\sigma$ which is chosen such that
	%\begin{equation} 
	%\e^{\beta\sigma_i \sigma_j}p_{ij} + \left(1 - p_{ij}\right)\,
	%=\, \left\{\begin{array}{ll} \e^{-\beta}p_{ij} + \left(1 - p_{ij}\right) \quad \quad \quad \; \mbox{if} \;\sigma_i \sigma_j = -1 ,\\ 
	%\e^{\beta}p_{ij} + \left(1 - p_{ij}\right) \quad \quad \quad \; \; \, \, \mbox{if} \; \sigma_i \sigma_j = +1. \end{array}\right.  
	%\end{equation}
%and
	\begin{equation*} 
	\e^{-\beta}p_{ij} + \left(1 - p_{ij}\right) \, =\, C_{ij}\e^{-\beta_{ij}},%\qquad\mbox{if} \;\sigma_i \sigma_j = -1 ,
	\qquad
	\text{and}
	\qquad \e^{\beta}p_{ij} + \left(1 - p_{ij}\right) 
	\, =\, C_{ij}\e^{\beta_{ij}}.%\qquad\mbox{if} \; \sigma_i \sigma_j = +1.
	\end{equation*}
Now, by adding and dividing the two equations of the system above, we get
	\begin{equation*} 
	C_{ij} \cosh\left(\beta_{ij}\right) \, = \, p_{ij} \cosh(\beta) + \left(1 - p_{ij}\right),
	\qquad\quad
	\beta_{ij}\, = \, \frac{1}{2} \log \frac{\e^{\beta}p_{ij} + \left(1 - p_{ij}\right)}{\e^{-\beta}p_{ij} + \left(1 - p_{ij}\right)}.
	\end{equation*}
Then, using the symmetry $\beta_{ij}=\beta_{ji}$ we arrive at
	\begin{align} \nonumber
	\q \left(Z_{\sss N} \left( \beta, B \right)\right) \, 
	& = \, \Big(\prod_{i < j} C_{ij}\Big)\sum_{\sigma \in \Omega_{\sss N}}
	\e^{B \sum_{i \in[N]} {\sigma_i}}\e^{\sum_{i<j}{\beta_{ij}\sigma_i \sigma_j}} \\
	& = \, G(\beta)G_1(\beta)\sum_{\sigma \in \Omega_{\sss N}}
	\e^{B \sum_{i \in[N]} {\sigma_i}}\e^{\frac{1}{2}\sum_{i,j \in[N]}{\beta_{ij}\sigma_i \sigma_j}},
	\label{explicit-parti-func-ann}
	\end{align}
where $G(\beta) = \prod_{i < j} C_{ij}$ and $G_1(\beta) = \prod_{i\in[N]} \e^{-\beta_{ii}/2}$ and we write $p_{ii}=w_i^2/(\ell_{\sss N}+w_i^2).$ This is the starting point of our analysis. We can recognize the r.h.s.\ as the partition function of an inhomogeneous Ising model on the complete graph, where the coupling constant between vertices $i$ and $j$ is equal to $\beta_{ij}.$ In the next step, we analyze this partition function in detail.\\

\paragraph{Towards an inhomogeneous Curie-Weiss model.}
We continue by showing that $\beta_{ij}$ is close to factorizing into a contribution due to $i$ and to $j$. For this, by a Taylor expansion of $x\mapsto \log(1+x)$,
	\begin{align*} \nonumber
	\beta_{ij}\, & = \, \frac{1}{2}\log\left(1 + p_{ij}(\e^{\beta} -1)\right) 
	- \frac{1}{2}\log\left(1 + p_{ij}(\e^{-\beta} -1)\right) \\
	& = \,\frac{1}{2}p_{ij}(\e^{\beta} -1) - \frac{1}{2}p_{ij}(\e^{-\beta} -1) + O(p_{ij}^2)
	= \, \sinh(\beta)p_{ij}+ O(p_{ij}^2).
	\end{align*}
Then,
	\eqan{
	&\q \left(Z_{\sss N}(\beta, B)\right) \,\\
	&\qquad= \, G_2(\beta)\sum_{\sigma \in \Omega_{\sss N}}\e^{B \sum_{i \in[N]} {\sigma_i}}
	\e^{\frac{1}{2}\sinh\left(\beta\right)\sum_{i,j \in[N]}p_{ij}\sigma_i \sigma_j 
	+ O\left(\sum_{i,j \in[N]}p_{ij}^2\sigma_i \sigma_j\right)}.\nn
	}
where $G_{2}(\beta)=G(\beta)G_{1}(\beta)$. To control the error in the exponent, we use $p_{ij}\leq w_iw_j/\ell_{\sss N}$ and 
the assumptions in Condition \ref{cond-WR-GRG}, to obtain
%
%For the generalized random graph, with $\ell_{\sss N}=\sum_{i\in[n]}w_i$,
%	\be
%	\Big|p_{ij}- \frac{w_i w_j}{\ell_{\sss N}}\Big|= \frac{w_i w_j}{\ell_{\sss N}}\Big|\frac{1}{1+\frac{w_i w_j}{\ell_{\sss N}}}-1\Big|
%	=O \Big(\frac{w_i w_j}{\ell_{\sss N}}\Big)^2,
%	\ee
%so that, also using that $p_{ij}\leq w_iw_j/\ell_{\sss N}$, and as long as $\sum_{i\in[n]}w_i^2=o(N^{3/2})$,
	\begin{equation*}
	\Big|\sum_{i,j \in[N]}p_{ij}^2\sigma_i \sigma_j\Big|\leq  \sum_{i,j \in[N]}\Big(\frac{w_i w_j}{\ell_{\sss N}}\Big)^2
	\, = \, \left(\frac{\sum_{i\in[N]} w_i^2}{\ell_{\sss N}}\right)^2 \, = \, o(N).
	\end{equation*}
%%We must have $ \nu = \frac{\mathbb{E}\left[W^2\right]}{\mathbb{E}\left[W\right]} \in \left(1, \infty\right)$.\\ Because we know that $\displaystyle w_i \sim \left(\frac{N}{i}\right)^{\frac{1}{\tau-1}}$ we have
Then,
	\begin{align*} 
	\q \left(Z_{\sss N}( \beta, B)\right) \, 
	&= \,G_2(\beta) \e^{o(N)}
	\sum_{\sigma \in \Omega_{\sss N}}\e^{B \sum_{i \in[N]} {\sigma_i}}\e^{\frac{1}{2}\sinh(\beta)
	\sum_{i,j \in[N]}\frac{w_i w_j}{\ell_{\sss N}}\sigma_i \sigma_j}\\
	&= \, G_2(\beta) \e^{o(N)}\sum_{\sigma \in \Omega_{\sss N}}
	\e^{B \sum_{i \in[N]} {\sigma_i}}\e^{\frac{1}{2}\frac{\sinh(\beta)}{\ell_{\sss N}}\left(\sum_{i\in[N]}w_i \sigma_i\right)^2}.
	\end{align*}
When $w_i\equiv w$ for all $i$, so that $\GRGw$ is the Erd\H{o}s-R\'enyi random graph, we retrieve the Curie-Weiss model at inverse temperature $\beta' = \sinh(\beta) w$.
%exactly.
In our inhomogeneous setting, we obtain an inhomogeneous Curie-Weiss model that we will analyze next. \\

\paragraph{Analysis of the inhomogeneous Curie-Weiss model.}
We use the Hubbard-Stratonovich identity, i.e., we write $\e^{t^2/2} = \mathbb{E}[\e^{tZ}]$, with $Z$ standard Gaussian. Then, we find
	\begin{align*}
	\q \left(Z_{\sss N}(\beta, B)\right) \, 
	&= \,G_{2}(\beta) \e^{o(N)}\sum_{\sigma \in \Omega_{\sss N}}
	\e^{B \sum_{i \in[N]} {\sigma_i}}
	\mathbb{E} \Big[\e^{\sqrt{\frac{\sinh\left(\beta\right)}{\ell_{\sss N}}}
	\left(\sum_{i\in[N]}w_i \sigma_i\right) Z}\Big] \\
	&= \,G_{2}(\beta) \e^{o(N)} 2^N 
	\mathbb{E} \Big[\prod_{i=1}^N \cosh \Big(\sqrt{\frac{\sinh\left(\beta\right)}{\ell_{\sss N}}}w_i Z + B\Big)\Big]\\
	&= \,G_{2}(\beta) \e^{o(N)} 2^N 
	\mathbb{E} \Big[\exp\Big\{\sum_{i=1}^N \log \cosh \Big(\sqrt{\frac{\sinh(\beta)}{\ell_{\sss N}}}w_i Z + B\Big)\Big\}\Big].
	\end{align*}
We rewrite the sum in the exponential, using the fact that $W_{\sss N}=w_{I_{\sss N}}$, where we recall that $I_{\sss N}$ is a uniform vertex in $[N]$, to obtain
	\begin{align*}
	\q \left(Z_{\sss N}(\beta, B)\right)
	&=G_{2}(\beta) \e^{o(N)} 2^N 
	\mathbb{E} \Big[\hspace{-0.05cm}\exp{\Big\{\hspace{-0.05cm}N \mathbb{E} \Big[\hspace{-0.05cm}\log \cosh\hspace{-0.08cm}\Big(\sqrt{\frac{\sinh(\beta)}{N \mathbb{E}[W_{\sss N}]}}W_{\sss N} Z + B\Big)
	\Big| Z \Big]\hspace{-0.05cm}\Big\}}\hspace{-0.05cm}\Big]\\ 
	&=G_{2}(\beta) \e^{o(N)} 2^N 
	\mathbb{E} \Big[\e^{N F_{\sss N}\left(\frac{Z}{\sqrt{N}}\right)}\Big],
	\end{align*}
where
	\be\label{FN-def}
	F_{\sss N}(z) \, = \,  \mathbb{E}\Big[\log \cosh 
	\Big(\sqrt{\frac{\sinh\left(\beta\right)}{\mathbb{E}\left[W_{\sss N}\right]}}W_{\sss N} z + B\Big)\Big].
	\ee
Here we emphasize the fact that in \eqref{FN-def}, the expectation is w.r.t.\ $W_{\sss N}$ only.\\We continue by analyzing $F_{\sss N}(z)$. We claim that, uniformly for $|z|\leq a$ and any $a<\infty$,
	\be\label{equicont-AA}
	\sup_{|z|\leq a} |F_{\sss N}(z)-F(z)|=o(1),
	\ee
where 
	\begin{equation*}
	F(z) \, = \,  \mathbb{E}\Big[\log \cosh 
	\Big(\sqrt{\frac{\sinh\left(\beta\right)}{\mathbb{E}[W]}}W z + B\Big)\Big].
	\end{equation*}
To see \eqref{equicont-AA}, we note that $F_{\sss N}(z)\rightarrow F(z)$ for every $z$ fixed by Condition \ref{cond-WR-GRG}(a)-(b), and the fact that $\log\cosh(x)\leq |x|$. Further,
	\begin{equation*}
	|F_{\sss N}'(z)|\leq 
	\frac{\sinh(\beta)}{\mathbb{E}[W_{\sss N}]}\mathbb{E}\Big[\tanh 
	\Big(\sqrt{\frac{\sinh\left(\beta\right)}{\mathbb{E}[W_{\sss N}]}}W_{\sss N} z + B\Big)W_{\sss N}\Big]
	\leq \sinh(\beta),
	\end{equation*}
since $\tanh(x)\leq 1$ for all $x$, so that $|F_{\sss N}'(z)|$ is uniformly bounded in $N$ and $z$. Therefore, $(F_{\sss N})_{N\geq 1}$ forms a uniformly equicontinuous family of functions, so that \eqref{equicont-AA} follows from Arzel\`a-Ascoli. Since $F_{\sss N}(z)\leq \sinh(\beta)|z|$, it further follows that, for $a > 4 \sinh(\beta)$,
	\begin{align*}
	\mathbb{E} \left[\e^{N F_{\sss N}\left(\frac{Z}{\sqrt{N}}\right)}\indic{|Z|>a\sqrt{N}}\right]
	&\leq \mathbb{E} \left[\e^{\sqrt{N}\sinh(\beta)|Z|}\indic{|Z|>a\sqrt{N}}\right]\\
	&=2\mathbb{E} \left[\e^{\sqrt{N}\sinh(\beta)Z}\indic{Z>a\sqrt{N}}\right]\\
	&=\frac{2}{\sqrt{2\pi}}\int_{a\sqrt{N}}^{\infty} \e^{\sqrt{N}\sinh(\beta)z}\e^{-z^2/2}dz\\
	&\leq \e^{a\sinh(\beta)N-a^2N/2} \int_0^{\infty} \e^{\sqrt{N}(\sinh(\beta)-a)x}dx \,\leq \,\e^{-a^2N/4},\nonumber
	\end{align*}
which, for $a$ sufficiently large, is negligible compared to $\mathbb{E} \left[\e^{N F_{\sss N}\left(\frac{Z}{\sqrt{N}}\right)}\indic{|Z|\leq a\sqrt{N}}\right]$. We conclude that
	\begin{align}
		\label{LD-form}
	\q \left(Z_{\sss N} \left( \beta, B \right)\right) \, 
	&= \,G_{2}(\beta) \e^{o(N)} 2^N 
	\mathbb{E} \Big[\e^{N F\left(\frac{Z}{\sqrt{N}}\right)}\Big](1+o(1)).
	\end{align}

\paragraph{A large deviation analysis.}
The expectation in \eqref{LD-form} is an expectation of an exponential functional, to which we apply large deviation machinery.
The Gaussian variable  $Z/\sqrt{N}$ satisfies a large deviation principle with rate function $I(z)=z^2/2$ and speed $N$, because $Z/\sqrt{N} \stackrel{d}{=} \frac{1}{N}\left(Z_1+...+Z_{\sss N}\right)$, where $(Z_i)_{i\in[N]}$ are i.i.d.\ standard Gaussian variables. 
%Let's take $x>0$ and prove that the rate function $I(x) = \frac{x^2}{2}$:
%\be
%\q \left(\frac{Z}{\sqrt{N}} > x\right) = \q\left(Z > x \sqrt{N}\right) = 1 - \Phi\left(x \sqrt{N}\right) \sim \e^{- \frac{\left(x \sqrt{N}\right)^2}{2}} = \e^{- \frac{x^2 N }{2}}
%\ee
%then, in particular
%\be
%- \frac{1}{N} \log \q \left(\frac{Z}{\sqrt{N}} > x\right) \,  \longrightarrow \,  \frac{x^2}{2}.
%\ee 
Using Varadhan's Lemma and the fact that $z\mapsto F(z)$ is continuous, we calculate the thermodynamic limit of the pressure as
	\begin{align} 
	\label{limannealedp}\nn
	\lim_{N \rightarrow \infty}\, 
	\frac{1}{N} \log \q \left(Z_{\sss N} \left( \beta, B \right)\right)\, 
	&= \, \log 2 + \lim_{N \rightarrow \infty}\frac{1}{N} \log G_2 (\beta) + \sup_{z}\left[F(z) - I(z)\right] \\ 
	&= \, \log 2 + \alpha \left(\beta\right) \\
	&\, \; + \sup_{z}\Big[\mathbb{E}\Big[\log \cosh \Big(\sqrt{\frac{\sinh(\beta)}{\mathbb{E}\left[W\right]}}W z + B\Big)\Big] - \frac{z^2}{2}\Big].\nn
	\end{align}
where $\alpha \left(\beta\right) = \lim_{N \rightarrow \infty}\frac{1}{N} \log G_2 \left(\beta\right)$. The equation that defines the 
%maximum 
supremum is
	\be\label{fixpGRG}
	z^*=z^*(\beta,B) \, = \, 
	\mathbb{E}\Big[\tanh \Big(\sqrt{\frac{\sinh\left(\beta\right)}{\mathbb{E}[W]}}W z^* + B\Big) 
	\sqrt{\frac{\sinh\left(\beta\right)}{\mathbb{E}[W]}}\,W\Big],
	\ee
and the annealed pressure is obtained by substituting the 
%maximum 
supremum point  $z^{*}$ in the right hand side of  (\ref{limannealedp})  as
	\be
	\label{annealedpGRG}
	\tilde{\psi}(\beta,B)=  \, \log 2 + \alpha(\beta) 
	+ \mathbb{E}\Big[\log \cosh \Big(\sqrt{\frac{\sinh\left(\beta\right)}{\mathbb{E}[W]}}W z^{*}(\beta,B) + B\Big)\Big] - z^{*}(\beta,B)^2/2.
	\ee
This completes the proof of Theorem \ref{term_lim_annealed}(i).\\

\paragraph{The critical inverse temperature.}
To identify $\beta_c^{\mathrm{an}}$ as stated in Theorem \ref{term_lim_annealed}(ii), we evaluate \eqref{fixpGRG} when $B \searrow 0$ to obtain
	\be
	\label{eq_max}
	z^* \, = H(z^*)\, \qquad \text{where} 
	\qquad H(z)=\mathbb{E}\left[\tanh \left(\sqrt{\frac{\sinh\left(\beta\right)}{\mathbb{E}\left[W\right]}}W z\right)
		\sqrt{\frac{\sinh\left(\beta\right)}{\mathbb{E}\left[W\right]}}\,W\right].
	\ee
%We call
%\be
%H(z)\, = \,  \sqrt{\frac{\sinh\left(\beta\right)}{\mathbb{E}\left[W\right]}} \, \mathbb{E}\left[\tanh \left(\sqrt{\frac{\sinh\left(\beta\right)}{\mathbb{E}\left[W\right]}}W z\right)\,W\right]
%\ee
We investigate the solutions of $z^* \, = H(z^*)$ in \eqref{eq_max}. 
%to determine the solutions. 
We note that $z\mapsto H(z)$ is an increasing and concave function in $[0, \infty)$. When $H'(0) > 1$, we have three solutions of \eqref{eq_max}, i.e., $\pm z^*$ and 0, where $z^*=z^*(\beta,0^+)>0$. When $H'(0) \leq 1$, instead, $z^*=0$ is the only solution. This leads us to compute that
	\begin{equation*}
	H'(0) %\, = \, \frac{\sinh\left(\beta\right)}{\mathbb{E}\left[W\right]}\mathbb{E}\left[W^2\right] 
	\, = \, \sinh\left(\beta\right) 	\frac{\mathbb{E}\left[W^2\right]}{\mathbb{E}\left[W\right]} \, = \, \sinh\left(\beta\right) \nu.
	\end{equation*}
Thus, the annealed critical temperature $\beta_c^{\mathrm{an}}$ satisfies $\sinh\left(\beta_c^{\mathrm{an}}\right) \, = \, 1/\nu$. Since $\tanh\left(\beta_c^{\mathrm{qu}}\right)=1/\nu$, and $\tanh(x) < \sinh(x)$  $\, \forall \, x > 0$,
we obtain $\beta_c^{\mathrm{qu}} \, > \, \beta_c^{\mathrm{an}}$, unless when $\nu=\infty$, in which case $\beta_c^{\mathrm{an}}=\beta_c^{\mathrm{qu}}=0$.\\

\paragraph{Thermodynamic limit of the magnetization.}
To prove the existence of the magnetization in the thermodynamic limit stated in Theorem \ref{term_lim_annealed}(ii), we follow the strategy used in \cite{DGH}. We use the following lemma:

\begin{lemma}\label{limit_deriv}
Let $(f_n)_{n\geq1}$ be a sequence of functions that are twice differentiable in $x$. Assume that
\begin{itemize}
\item[(a)] $\lim_{n\rightarrow \infty} f_n(x) = f(x)$ for some function $y \mapsto f(y)$ that is differentiable in $x$;
\item[(b)] $\frac{d}{dx}f_n(x)$ is monotone in $\left[x-h, x+h\right]$ for all $n\geq 1$ and some $h>0$.
\end{itemize}
Then,
	\begin{equation*}
	\lim_{n\rightarrow \infty}\frac{d}{dx} f_n(x) = \frac{d}{dx} f(x).
	\end{equation*}
\end{lemma}
\noindent
We apply Lemma \ref{limit_deriv} with $n=N$ and $f_n$ equal to $B \mapsto  \widetilde{\psi}_{\sss N} (\beta, B)$. We verify the conditions in Lemma \ref{limit_deriv} and start by noting that
	\begin{equation*}
	\widetilde{M}_{\sss N} (\beta, B) \,=\, \widetildep\big(S_{\sss N}/N\big)  
	\,=\, \frac{\partial}{\partial B}\widetilde{\psi}_{\sss N} (\beta, B),
	\end{equation*}
and $\lim_{N\rightarrow \infty}\widetilde{\psi}_{\sss N} (\beta, B) = \widetilde{\psi} (\beta, B)$ by Theorem \ref{term_lim_annealed}(i) with $B \mapsto  \widetilde{M}_{\sss N} (\beta, B)$ 
non-decreasing:
%. Thus, we can indeed conclude that
	\begin{equation*}
	\frac{\partial}{\partial B}\widetilde{M}_{\sss N} (\beta, B) 
	\,=\, \frac{1}{N} \left[ \widetildep\left(S_{\sss N}^2 \right)- \widetildep\left(S_{\sss N}\right)^2\right] \, \geq \,0. 
	%\,=\, \frac{1}{N} 	\sum_{i,j \in [N]} \left[\widetildep\left(\sigma_i \sigma_j\right) -
	%\widetildep\left(\sigma_i\right)\widetildep\left(\sigma_j\right) \right] \, \geq 0
	\end{equation*}
Thus, we can indeed conclude that
%Therefore,
	\begin{equation*}
	\widetilde{M}(\beta,B)
	=\lim_{N\rightarrow \infty}\widetilde{M}_{\sss N} (\beta, B)
	\,=\,\lim_{N\rightarrow \infty} \frac{\partial}{\partial B}\widetilde{\psi}_{\sss N} (\beta, B)
	\,=\,\frac{\partial}{\partial B}\widetilde{\psi}(\beta, B).
	\end{equation*}
The limit magnetization $\widetilde{M}(\beta,B)$ can be explicitly computed by taking the derivative of $\widetilde{\psi}(\beta, B),$ (\ref{annealedpGRG}) and using the fixed point equation (\ref{fixpGRG}), to obtain
	\begin{equation*}
	\widetilde{M}(\beta,B)\, 
	= \, \mathbb{E}\left[\tanh \left(\sqrt{\frac{\sinh\left(\beta\right)}{\mathbb{E}\left[W\right]}}W z^* + B\right) \right].
	\end{equation*}

\paragraph{Thermodynamic limit of the susceptibility.}
Finally, the thermodynamic limit of the susceptibility in Theorem \ref{term_lim_annealed}(iv) is proved using Lemma \ref{limit_deriv} by combining Theorem \ref{term_lim_annealed}(ii) and the fact that $B \mapsto  \frac{\partial}{\partial B}\widetilde{M}_{\sss N} (\beta, B)$ is non-increasing by the GHS inequality. Indeed, by the explicit computation in \eqref{explicit-parti-func-ann}, we see that the annealed partition function can be viewed as the partition function of an inhomogeneous Curie-Weiss model, where the field is  homogeneous and the coupling constants depend on the edges. Since such an inhomogeneous Ising model also satisfies the GHS inequality, the same follows for the annealed partition function for $\GRGw$. Therefore,
	\begin{equation*}
	\frac{\partial^2}{\partial B^2}\widetilde{M}_{\sss N} (\beta, B) \,=\, \frac{1}{N}\sum_{i \in [N]} \frac{\partial^2}{\partial B^2} 	\widetildep\left(\sigma_i\right) \,\leq \,0.
	\end{equation*}
\qed

\subsection{Annealed SLLN and CLT: Proofs of Theorems \ref{slln_ann} and \ref{ann_CLT}}
\noindent
With Theorem \ref{term_lim_annealed} in hand, we now have all the hypotheses to prove Theorems \ref{slln_ann} and \ref{ann_CLT} following the strategy used for the random quenched setting in Sections 2.2 and 2.3 of \cite{GGPvdH1} verbatim. Indeed, for the proof of the annealed SLLN, referring to \cite[Section 2.2]{GGPvdH1}, we obtain the existence of the thermodynamic limit of the annealed cumulant generating function
	\begin{equation*}
	\widetilde{c}_{\sss N}(t) \, = \, \frac{1}{N} \log \widetildep \left[\exp\left(t S_{\sss N}\right)\right] 
	\, = \, \widetilde{\psi}_{\sss N} (\beta, B+t) - \widetilde{\psi}_{\sss N} (\beta, B)
	\end{equation*}
by Theorem \ref{term_lim_annealed}(i). Then, from \cite[Theorem II.6.3]{El} %\cite[Theorem  2.1]{GGPvdH1} 
and Theorem \ref{term_lim_annealed}(ii) we conclude the proof.

To prove the annealed CLT (see \cite[Section 2.3]{GGPvdH1} for the proof of the random quenched CLT) we need the existence in the thermodynamic limit of pressure, magnetization and susceptibility given by Theorem \ref{term_lim_annealed} together with the GHS inequality that is still true in the annealed setting
thanks to the mapping to the inhomogeneous Curie-Weiss model.

\section{Proofs for $\CMNtwo$}
\label{three}

In this section we prove the CLT with respect to the annealed measure for the 2-regular random graph. We start by computing the annealed pressure using the partition functions for the one-dimensional Ising model with periodic boundary conditions.%  We will use them  in the proofs of CLT's because the structures  formed in  $\CMNtwo$ and $\CMNonetwo$, are lines (indicated with the letter $l$) and cycles (or tori, indicated with the letter $t$).

\subsection{Annealed thermodynamic limits and SLLN: Proofs of Theorems \ref{term_lim_ann_CM2} and \ref{slln_ann-CM2}}
From our previous paper \cite{GGPvdH1}, we remember that  any 2-regular random graph is formed by cycles only. Thus, as in \cite{GGPvdH1}, denoting the  random number of cycles in the graph by  $K^t_{\sss N}$, we can enumerate them in an arbitrary order from 1 to $K^t_{\sss N}$ and call $L_{\sss N}(i)$ the length (i.e., the number of vertices) of the $i$th cycle.  The random variable $K^t_{\sss N}$  has distribution given  by 
	\be\label{kappa_N}
	K^t_{\sss N}= \sum_{j=1}^N I_j,
	\ee
where $I_j$ are independent Bernoulli variables given by
	\be\label{bernI}
	I_j = \mbox{Bern}\,\left(\frac{1}{2N -2j+ 1}\right).
	\ee
See \cite{GGPvdH1} for a proof of this fact.
%Indeed,  at every step $j$ in the construction of the Random 2-regular graph, namely at each coupling of two half-edge, we have one and only one possibility to close a cycle. This chance corresponds to taking (uniformly)  exactly that half-edge (out of  the remaining $2N-2j+1$ uncoupled half-edges) which allows to close the cycle. The event  is indicated by the variable  $I_j$. Thus $\left(I_j\right)_{j=1}^N$ are independent, but not identical and  the number of the cycle $K_{\sss N}$ is distributed as their sum.
%Indeed,  at every step $j$ in the construction of the 2-regular random graph, namely at each pairing of two half-edges, we have one and only one possibility to close a cycle. This chance corresponds to taking (uniformly)  exactly one half-edge out of  the remaining $2N-2j+1$ uncoupled half-edges. The indicator of this  event  is the Bernoulli the variable  $I_j$. Obviously, the variables $\left(I_j\right)_{j=1}^N$ are independent, but not identical, and  the number of the cycles $K^t_{\sss N}$ is distributed as their sum.\\
Since the random  graph splits into (disjoint) cycles, its quenched partition function factorizes into the product of the partition functions of each cycle. Therefore,
	\be\label{zprod}
	Z_{\sss N}(\beta, B) = \prod_{i=1}^{K^t_{\sss N}} Z_{L_{\sss N}(i)}^{\sss (t)} (\beta, B).
	\ee
By \cite[Section 3.1]{GGPvdH1}, we have that the partition function of the one-dimensional Ising model with periodic boundary conditions $Z^{\sss (t)}_{\sss N}$ is given by
\begin{equation}
\label{tre}
Z^{\sss (t)}_{\sss N}(\beta, B) \, = \, \lambda_+^N(\beta, B) + \lambda_-^N(\beta, B),
\end{equation}
where
\be\label{lambda_piu_meno}
\lambda_{\pm}(\beta, B) = \e^{\beta } \left[ \cosh(B)\pm \sqrt{\sinh^2(B)+\e^{-4\beta }}\right]\;,
\ee
so we can write
\begin{equation*}
Z_{L_{\sss N}(i)}^{\sss (t)}(\beta, B) = \lambda_+^{L_{\sss N}(i)}(\beta, B) + \lambda_-^{L_{\sss N}(i)}(\beta, B).
\end{equation*}
Because $\beta > 0$, we have $0 < \lambda_-(\beta, B) < \lambda_+(\beta, B)$, so that, for every $i$,
\begin{equation*}
\lambda_+^{L_{\sss N}(i)}(\beta, B) \leq Z_{L_{\sss N}(i)}^{\sss (t)}(\beta, B) \leq 2 \lambda_+^{L_{\sss N}(i)}(\beta, B).
\end{equation*}
As a result, we can bound the the pressure as follows:
\begin{equation*}
\prod_{i=1}^{K^t_{\sss N}} \lambda_+^{L_{\sss N}(i)}(\beta, B) \leq \prod_{i=1}^{K^t_{\sss N}} Z_{L_{\sss N}(i)}^{\sss (t)}(\beta, B) \leq \prod_{i=1}^{K^t_{\sss N}} 2 \lambda_+^{L_{\sss N}(i)}(\beta, B),
\end{equation*}
and, since $\sum_{i=1}^{K^t_{\sss N}}L_{\sss N}(i) = N$, we finally obtain
\be\label{Bound}
\lambda_+^N(\beta, B) \leq Z_{\sss N}(\beta, B) \leq 2^{K^t_{\sss N}} \lambda_+^N(\beta, B).
\ee

\noindent
The  thermodynamic limit  of the annealed pressure $\widetilde{\psi}_{\sss N} (\beta, B)$,  defined in \eqref{def-annealed-press}, can be computed along the same lines of the averaged quenched one in \cite{GGPvdH1}. Indeed, 
by applying the monotone operator $N^{-1}\log (Q_{\sss N}(\cdot))$ to \eqref{Bound} and using the fact that $\lambda_{+}(\beta, B)$ is non random, we obtain
	\begin{equation*}
	\log \lambda_+(\beta, B) \leq \widetilde{\psi}_{\sss N} (\beta, B) 
	\leq \frac{1}{N} \log \left(\q \left(2^{K^t_{\sss N}}\right)\right) + \log \lambda_+(\beta, B).
	\end{equation*}
Now using the fact 
	\begin{align}
	\label{dueallaK}\nonumber
	\frac{1}{N} \log \left(\q \left(2^{K^t_{\sss N}}\right)\right) 
	&= \frac{1}{N} \log \prod_{i=1}^N \q \left(2^{I_i}\right)\\ \nn
	&= \, \frac{1}{N} \log 
	\prod_{i=1}^N \left[ \frac{2}{2N -2i +1} + \left(1 - \frac{1}{2N -2i +1} \right)\right] & \\ 
	& = \, \frac{1}{N} \sum_{i=1}^N \log \left(1 + \frac{1}{2N -2i +1} \right) 
	\stackrel{N \rightarrow \infty}{\longrightarrow} 0,
	\end{align}
we conclude that the annealed pressure of $\CMNtwo$ coincides with the pressure of the one-dimensional Ising model $\psi^{d=1} (\beta, B)$, i.e.,
	\be\label{espr_ann_press_CM2}
	\widetilde{\psi} (\beta, B)\,=\, \psi^{d=1} (\beta, B) \, \equiv \,\log \lambda_+(\beta, B).
	\ee
Moreover, it also agrees with the averaged and random quenched pressures  \cite{GGPvdH1}, i.e.,
	\begin{equation*}
	\widetilde{\psi} (\beta, B) =\overline{\psi} (\beta, B) = {\psi} (\beta, B),
	\end{equation*}
where
	$$
	\overline{\psi} (\beta, B) := \lim_{N\to\infty} \frac{1}{N} \q (\ln Z_N(\beta,B))
	\qquad
	\text{and}
	\qquad
	{\psi} (\beta, B) := \lim_{N\to\infty} \frac{1}{N}\ln Z_N(\beta,B)\;.
	$$
It also straightforwardly follows that the annealed cumulant generating function of $\CMNtwo$ coincides with the random and averaged quenched ones  \cite{GGPvdH1} i.e.,
	\be\label{tre_c}
	\widetilde{c}(t)=\overline{c}(t) =c(t)= \log \lambda_+(\beta, B+t) - \log \lambda_+(\beta, B).
	\ee
The existence of the magnetization in the thermodynamic limit (Theorem  \ref{term_lim_ann_CM2}(ii)) can be proved, as in the previous section, using Lemma \ref{limit_deriv} and the existence of the thermodynamic limit of the pressure (\ref{espr_ann_press_CM2}), so we obtain
	\begin{equation*}
	\widetilde{M}(\beta,B) \,=\, \frac{\partial}{\partial B}\widetilde{\psi}(\beta, B) 
	\,=\, \frac{\sinh(B)}{\sqrt{\sinh^2(B) + \e^{-4\beta}}},
	\end{equation*}
as required.\qed
%Taking $\lim_{B \rightarrow 0^+} \widetilde{M}(\beta,B)$ and using the definition \ref{beta_c_qu} we observe that there is not a phase transition for this model.
\medskip

\noindent
{\it Proof of Theorem \ref{slln_ann-CM2}:} The proof, as for Theorem \ref{slln_ann}, follows immediately from the existence of the annealed pressure in the thermodynamic limit and its differentiability with respect to $B$. See also \cite[Section 2.2]{GGPvdH1}.
\qed

\subsection{Annealed CLT: Proof of Theorem \ref{CLT_annealed1}}
To prove the CLT in the annealed setting, we follow the strategy used in  \cite{GGPvdH1}  for the averaged quenched CLT. \\
%Hence we need to show that
%\be
%\lim_{N \rightarrow \infty} \widetildep \left[\exp\left(t \, \frac{S_{\sss N} - \widetildep(S_{\sss N})}{\sqrt{N}}\right) \right] \, = \, \exp \left(\frac{1}{2} \widetilde{\chi}(\beta, B) t^2\right)  \qquad \mbox{for all}\; t \in \mathbb{R}.
%\ee

\paragraph{Rewrite in terms of cumulant generating functions.}
Using the annealed cumulant generating function and using a Taylor expansion, we write
	\begin{align}\label{cumgfcm2cder2}
		\log \widetildep  \left[\exp\left( \frac{t S_{\sss N} - t \widetildep(S_{\sss N})}{\sqrt{N}}\right) \right] 
		\, = \: \frac{t^2}{2} \widetilde{c}''_{\sss N}(t_{\sss N}),
	\end{align}
%\begin{align}\nonumber
%\log \widetildep  \left[\exp\left( \frac{t S_{\sss N} - t \widetildep(S_{\sss N})}{\sqrt{N}}\right) \right] \, &= \, \log \left[\widetildep\left(\exp \left(\frac{t}{\sqrt{N}} S_{\sss N}\right)\right)\right] - \frac{t}{\sqrt{N}} \widetildep(S_{\sss N}) \\
%& = \, N \widetilde{c}_{\sss N} \left(\frac{t}{\sqrt{N}}\right) - t \sqrt{N} \widetilde{c}'_{\sss N} (0) \: = \: \frac{t^2}{2} \widetilde{c}''_{\sss N}(t_{\sss N}),
%\end{align}
where $t_{\sss N} \in [0 , t/\sqrt{N}]$. Then the aim is to prove that  $\lim_{N \rightarrow \infty} \, \widetilde{c}''_{\sss N}(t_{\sss N})$ exists as a finite limit. 

By expressing $\widetilde{c}_{\sss N}(t)$ in terms of $Z^{\sss (t)}_{N}=\lambda_{+}^{N}+\lambda_{-}^{N}$ and using \eqref{tre_c}, we can compute the difference as
	\begin{equation*}
		\widetilde{c}_{\sss N}(t) - \widetilde{c}(t) = \frac{1}{N} \log \hspace{-0.05cm} \left[ \frac{\q \hspace{-0.05cm}\left(\frac{Z^{\sss (t)}_{\sss N}(\beta, B+t)}
		{\left(\lambda_+(\beta, B+t)\right)^N}\right)}
		{\q \hspace{-0.05cm}\left(\frac{Z^{\sss (t)}_{\sss N}(\beta, B)}{\left(\lambda_+(\beta, B)\right)^N}\right)}\right] \hspace{-0.1cm}
%\, = \, \frac{1}{N} \log \left[\frac{\q \left(\prod_{i=1}^{K^t_{\sss N}}\frac{\left(\lambda_+(\beta, B+t)\right)^{\sss L_{\sss N}(i)} + \left(\lambda_-(\beta, B+t)\right)^{\sss L_{\sss N}(i)}}{\left(\lambda_+(\beta, B+t)\right)^{\sss L_{\sss N}(i)}}\right)}{\q \left(\prod_{i=1}^{K^t_{\sss N}}\frac{\left(\lambda_+(\beta, B)\right)^{\sss L_{\sss N}(i)} + \left(\lambda_-(\beta, B)\right)^{\sss L_{\sss N}(i)}}{\left(\lambda_+(\beta, B)\right)^{\sss L_{\sss N}(i)}}\right)} \right]  \\ 
		= \frac{1}{N} \log \hspace{-0.05cm} \left[ \frac{\q \hspace{-0.05cm}\left(\prod_{i=1}^{K^t_{\sss N}} 
		\left(1 + (r_{B+t})^{\sss L_{\sss N}(i)}\right)\right)}{\q \hspace{-0.05cm} \left(\prod_{i=1}^{K^t_{\sss N}}
		\left(1 + (r_B)^{\sss L_{\sss N}(i)}\right)\right)} \right],
	\end{equation*}
where, as in \cite{GGPvdH1}, we have defined
	\be\label{defalpha}
	r_{\sss B} \, = \, r(\beta, B) \, = \, \frac{\lambda_-(\beta, B)}{\lambda_+(\beta, B)}\;. 
	\ee
Then
	\be \label{c_diff_ann}
	\widetilde{c}_{\sss N}(t) 
	= \log \lambda_+(\beta, B+t) - \log \lambda_+(\beta, B) 
	+ \frac{1}{N} \log \left[ \frac{\q\left(\prod_{i=1}^{K^t_{\sss N}} \left(1 + (r_{B+t})^{\sss L_{\sss N}(i)}\right)\right)}
	{\q \left(\prod_{i=1}^{K^t_{\sss N}} \left(1 +(r_B)^{\sss L_{\sss N}(i)}\right)\right)} \right]. 
	\ee
Our aim is to show that the double derivative arises from the first term only, the second derivative of the last term vanishes.\\

\paragraph{Computation of the second derivative of the cumulant generating function.}
The second derivative of (\ref{c_diff_ann}) is
	\be
	\label{der2ctildalim}
	\widetilde{c}''_{\sss N}(t) 
	\,=\, \frac{\partial^2}{\partial t^2} \log \lambda_+(\beta, B+t) 
	\,+\, \frac{1}{N\widetilde{D}_{\sss N}(t)}\left[\widetilde{I}_{\sss N}(t) + \widetilde{II}_{\sss N}(t) + 	\frac{\widetilde{III}_{\sss N}(t)}{\widetilde{D}_{\sss N}(t)}\right],
	\ee
where
	\begin{align*}
	\widetilde{I}_{\sss N}(t) 
	&= \q \Big[ \displaystyle \sum_{i=1}^{K^t_{\sss N}}L_{\sss N}(i)(L_{\sss N}(i)-1)(r_{\sss B+t})^{\sss L_{\sss N}(i)-2}
	(r'_{\sss B+t})^2 \\
	& \quad \quad \quad \quad \quad \; + L_{\sss N}(i)(r_{\sss B+t})^{\sss L_{\sss N}(i)-1} r''_{\sss B+t}
	\prod \limits_{\substack{j=1 \\ j\neq i}}^{K^t_{\sss N}}\left(1 + (r_{\sss B+t})^{\sss L_{\sss N}(j)}\right) \Big],\\ 
	\widetilde{II}_{\sss N}(t)  
	&=  \q \Big[ \displaystyle \sum_{i=1}^{K^t_{\sss N}}\sum \limits_{\substack{j=1 \\ j\neq i}}^{K^t_{\sss N}}
	L_{\sss N}(i)L_{\sss N}(j) (r_{\sss B+t})^{\sss L_{\sss N}(i)+L_{\sss N}(j)-2}(r'_{\sss B+t})^2 
	\prod \limits_{\substack{l=1 \\ l\neq i,j}}^{K^t_{\sss N}}\left(1 + (r_{\sss B+t})^{\sss L_{\sss N}(l)} \right)\Big],\\ 
	\widetilde{III}_{\sss N} (t) 
	&=  \Big[\q \Big( \displaystyle 
	\sum_{i=1}^{K^t_{\sss N}}L_{\sss N}(i)(r_{\sss B+t})^{\sss L_{\sss N}(i)-1} r'_{\sss B+t}
	\prod \limits_{\substack{j=1 \\ j\neq i}}^{K^t_{\sss N}}\left(1 + (r_{\sss B+t})^{\sss L_{\sss N}(j)}\right)\Big)\Big]^2\; ,\\ 
	\widetilde{D}_{\sss N}(t) 
	&=  \q\Big[\displaystyle \prod_{i=1}^{K^t_{\sss N}}\left(1 + (r_{\sss B+t})^{\sss L_{\sss N}(i)}\right)\Big]. 
\end{align*}
	
\paragraph{Uniform bound of the averaged normalized partition function.}
To analyze the contributions above we show that the averaged normalized partition
function of $\CMtwo{N}$ is uniformly bounded:
\begin{lemma}[The partition function on tori]
\label{lem-part-function-tori}
%Let $K_n$ denote the number of cycles in $\CMtwo{n}$ and $(L_n(i))_{i=1}^{K_n}$ their lengths. 
For every $\gamma<1$ and $\alpha\in(0,\infty)$, there exists a constant $A=A(\alpha,\gamma)$ such that, uniformly in $N$,
	\begin{equation*}
	\q\big[\prod_{i=1}^{K_N^t} \big(1+\alpha \gamma^{L_N(i)}\big)\big]\leq A.
	\end{equation*}
%Consequently, 
%	\eqn{
%	\bar{Z}_n^{\sss(t)}(\beta,B)\leq A(ar^2, r).
%	}
\end{lemma}

\proof Denote ${\cal Z}_N=\q\big[\prod_{i=1}^{K_N^t} \big(1+\alpha \gamma^{L_N(i)}\big)\big]$.
For the proof we use induction in $N$. The induction hypothesis is that there exists an $A>1$ such that
	\eqn{
	\label{IH-Zn}
	{\cal Z}_N\leq A\big(1-\frac{1}{2\sqrt[3]{N+1}}\big).
	}
Fix $M\geq 1$ large. We note that we can fix $A$ so large that the inequality is trivially satisfied for $N\leq M$. 
To advance the induction hypothesis we first derive a recursion relation for ${\cal Z}_N$.
We have	
\begin{align}\label{induct_alpha}
		{\cal Z}_ N\, &= \, \sum_{l=1}^N 
		\q(L_N(1)=l) \; \q \Big(\prod_{i=1}^{K^t_{\sss N}}  
		\left(1 + \alpha \gamma^{\sss L_{\sss N}(i)}\right)\Big| L_N(1)=l\Big) \nn\\
		&= \, \sum_{l=1}^N \q(L_N(1)=l) \left(1 + \alpha \gamma^l\right) {\cal Z}_{\sss N-l}.
	\end{align}
Indeed, the average of $\prod_{i=1}^{K^t_{\sss N}}\left(1 + \alpha \gamma^{\sss L_{\sss N}(i)}\right)$ conditioned on  $L_N(1)=l$, reduces  to the average on a $\CMNtwo$ graph with $N-l$ vertices of a similar product. This average gives rise to the factor ${\cal Z}_{\sss N-l}$ in \eqref{induct_alpha}, while the term corresponding to the first cycle is factorized, being $\left(1 + \alpha \gamma^l\right)$. 	
Substituting the induction hypothesis into \eqref{induct_alpha} leads to
	\begin{align*}
	{\cal Z}_N&\leq A\sum_{l=1}^{N} \q(L_N(1)=l)\big(1+\alpha \gamma^{l}\big)\big(1-\frac{1}{2\sqrt[3]{N-l+1}}\big)\\
	&\leq A\sum_{l=1}^{N} \q(L_N(1)=l)\big(1+\alpha \gamma^{l}\big)
	-A\sum_{l=1}^{N} \q(L_N(1)=l)\frac{1}{2\sqrt[3]{N-l+1}}.
	\end{align*}
It is not hard to see that
	\begin{equation*}
	\sum_{l=1}^{N} \q(L_N(1)=l)\gamma^{l}\leq c/(N+1),
	\end{equation*}
while there exists a constant $\theta>1$ such that
	\eqn{
	\label{theta-bd-cube-root}
	\sum_{l=1}^{N} \q(L_N(1)=l)\frac{1}{\sqrt[3]{N-l+1}}\geq \frac{\theta}{\sqrt[3]{N+1}}.
	}
Indeed, by \cite[Exercise 4.1]{vdHII}, or an explicit computation, $L_N(1)/N\convd T$, where $T$ has density $f_{\sss T}(x)$ given by
	\begin{equation*}
	f_{\sss T}(x)=\frac{1}{2\sqrt{1-x}}.
	\end{equation*}
Therefore, rewriting the sum in \eqref{theta-bd-cube-root} we have: %using the symbol $\q$ to denote both the measure and the corresponding expectation
	\begin{equation*}
	\sum_{l=1}^{N} \q(L_N(1)=l)\frac{1}{\sqrt[3]{N-l+1}}
	=\frac{1}{\sqrt[3]{N+1}}\q\Big[\frac{1}{\sqrt[3]{1-L_N(1)/(N+1)}}\Big],
	\end{equation*}
and by Fatou's Lemma and weak convergence, we obtain
	\begin{equation*}
	\liminf_{N\rightarrow \infty} \q\Big[\frac{1}{\sqrt[3]{1-L_N(1)/(N+1)}}\Big]\geq\expec\Big[\frac{1}{\sqrt[3]{1-T}}\Big]>1.
	\end{equation*}
Since we can assume that $N\geq M$, which is sufficiently large, we thus obtain \eqref{theta-bd-cube-root}.
Thus,
	\begin{equation*}
	{\cal Z}_N\leq A\Big(1+\frac{c}{N+1}-\frac{\theta}{2\sqrt[3]{N+1}}\Big)\leq A\Big(1-\frac{1}{2\sqrt[3]{N+1}}\Big),
	\end{equation*}
when $N$ is sufficiently large. This advances the induction hypothesis and completes the proof
of the lemma.
\qed
\bigskip

\paragraph{Analysis of the second derivative of the cumulant generating function.}
Armed with Lemma \ref{lem-part-function-tori}, it is now easy to show that all the contributions 
in the second term of the r.h.s.\ of \eqref{der2ctildalim} indeed vanish on a sequence $t_{\sss N} = o(1)$.
%\begin{lemma}
%	\label{lem-I-II-III-terms}
To see this, let $t>0$ and $(t_{\sss N})_{N\geq 1}$ a sequence of real numbers  such that $t_{\sss N} \in [0 , t/\sqrt{N}]$. 
%Then,
%	\be
%	\lim_{N \rightarrow \infty} \frac{1}{N\widetilde{D}_{\sss N}(t_{\sss N})}
%	\left[\widetilde{I}_{\sss N}(t_{\sss N}) + \widetilde{II}_{\sss N}(t_{\sss N}) 
%	+ \frac{\widetilde{III}_{\sss N}(t_{\sss N})}{\widetilde{D}_{\sss N}(t_{\sss N})}\right] = 0.
%	\ee
%\end{lemma}
%%\vspace{0.5cm}
%
%\begin{proof}
We consider first the term $\widetilde{I}_{\sss N}(t_{\sss N})$. As in \cite[Lemma 3.1]{GGPvdH1}, there exists a constant $C>0$  such that
	\begin{equation*}
	\sum_{i=1}^{K^t_{\sss N}} \left |L_{\sss N}(i)(L_{\sss N}(i)-1)(r_{\sss B+t_{\sss N}})^{\sss L_{\sss N}(i)-2}(r'_{B+t_{\sss N}})^2  + L_{\sss N}(i)(r_{\sss B+t_{\sss N}})^{\sss L_{\sss N}(i)-1} r''_{B+t_{\sss N}}\right |  \leq C \cdot K^t_{\sss N}, 
%& \sum_{i=1}^{K^t_{\sss N}} L_{\sss N}(i)(L_{\sss N}(i)-1)r_{\sss B+t_{\sss N}}^{\sss L_{\sss N}(i)-2}(r'_{B+t_{\sss N}})^2  + L_{\sss N}(i)r_{\sss B+t_{\sss N}}^{\sss L_{\sss N}(i)-1} r''_{B+t_{\sss N}} \, \leq \, \\ 
	\end{equation*}
since $r_{\sss B+t_{\sss N}}<1$.
Then, using the Cauchy-Schwarz inequality and recalling that $0< r <1$, we obtain
	\begin{align*}
		|\widetilde{I}_{\sss N}(t_{\sss N})| \; 
		&\leq \, C \cdot \q \Big(K^t_{\sss N} \, \prod_{i=1}^{K^t_{\sss N}}
		\left(1 + (r_{\sss B+t_{\sss N}})^{\sss L_{\sss N}(i)}\right)\Big) \\ 
		&\leq \, C \cdot \q\left(\left(K^t_{\sss N}\right)^2\right)^{1/2} \cdot 
		\q \Big(\prod_{i=1}^{K^t_{\sss N}}\left(1 + (r_{\sss B+t_{\sss N}})^{\sss L_{\sss N}(i)}\right)^2\Big)^{1/2} \\ 
		& \leq \, C \q\Big(\left(K^t_{\sss N}\right)^2\Big)^{1/2} \cdot \q \Big(\prod_{i=1}^{K^t_{\sss N}}
		\left(1 + 3(r_{\sss B+t_{\sss N}})^{\sss L_{\sss N}(i)}\right)\Big)^{1/2}.
	\end{align*}
Using Lemma \ref{lem-part-function-tori} with $\alpha = 3$ and $\gamma = r_{\sss B+t_{\sss N}}$
we conclude that
	\begin{equation*}
	\q \Big(\prod_{i=1}^{K^t_{\sss N}}\left(1 + 3(r_{\sss B+t_{\sss N}})^{\sss L_{\sss N}(i)}\right)\Big)^{1/2} 
	\, \leq \, A^{\frac{1}{2}}.
	\end{equation*}
Finally, since $\widetilde{D}_{\sss N}(t_{\sss N}) \geq 1$,
	\begin{equation*}
	\frac{|\widetilde{I}_{\sss N}(t_{\sss N})|}{N \widetilde{D}_{\sss N}(t_{\sss N})} 
	\, \leq \, \frac{C \cdot A^{\frac{1}{2}} \cdot \log{N} }{N} 
	\, \stackrel{N \to \infty}{\longrightarrow}\, 0.
	\end{equation*}
Similar computations allow us to estimate $\widetilde{II}_{\sss N}(t_{\sss N})$ and $\widetilde{III}_{\sss N}(t_{\sss N})$ to obtain
	\begin{equation*}
	\lim_{N \to \infty } \frac{\widetilde{II}_{\sss N}(t_{\sss N})}{N \widetilde{D}_{\sss N}(t_{\sss N})} = 0, 
	\qquad  \lim_{N \to \infty } \frac{\widetilde{III}_{\sss N}(t_{\sss N})}{N \left(\widetilde{D}_{\sss N}(t_{\sss N})\right)^2} = 0\;.
	\end{equation*}
%concluding the proof of the lemma.
%\end{proof}

\paragraph{Completion of the proof of Theorem \ref{CLT_annealed1}.}
Having proved that 
	\begin{equation*}
	\lim_{N \rightarrow \infty} \frac{1}{N\widetilde{D}_{\sss N}(t_{\sss N})}
	\left[\widetilde{I}_{\sss N}(t_{\sss N}) + \widetilde{II}_{\sss N}(t_{\sss N}) 
	+ \frac{\widetilde{III}_{\sss N}(t_{\sss N})}{\widetilde{D}_{\sss N}(t_{\sss N})}\right] = 0\;,
	\end{equation*}
the combination of  
%Lemma \ref{lem-I-II-III-terms}, 
\eqref{cumgfcm2cder2} and \eqref{der2ctildalim} yields the proof of the annealed CLT, i.e.,
	\begin{align*}
		\lim_{N \rightarrow \infty} \log \widetildep 
		\Big[\exp\Big(t \frac{S_{\sss N} - \widetildep(S_{\sss N})}{\sqrt{N}}\Big) \Big]   
		&=   \frac{t^2}{2} \left.\frac{\partial^2}{\partial t^2} \log \lambda_+(\beta, B+t)\right|_{t=0}\\
		&= \, \frac{t^2}{2} \frac{\cosh(B) \e^{-4\beta}}{(\sinh(B)+\e^{-4\beta})^{3/2}}.
	\end{align*}
Therefore, we conclude that the annealed CLT has the same variance as in averaged quenched case  \cite{GGPvdH1}, i.e., the variance in both cases is the susceptibility of the one-dimensional Ising model.\qed
%\begin{align}\nonumber
%\lim_{N \rightarrow \infty} \log \widetildep \left[\exp\left(t \, \frac{S_{\sss N} - \widetildep(S_{\sss N})}{\sqrt{N}}\right) \right] \, &= \, \lim_{N \rightarrow \infty} \frac{t^2}{2}\widetilde{c}''_{\sss N}(t_{\sss N}) \, =  \, \frac{t^2}{2}\cdot \left.\frac{\partial^2}{\partial t^2} \log \lambda_+(\beta, B+t)\right|_{t=0}\\ 
%&= \, \frac{t^2}{2}\cdot \frac{\cosh(B) e^{-4\beta}}{(\sinh(B)+e^{-4\beta})^{3/2}}.\\ \nonumber
%\end{align}

\section{Proofs for $\CMNonetwo$}
\label{four}

In this section, we consider the Configuration Model $\CMNonetwo$, introduced in Section 1.2. In this graph, the connected components are either cycles or tori (which we indicate by a superscript $(t)$) connecting vertices of degree $2$, or lines (indicated by a superscript $(l)$) having vertices of degree 2 between two vertices of degree 1. In order to state some properties of the number of lines and tori, we need to introduce some notation. By taking $p \in (0,1)$, let us define the number of vertex of degree 1 and 2 by
	\begin{equation*}
		n_1 \;:=\;  \# \left\{i \in [N] \colon d_i = 1\right\}\,=\, N - \lfloor pN \rfloor,\qquad
		n_2  \;:=\;  \# \, \left\{i \in [N] \colon d_i =2\right\}\,=\, \lfloor pN\rfloor,
	\end{equation*}
and the total degree of the graph by
	\be\label{defln}
	\ell_{\sss N} \; =\; \sum_{i \in [N]} d_{i} \,=\, 2n_2 + n_1 \,=\, N+\lfloor pN\rfloor.
	\ee
Then, the number of edges is given by $\ell_{\sss N}/2$.  Let us also denote by $K_{\sss N}$ the number of connected components in the graph and by $K_{\sss N}^{\sss (l)}$ and $K_{\sss N}^{\sss (t)}$ the number lines and tori. Obviously,
\begin{equation*}
	K_{\sss N}= K_{\sss N}^{\sss (l)} + K_{\sss N}^{\sss (t)}.
	\end{equation*}
Because every line uses up two vertices of degree 1, the number of lines is given by $n_1/2$, i.e., $K_{\sss N}^{\sss (l)} = (N - \lfloor pN \rfloor)/2$ a.s.. Regarding the number of cycles, we have that $K_{\sss N}^{\sss (t)}$ has the same distribution  of $K_{\bar{N}}^t$, where $\overline{N}$ is the (random) number of vertices with degree 2 that do not belong to any line and $K_{\bar{N}}^t$ is the number of tori on this set of vertices.  Then, since this subset forms a $\mathrm{CM}_{\sss \bar{N}}(\mathrm{\textbf{2}})$ graph, we can apply \cite[(3.16) in Section 3.2]{GGPvdH1}, obtaining that $K_{\sss N}^{\sss (t)}/N\stackrel{\mathbb{P}}{\longrightarrow} 0$, so that also
	\begin{equation*}
	K_{\sss N}/N \stackrel{\mathbb{P}}{\longrightarrow}(1-p)/2.
\end{equation*}
%Because to form a line  two vertices of degree 1 are needed, the number of lines is given by $n_1/2$. Indeed, it can be shown that $K_{\sss N}(l) = \frac{(1-p)N}{2}\,$ almost surely. %then $\, \frac{K_{\sss N}(l)}{N} \stackrel{\mathbb{P}}{\longrightarrow} \frac{(1-p)}{2} \,$.\\ 
%Regarding the number of cycles, we have that $ \frac{K_{\sss N}(t)}{N}$ has the same distribution  of $\frac{K^{t}_{\bar{N}}}{N}$, where $\overline{N}$ is the number of vertices with degree 2 that do not belong to any line.  Then, since this subset of vertices form a $\CMNtwo$ graph, we can apply $\frac{K^t_{\sss N}}{N} \stackrel{\mathbb{P}}{\longrightarrow} 0$, from \cite[Section 3.2]{GGPvdH1}, obtaining that $ \frac{K_{\sss N}(t)}{N} \stackrel{\mathbb{P}}{\longrightarrow} 0$. 
%From this, we have that
%\be
%\frac{K_{\sss N}}{N} \stackrel{\mathbb{P}}{\longrightarrow} \frac{(1-p)}{2}.
%\ee
%because $\, \frac{K_{\sss N}(l)}{N} \stackrel{\mathbb{P}}{\longrightarrow} \frac{(1-p)}{2} \,$ and $\, \frac{K_{\sss N}(t)}{N} \stackrel{\mathbb{P}}{\longrightarrow} 0$.\\ $K_{\sss N}(l) = \frac{(1-p)N}{2}\,$ a.s. and $\frac{K_{\sss N}(t)}{N} \stackrel{d}{=} \frac{K_{\bar{N}}}{N}$, where $\bar{N}$ is the number of vertices with degree = 2 that are not in lines.
\medskip

Denoting the length (i.e. the number of vertices) in the $i$th line and $j$th torus (for an arbitrary labeling) by $L^{\sss (l)}_{\sss N}(i)$ and $L^{\sss (t)}_{\sss N}(j)$, the partition function can be computed as 
%in the case of $\CMNtwo$ as
	\be\label{zcm12}
	Z_{\sss N}(\beta, B) = \prod_{i=1}^{K_{\sss N}^{\sss (l)}} Z_{L^{\sss (l)}_{\sss N}(i)}^{\sss (l)} (\beta, B) \cdot 	\prod_{i=1}^{K_{\sss N}^{\sss (t)}} Z_{L^{\sss (t)}_{\sss N}(i)}^{\sss (t)} (\beta, B),
	\ee
where, by (\ref{tre}),
	\begin{equation*}
%Z_{L^{\sss (l)}_{\sss N}(i)}^{\sss (l)} (\beta, B) = \left(A_+ \lambda_+^{L_{\sss N}^{\sss (l)}(i)} + A_- \lambda_-^{L_{\sss N}^{\sss (l)}(i)}\right),\quad 
	Z_{L^{\sss (t)}_{\sss N}(i)}^{\sss (t)} (\beta, B) = \lambda_+^{L_{\sss N}^{\sss (t)}(i)} + \lambda_-^{L_{\sss N}^{\sss (t)}(i)},
	\end{equation*}
while the partition function on each line is obtained using the partition function on one-dimensional Ising model with free boundary condition
%, as perfomed in 
\cite[Section 3.1]{GGPvdH1}
as
	\begin{equation*}
	Z^{\sss (l)}_{\sss N}  =  A_+ \lambda_+^N + A_- \lambda_-^N,
	\end{equation*}
where
	\be\label{A_piu_meno}
	A_{\pm} = A_{\pm}(\beta, B) 
	= \frac{\e^{-2\beta}\e^{\pm B} + (\lambda_+-\e^{\beta+B})^2 \e^{\mp B} 
	\pm 2 \e^{-\beta} (\lambda_+ - \e^{\beta+B})}
	{[\e^{-2\beta} + (\lambda_+ - \e^{\beta+B})^2 ] \lambda_{\pm}}.
	\ee
This is the starting point of our analysis of the annealed Ising measure on $\CMNonetwo$.

\subsection{Annealed CLT: proof of Theorem \ref{CLT_annealed2}}
\label{sec-ann-CLT-CM12}
In order to prove the CLT in the annealed setting, we will show that
	\eqn{
	\label{CLT-ann-12-aim}
	\lim_{N \rightarrow \infty}\, 
	\widetildep \Big[\exp \Big(\frac{t}{\sqrt{N}} \big(S_{\sss N} - \widetildep(S_{\sss N})\big)\Big)\Big] \, = \, \exp(\sigma_2^2t^2/2), \quad t\in \bbR.
	}
From now on, to alleviate notation we will omit the dependence on $\beta$ and 
abbreviate $B_{\sss N}=B + \frac{t}{\sqrt{N}}$. We start by writing
	\begin{equation*}
	\widetildep \Big[\exp \Big(\frac{t}{\sqrt{N}} S_{\sss N} \Big)\Big ]=\frac{\q [Z_{\sss N}(B_{\sss N})]}{\q [Z_{\sss N}(B)]}
=\, \frac {\q [\e^{NF_{B_{\sss N}} (p^{\sss(N)})+ NE_{\sss N}(B_{\sss N})}]}{\q [\e^{NF_{B} (p^{\sss(N)})+ N E_{\sss N}(B)}]},
	\end{equation*}
where, by \cite{GGPvdH1},
	\eqan{
	\label{f-def}
	F_{B}(p^{\sss(N)})
	%\,=\, F_{\beta, B}\left(p^{\sss(N)}\right) 
	&= \log \lambda_+(B)  
	+ \sum_{l \geq 2} p_l^{\sss{(N)}} \log \left(A_+(B) + A_- (B)\left(r(B)\right)^{l}\right),\\
	\label{e-def}
	E_{\sss N}(B)
	%\, = \,E_{\sss N}(\beta, B) 
	&= \frac{1}{N} \sum_{i=1}^{K_{\sss N}^{\sss (t)}} 
	\log \Big(1 + r(B)^{L_{\sss N}^{\sss (t)}(i)} \Big),
	}
with $r(B)=r_B$ defined in \eqref{defalpha} and $p^{\sss(N)} = \left(p_l^{\sss{(N)}}\right)_{l\geq 2}$ the empirical distribution of the lines lengths given by 
	\be\label{def_pelle}
	p_l^{\sss{(N)}} \,:=\,  \frac{1}{N} \sum_{i=1}^{K_{\sss N}^{\sss (l)}} 
	\mathbbm{1}_{\{L_{\sss N}^{\sss (l)}(i) = l\}}.
	\ee
	
\paragraph{Analysis of the annealed partition function.}
We have 
	\eqan{\label{pluto}
	\e^{N F_{B}\left(p^{\sss(N)}\right)}
	&=\left(\lambda_+(B)\right)^N \prod_{l=2}^{\infty}\left(A_+(B) + A_-(B) \left(r(B)\right)^{l}\right)^{N_l }\nn\\
	&=\left(\lambda_+(B)\right)^N  \prod_{l=2}^{\infty} \left(A_{+}(B)\right)^{N_l}\prod_{l=2}^{\infty}\left(1+ a(B)\left(r(B)\right)^{l}\right)^{N_l}.
	}
where 
	\be
	\label{a-def}
	a(B) = \frac{A_-(B)}{A_+(B)}
	\ee	
and $N_l = N p_l^{\sss{(N)}}$ is the number of lines of length $l$. We rewrite the second factor in \eqref{pluto} as
	\begin{equation*}
	\prod_{l=2}^{\infty} \left(A_{+}(B)\right)^{N_l}=\left(A_+(B)\right)^{n_1/2},
	\end{equation*}
since $\sum_{l\geq 2} N_l=n_1/2$. Therefore, we arrive at
	\begin{equation*}
	\e^{N F_{B}\left(p^{\sss(N)}\right)} = \lambda_+^N(B)  A_+^{n_1/2}(B)\prod_{l=2}^{\infty}\left(1+ a(B) r^l(B)\right)^{N_l}\, = \, \lambda_+^N(B)  A_+^{n_1/2}(B)\prod_{l=2}^{\infty}c_l(B)^{N_l},
	\end{equation*}
where we define 
\begin{equation*}
c_l(B) := 1+ a(B) r^l(B)\;.
\end{equation*}
Next, define 
	\begin{equation*}
	M_{\sss N}=N-\sum_{l\geq 2}l N_l
	\end{equation*}
for the number of vertices that are {\em not} part of a line. Then, 
denoting by $Z_{\sss N}^{(2)}(B)$  the partition function of $\CMNtwo$,
	\eqan{
	\label{rewrite-part-func-12}
	\frac{\q [Z_{\sss N}(B_{\sss N})]}{\q [Z_{\sss N}(B)]}&= \frac{\q [\e^{NF_{B_{\sss N}} (p^{\sss(N)})}\bar{Z}_{\sss M_{\sss N}}^{\sss (2)}(B_{\sss N})]}{\q[\e^{NF_{B} (p^{\sss(N)})}\bar{Z}_{\sss M_{\sss N}}^{\sss (2)}(B)]}\\ \nonumber
	& = \frac{\lambda_+^N(B_{\sss N})A_+^{n_1/2}(B_{\sss N}) Q_{\sss N}\big[\bar{Z}_{\sss M_{\sss N}}^{\sss (2)}(B_{\sss N})
	\prod_{l=2}^{\infty} c_l(B_{\sss N})^{N_l}\big]}
	{\lambda_+^N(B)A_+^{n_1/2}(B) Q_{\sss N}\big[\bar{Z}_{\sss M_{\sss N}}^{\sss (2)}(B)\prod_{l=2}^{\infty} c_l(B)^{N_l}\big]},
}
%\MLP{Who is $c_l(B)^{N_l}$?\\If $c_l(B)^{N_l}=\left(A_+ + A_- \left(\frac{\lambda_-}{\lambda_+}\right)^l\right)^{N_l }$, the (\ref{rewrite-part-func-12}) is correct.\\ If $c_l(B)^{N_l}=\left(1+ ar^l\right)^{N_l}$, from (\ref{4.82}), in the formula (\ref{rewrite-part-func-12}) there is an extra factor $(A_+)^{n_1/2}$, i. e. :}
where we write
	\begin{equation*}
	\bar{Z}_{\sss N}^{\sss(2)}(B)=\lambda_+^{-N}Z_{\sss N}^{\sss(2)}(B).
	\end{equation*}
	
\paragraph{Asymptotic behavior of the annealed partition function.}
The key result for the proof of Theorem \ref{CLT_annealed2}  is the following proposition
that establishes the exponential growth of the annealed partition function with polynomial corrections:

\begin{proposition}\label{prop-exp-part-funct}
The following holds true:
\begin{itemize}
\item[(a)] For $B \neq 0$, there exist  $I=I(B)$ and $J=J(B)$ such that, as $N \rightarrow \infty$,
	\eqn{
	\label{asymp-part-func-12}
	Q_{\sss N}\big[\bar{Z}_{\sss M_{\sss N}}^{\sss (2)}(B)\prod_{l=2}^{\infty} c_l(B)^{N_l}\big]
	=J(B)\e^{I(B)N}(1+o(1)).
	}
The function $B\mapsto J(B)$ is continuous, while $B\mapsto I(B)$ is infinitely differentiable.
\item[(b)] 
Given $t\in \bbR$ 
%and the sequence $B_{\sss N} = \frac{t}{\sqrt{N}}$, 
there exist  $\bar{I}=\bar{I}(t)$ and $\bar{J}$ such that, as $N \rightarrow \infty$,

\be\label{IeJBN0}
Q_{\sss N}\big[\bar{Z}_{\sss M_{\sss N}}^{\sss (2)}\left(\frac{t}{\sqrt{N}}\right)\prod_{l=2}^{\infty} c_l\left(\frac{t}{\sqrt{N}}\right)^{N_l}\big] =\bar{J}\e^{\bar{I}(t/\sqrt{N}) N}(1+o(1)).
\ee
%\be\label{IeJB0}
%Q_{\sss N}\big[\bar{Z}_{\sss M_{\sss N}}^{\sss (2)}(0)\prod_{l=2}^{\infty} c_l(0)^{N_l}\big] =\bar{J}(0).
%\ee
The function  $t\mapsto \bar{I}(t)$ is infinitely differentiable.
\end{itemize}
\end{proposition}
\medskip

\noindent
%We now perform the proof of \eqref{CLT-ann-12-aim} subject to Proposition \ref{prop-exp-part-funct}.
\medskip

\noindent
{\em Proof of Theorem \ref{CLT_annealed2} subject to Proposition \ref{prop-exp-part-funct}.}
We start proving the theorem for  $B \neq 0$. 
%For the sake of notation, we write $I(B)=I(\beta,B)$ and $J(B)=J(\beta,B)$. 
We substitute \eqref{asymp-part-func-12} into \eqref{rewrite-part-func-12} to arrive at
	\eqan{
	\frac{\q [Z_{\sss N}(B_{\sss N})]}{\q [Z_{\sss N}(B)]}
	&= (1+o(1))\frac{\lambda_+^N(B_{\sss N})A_+^{n_1/2}(B_{\sss N})J(B_{\sss N})\e^{I(B_{\sss N})N}}
	{\lambda_+^N(B)A_+^{n_1/2}(B)J(B)\e^{I(B)N}}\\
	&=(1+o(1))\Big(\frac{\lambda_+(B_{\sss N})}{\lambda_+(B)}\Big)^N\Big(\frac{A_+(B_{\sss N})}{A_+(B)}\Big)^{n_1/2}\e^{N(I(B_{\sss N})-I(B))},\nn
	}
where we use the fact that $B\mapsto J(B)$ is continuous to obtain that $J(B_{\sss N})=(1+o(1))J(B)$. We can next use the differentiability of $B\mapsto I(B)$ and the fact that $B_{\sss N}=B+t/\sqrt{N}$ to expand out
	\begin{align}\label{lapiuimp}\nn
	\frac{\q [Z_{\sss N}(B_{\sss N})]}{\q [Z_{\sss N}(B)]}
	= &(1+o(1))\e^{t\sqrt{N} \big[\frac{\partial}{\partial t}\log{\lambda_+(B+t)}|_{t=0}+\frac{n_1}{2N}\frac{\partial}{\partial t}\log{A_+(B+t)}|_{t=0}+\frac{\partial}{\partial t}I(B+t)|_{t=0}\big]}\\
	& \times \e^{\sigma_2^2t^2/2},
	\end{align}
where
	\eqn{\label{var_sigma_2}
	\sigma_2^2 \,=\, \frac{\partial^2}{\partial t^2}\log{\lambda_+(B+t)}|_{t=0}+\frac{(1-p)}{2}\frac{\partial^2}{\partial t^2}\log{A_+(B+t)}|_{t=0}+\frac{\partial^2}{\partial t^2}I(B+t)|_{t=0}.
	}
Since 
\begin{align*}
\widetildep(S_{\sss N})= &N\big[\frac{\partial}{\partial t}\log{\lambda_+(B+t)}|_{t=0}+\frac{n_1}{2N}\frac{\partial}{\partial t}\log{A_+(B+t)}|_{t=0}+\frac{\partial}{\partial t}I(B+t)|_{t=0}\big]\\
&+o(\sqrt{N}) \;,
\end{align*}
then \eqref{lapiuimp}  implies  \eqref{CLT-ann-12-aim}, thus proving the theorem in the case $B\ne 0$.
\medskip

\noindent
For $B=0$, in a similar way now using \eqref{IeJBN0}, we get
	\begin{equation*}
	\frac{\q [Z_{\sss N}(t/\sqrt{N})]}{\q [Z_{\sss N}(0)]}
	= (1+o(1))\e^{t\sqrt{N} \big[\frac{\partial}{\partial t}\log{\lambda_+(t)}|_{t=0}+\frac{n_1}{2N}\frac{\partial}{\partial t}\log{A_+(t)}|_{t=0}+\frac{\partial}{\partial t}\bar{I}(t)|_{t=0}\big]} \e^{\bar{\sigma}_2^2t^2/2},\nn
	\end{equation*}
where
	\begin{equation*}
	\bar{\sigma}_2^2 \,=\, \frac{\partial^2}{\partial t^2}\log{\lambda_+(t)}|_{t=0}+\frac{(1-p)}{2}\frac{\partial^2}{\partial t^2}\log{A_+(t)}|_{t=0}+\frac{\partial^2}{\partial t^2}\bar{I}(t)|_{t=0}.
	\end{equation*}
\qed
\smallskip

\paragraph{Strategy to prove asymptotic behavior.}
The remainder of this section is devoted to the proof of Proposition \ref{prop-exp-part-funct}.
We use the law of total probability to write
	\begin{align}\label{prob_tot_prop4.1}\nn
	&Q_{\sss N}\big[\bar{Z}_{\sss M_{\sss N}}^{\sss(2)}(B)\prod_{l=2}^{\infty} c_l(B)^{N_l}\big]\\
	&=\sum_{m=0}^{n_2}\expec_{m}[\bar{Z}_{m}^{\sss(2)}(B)] Q_{\sss N}\big[\prod_{l=2}^{\infty} c_l(B)^{N_l}\mid M_{\sss N}=m\big] \q(M_{\sss N}=m),
	\end{align}
where we denote by the symbol $\expec_{m}$ the expectation with respect to an independent ${\mathrm{CM}_{m}(\mathrm{\textbf{2}})}$.
%{where we denote by $\expec_{m}$  the expectation with respect to ${\mathrm{CM}_{m}(\mathrm{\textbf{2}})}$, which is conditionally independent of   $\CMNonetwo$ given $m$, the number of vertices in tori.}
%\medskip
%The above formula contains three parts.  First we observe that,  by  Lemma \ref{lem-part-function-tori},  
%$\sup_m  \expec_m[\bar{Z}_{m}^{\sss(2)}(B)]$ is bounded.
%\medskip

Our aim is to prove that the asymptotic behavior  of \eqref{prob_tot_prop4.1} is essentially dominated by the term with $m=0$, which gives the exponential 
growth $J(B) e^{NI(B)}$ stated in Proposition \ref{prop-exp-part-funct}. 
%The terms of the sum corresponding to $m>0$, being subdominant  as $N\to \infty$,  contribute to the prefactor $J(B)$.
To achieve a full control we analyze in the following the three contributions whose product gives rise to the summand of \eqref{prob_tot_prop4.1}:
\begin{itemize}
\item[i)] $\expec_m[\bar{Z}_{m}^{\sss(2)}(B)]$: this is subdominant in the  limit $N\to\infty$ since, by  Lemma \ref{lem-part-function-tori},  
$\sup_m  \expec_m[\bar{Z}_{m}^{\sss(2)}(B)]$ is bounded. Therefore it will appear only in the prefactor $J(B)$.
\item[ii)]  $\q(M_{\sss N}=m)$: we study the distribution of the number of vertices in tori $M_{\sss N}$  in Lemma \ref{lem-vertices-tori};
in particular we prove the existence of a limiting distribution function in the limit $N\to\infty$.
\item[iii)] $Q_{\sss N}\big[\prod_{l=2}^{\infty} c_l(B)^{N_l}\mid M_{\sss N}=m\big]$:  this is rewritten explicitly in Lemma  \ref{lem-ZN-condM=m} 
and its asymptotics is computed in Lemmata \ref{lemma_asymptB} and \ref{asy_M_N_equal_zero}.
\end{itemize}
\medskip

\paragraph{The number of vertices in tori.}
We start by analyzing the random variable $M_{\sss N}$  representing the number of vertices belonging to tori.

\begin{lemma}[The number of vertices in tori]
\label{lem-vertices-tori}
When $N\rightarrow \infty$, there exists a random variable $M$ such that
	\begin{equation*}
	M_{\sss N}\convd M.
	\end{equation*}
%and, for every $b\in {\mathbb R}$ such that $|b|<(1+p)/(2p)$,
	%\eqn{
	%\expec[b^{M_{\sss N}}]\rightarrow \expec[b^M].
	%}
Further,
	\eqan{
	\label{law-MN}\nn
	&\q(M_{\sss N}=m)\\ \nn
	&=\frac{1}{(n_1+2n_2-1)!!} {{n_2}\choose{m}} 2^{n_2-m} (n_2-m)! (n_1-1)!! (2m-1)!! {{n_1/2+n_2-m-1}\choose {n_2-m}}\\ 
	&=2^{n_2} \frac{(n_1-1)!!n_2!}{(n_1+2n_2-1)!!} 2^{-2m} {{2m}\choose{m}} {{n_1/2+n_2-m-1}\choose {n_2-m}}.
	}
\end{lemma}

\proof Number the vertices of degree 2 in an arbitrary way. We write
	\begin{equation*}
	M_{\sss N}=\sum_{l=1}^{\infty} l N_{l}^{\sss(t)},
	\qquad
	\text{where}
	\qquad
	N_{l}^{\sss(t)}=\sum_{i=1}^{n_2} J_i(l),
	\end{equation*}
and $J_i(l)$ is the indicator that vertex $i$ is in a cycle of length $l$ of which vertex $i$ has the smallest label. We compute that
	\begin{equation*}
	\q[N_{l}^{\sss(t)}]
	=\frac{n_2}{l}\q(\text{vertex $1$ is in cycle of length $l$})
	\rightarrow \frac{1}{2 l} (2p/(1+p))^{l}\equiv \lambda_{l}.
	\end{equation*}
It is not hard to see, along the lines of \cite[Proposition 7.12]{vdH}, that $(N_{l}^{\sss(t)})_{l\geq 1}$ converges in distribution to a collection of {\em independent} Poisson random variables $(P_{l})_{l\geq 1}$ with parameters $(\lambda_{l})_{l\geq 1}$. Further, since $\q[N_{l}^{\sss(t)}]\leq \frac{1}{2 l} (n_2/\ell_{\sss N})^{l}$, which decays exponentially, the contribution from large $l$ equals zero whp, i.e., $\q(\exists l>T\colon N_{l}^{\sss(t)}>0)$ is small uniformly in $N$ for $T$ large. This shows that
	\begin{equation*}
	M_{\sss N}\convd \sum_{l\geq 1} l P_{l}\equiv M.
	\end{equation*}
Note that
	\eqn{\label{mom_gen_func_b}
	\q[b^M]=\prod_{l=1}^{\infty}\q[b^{l P_{l}}]=\prod_{l=1}^{\infty}\e^{(b^{l}-1)\lambda_{l}}
	=\e^{\sum_{l\geq 1}(b^{l}-1)\lambda_{l}},
	}
which is finite only when $b<(1+p)/(2p)$.%, which partly explains the condition on $b$ in Lemma \ref{lem-vertices-tori}. 

To prove \eqref{law-MN}, we note that
	\eqn{
	\label{law-MN-proof}
	\q(M_{\sss N}=m)=\frac{1}{(n_1+2n_2-1)!!}N(n_1,n_2,m),
	}
where $N(n_1,n_2,m)$ is the number of ways in which the half-edges can be paired such that there are precisely $m$ degree 2 vertices in cycles. We claim that
	\eqn{
	\label{Nn1n2m-form}
	N(n_1,n_2,m)={{n_2}\choose{m}}2^{n_2-m} (n_2-m)! (n_1-1)!!(2m-1)!!{{n_1/2+n_2-m-1}\choose {n_2-m}}.
	}
For this, note that 
\begin{enumerate}
	\item[(1)] there are ${{n_2}\choose{m}}$ ways to choose the $m$ vertices of degree 2 that are in cycles;
	\item[(2)] there are $(2m-1)!!$ ways to pair the half-edges that are incident to vertices in cycles;
	\item[(3)] there are $(n_1-1)!!$ ways to pair the vertices of degree 1 (and this corresponds to the pairing of degree 1 vertices in lines);
	\item[(4)] there are $(n_2-m)!$ ways to order the vertices that are in lines; 
	\item[(5)] there are $2$ ways to attach the half-edges of a degree 2 vertex inside a line, and there are in total $n_2-m$ degree 2 vertices in lines, giving $2^{n_2-m}$ ways to attach their half-edges; and
	\item[(6)] finally, there are ${{n_1/2+n_2-m-1}\choose {n_2-m}}$ ways to create $n_1/2$ lines with $n_2-m$ vertices of degree 2. 
\end{enumerate}
Multiplying these numbers out gives \eqref{Nn1n2m-form}.
This completes the proof of Lemma \ref{lem-vertices-tori}.\qed
\smallskip

\paragraph{Combinatorial expression of the partition function.}
%In the proof of Proposition \ref{prop-exp-part-funct} we rely on the following results. 
To perform the asymptotic analysis of the partition function $Q_{\sss N}\big[\prod_{l=2}^{\infty} c_l(B)^{N_l}\mid M_{\sss N}=m\big]$,
we rewrite it as double sum in  Lemma  \ref{lem-ZN-condM=m} and then we investigate the asymptotics of the summand in 
Lemma \ref{lemma_asymptB} by Stirling's formula. Finally, in Lemma \ref{asy_M_N_equal_zero}, we use the Laplace method to estimate the asymptotics of the double sum.

\begin{lemma}[Generating function of number of lines in $\CMNonetwo$]
\label{lem-ZN-condM=m}
For every $a,r$, for $c_l=1+ar^l$ for every $l\geq 2$,
	\eqn{\label{gen_fun_numb_lin}
	Q_{\sss N}\big[\prod_{l=2}^{\infty} c_l(B)^{N_l}\mid M_{\sss N}=m\big]
	\,=\, \sum_{\ell=0}^{n_1/2}\sum_{k=0}^{n_2-m}\,B_{\ell,k}^{\sss (N)}(n_2-m),
	}
where 
\be\label{def_Bmlk}
B_{\ell,k}^{\sss (N)}(n_2-m) \,=\, {{{n_1/2}\choose {\ell}}} (ar^2)^{\ell}  r^k \frac{{{{\ell+k-1}\choose {k}}}\,{{{n_1/2-\ell + n_2-m-k-1}\choose {n_2-m-k}}}}{{{{n_1/2+n_2-m-1}\choose {n_2-m}}}}.
\ee
\end{lemma}

\begin{proof} When $M_{\sss N}=m$, we have that $n_2-m$ vertices of degree 2 have to be divided over $n_1/2$ lines. Number the lines as $1, \ldots, n_1/2$ in an arbitrary way. Denote the number of degree 2 vertices in line $j$ by $Y_j$ and rewrite
	\begin{equation*}
	Q_{\sss N}\big[\prod_{l=2}^{\infty} c_l(B)^{N_l}\mid M_{\sss N}=m\big] = \hspace{-0.4 cm} \sum_{(i_1,\ldots, i_{n_1/2})} \hspace{-0.4 cm}
	\q\left(Y_1 = i_1,\ldots, Y_{n_1/2}=i_{n_1/2}\right)\prod_{j=1}^{n_1/2} (1+ar^{i_j+2})
	\end{equation*}
where $(i_1,\ldots, i_{n_1/2})$ is such that $i_1+\cdots+i_{n_1/2}=n_2-m$. Let $[n_1/2]=\{1, \ldots, n_1/2\}$, and expand out  $\prod_{j=1}^{n_1/2} (1+ar^{i_j+2})$ to obtain
	\be\label{4.65}
	\sum_{(i_1,\ldots, i_{n_1/2})}\q\left(Y_1 = i_1,\ldots, Y_{n_1/2}=i_{n_1/2}\right)
	\sum_{\Gamma\subseteq [n_1/2]} (ar^2)^{|\Gamma|}  \prod_{j\in \Gamma} r^{i_j}.
	\ee
where the sum over $\Gamma$ is over all subsets of $[n_1/2]$. We denote
	\begin{align*}
	N_{n_1,n_2-m} \, &= \, \# \big\{(i_1,\ldots, i_{n_1/2}) 
	\colon i_j \geq 0 \; \forall \, j \;\, \mbox{and} \, \sum_{j=1}^{n_1/2} i_j = n_2-m \big\} \\
	&= \; {{n_1/2 + n_2-m -1}\choose {n_2-m}},
	\end{align*}
so that \eqref{4.65} is equal to
	\begin{align*}
		&\sum_{(i_1,\ldots, i_{n_1/2})} \frac{1}{{{n_1/2 + n_2-m-1}\choose {n_2-m}}}
		\sum_{\Gamma\subseteq [n_1/2]} (ar^2)^{|\Gamma|} \sum_{k=0}^{n_2-m} r^k \indic{\sum_{j\in \Gamma}i_j=k} \,  \\ \nn
		&= \, \sum_{(i_1,\ldots, i_{n_1/2})} \frac{1}{{{n_1/2 + n_2-m-1}\choose {n_2-m}}}
		\sum_{\ell=0}^{n_1/2}{{{n_1/2}\choose {\ell}}}  (ar^2)^{\ell} \sum_{k=0}^{n_2-m} r^k \indic{(i_1+\ldots+i_\ell=k)}\,  \\ \nn
		&= \, \frac{1}{{{n_1/2 + n_2-m-1}\choose {n_2-m}}}\\ 
		&\quad \; \, \times \sum_{\ell=0}^{n_1/2}{{{n_1/2}\choose {\ell}}}  (ar^2)^{\ell} 		\sum_{k=0}^{n_2-m} r^k {{{\ell+k-1}\choose {k}}}\,{{{n_1/2-\ell + n_2-m-k-1}\choose {n_2-m-k}}}\,\\ 
		&= \, \sum_{\ell=0}^{n_1/2}\sum_{k=0}^{n_2-m}{{{n_1/2}\choose {\ell}}} (ar^2)^{\ell}  r^k 
		\frac{{{{\ell+k-1}\choose {k}}}\,{{{n_1/2-\ell + n_2-m-k-1}\choose {n_2-m-k}}}}{{{{n_1/2+n_2-m-1}\choose {n_2-m}}}}.\nn
	\end{align*}
%Furthermore, recalling that
%\begin{align}
%n_1 &= (1-p) N \\ \nn
%n_2 &= p N \\ \nn
%\end{align}
%we guess that the sums in \eqref{2sommatorie} are carried by $l^*$ and $k^*$ such that $l^*/N \rightarrow s^*$ and $k^*/N \rightarrow t^*$. Then we take 
%\begin{align}
%l &= s N \\ \nn
%k &= t N \\ \nn
%\end{align}
%and we rewrite
%\be\label{def_Bmst}
%B_{s,t}^{\sss (N)}(n_2-m) \, = \, {{{((1-p)/2) N}\choose {sN}}} (ar^2)^{sN}  r^{tN} \frac{{{{(s+t)N-1}\choose {tN}}}\,{{{((1-p)/2-s+ p-t)N-m-1}\choose {(p-t)N-m}}}}{{{{((1-p)/2+p)N-m-1}\choose {pN-m}}}}.
%\ee
\end{proof}

%..............
%We call
%\be
%B_{l,k}^{\sss (N)} \, = \, {{{n_1/2}\choose {l}}} (ar^2)^{l}  r^k \frac{{{{l+k-1}\choose {k}}}\,{{{n_1/2-l + n_2-k-1}\choose {n_2-k}}}}{{{{n_1/2+n_2-1}\choose {n_2}}}}
%\ee
%and taking
%\begin{align}
%n_1 &= (1-p) N \\ \nn
%n_2 &= p N \\ \nn
%%l &= s N \\ \nn
%%k &= t N \\ \nn
%\end{align}
%we guess that the sums in \eqref{2sommatorie} are carried by $l^*$ and $k^*$ such that $l^*/N \rightarrow s^*$ and $k^*/N \rightarrow t^*$. Then we take 
%\begin{align}
%l &= s N \\ \nn
%k &= t N \\ \nn
%\end{align}
%and rewriting
%\be
%B_{s,t}^{\sss (N)} \, = \, {{{((1-p)/2) N}\choose {sN}}} (ar^2)^{sN}  r^{tN} \frac{{{{(s+t)N-1}\choose {tN}}}\,{{{((1-p)/2-s+ p-t)N-1}\choose {(p-t)N}}}}{{{{((1-p)/2+p)N-1}\choose {pN}}}}.
%\ee
%we want to find a function $H(s,t)$ such that:

\paragraph{Asymptotics by Stirling's formula.}
We continue the analysis by investigating the asymptotics of $B_{\ell,k}^{\sss (N)}(n_2)$ in \eqref{def_Bmlk} when $\ell,k$ and $n_2$ are of the same asymptotic order. To alleviate the notation we write $B_{a,b}^{\sss (N)}(n_2):=B_{\lfloor a\rfloor ,\lfloor b\rfloor}^{\sss (N)}(n_2)$  when $a,b$ are not necessarily integers.
\begin{lemma}[Asymptotics of $B_{\ell,k}^{\sss (N)}(n_2)$]
\label{lemma_asymptB}
Let $D_{p}=[0,(1-p)/2]\times [0,p]$. {For external fields $B \neq 0$}, there exists a function $H(s,t)$ continuous  in $D_{p}$ and smooth  in $D_{p}^{\circ}$ (the interior of $D_{p}$) and a function $C(s,t)$ smooth  in $D_{p}^{\circ}$, such that
	\be\label{asymptB}
	B_{sN,tN}^{\sss (N)}(n_2) \, = \, \frac{C(s,t)}{N}\,\exp \left\{N H(s,t)\right\}(1+o(1)),\quad \mbox{as}\quad N\to \infty,
	\ee
%for $\ell=1,\ldots, n_{1}/2$ and $k=1,\ldots, n_{2}$. 
Moreover,  $H(s,t)$ is strictly concave on its domain $D_{p}$ and its (unique) maximum point $(s^{*}, t^{*})$ lies in the interior $D_{p}^{\circ}$. In $D_{p}^{\circ}$, the functions are defined as follows:
	\begin{align}
		\label{defH}\nn
		H(s,t) \, = \, &(1-p)\log\left(\frac{1-p}{2}\right) - 2s\log(s) 
		- 2\left(\frac{1-p}{2} -s\right)\log\left(\frac{1-p}{2} -s\right)  \\ \nn
		&+ s\log(ar^2) + t\log(r) + (s+t)\log(s+t) - t \log(t)   \\ \nn
		& + \left(\frac{1+p}{2}-s-t\right)\log\left(\frac{1+p}{2}-s-t\right) \\ 
		&-(p-t)\log(p-t) - \left(\frac{1+p}{2}\right) \log\left(\frac{1+p}{2}\right) + p \log(p),
	\end{align}
and
	\be\label{defC}
	C(s,t) \, = \, \frac{1}{2 \pi} \frac{\frac{1-p}{2}}{s \left(\frac{1-p}{2}-s\right)} \frac{ 	\sqrt{(\frac{1+p}{2}-s-t)(s+t)p}}{\sqrt{(\frac{1+p}{2})(p-t)t}}.
	\ee
Finally, uniformly in $(s,t)\in D_{p}$, 
	\eqn{
	\label{unif-bd-BN}
	B_{sN,tN}^{\sss (N)}(n_2)\leq CN^{1/2}\exp \left\{N H(s,t)\right\}.
	}
%\RvdH{Is exponent of $N$ correct?}
\end{lemma}
%\begin{lemma}
%Taking $m=0$, for $N$ large we have
%\be
%B_{s,t}^{\sss (N)}(n_2) \, \approx \, \frac{C(s,t)}{N}\,\exp \left\{N \left[H(s,t)\right]\right\},
%\ee
%where
%\begin{align}\nn
%H(s,t) \, =& \, (1-p)\log\left(\frac{1-p}{2}\right) - 2s\log(s) - 2\left(\frac{1-p}{2} -s\right)\log\left(\frac{1-p}{2} -s\right) + s\log(ar^2) \\ \nn
%&\; + t\log(r) + (s+t)\log(s+t) - t \log(t) + \left(\frac{1+p}{2}-s-t\right)\log\left(\frac{1+p}{2}-s-t\right)   \\ 
%&\,   -(p-t)\log(p-t) - \left(\frac{1+p}{2}\right) \log\left(\frac{1+p}{2}\right) + p \log(p),
%\end{align}
%and
%\be
%C(s,t) \, = \, \frac{1}{2 \pi} \frac{\frac{1-p}{2}}{s \left(\frac{1-p}{2}-s\right)} \frac{ \sqrt{(\frac{1+p}{2}-s-t)(s+t)p}}{\sqrt{(\frac{1+p}{2})(p-t)t}}.
%\ee
%Moreover, $H(s,t)$ is a strictly concave function on its domain.
%\end{lemma}
\begin{proof}
Using Stirling's approximation in the form $n! =  \e^{-n}n^n \sqrt{2 \pi n}\, (1+o(1))$ for $n$ large, taking $a,b \in \mathbb{N}$, we can rewrite the binomial coefficients as
	\begin{equation*}
	{{b\,n\choose {a\,n}}} \,= \, \e^{\left[b \log{b}-a \log{a}-(b-a) \log{(b-a)}\right]n} \cdot \frac{\sqrt{b}\, 	(1+o(1))}{\sqrt{a}\sqrt{b-a}\sqrt{2 \pi n}}.
	\end{equation*}
Plugging  the previous formula into  \eqref{def_Bmlk},
%Then,  \eqref{def_Bmlk} reads,
%
%
% for $\lfloor sN\rfloor,\lfloor tN\rfloor \ne 0$,
%\RvdH{Where do we use this?}
%
%
%	\begin{align}\nn
%		B_{sN,tN}^{\sss (N)}(n_2)  =\exp
%		&\Big\{N\Big[ (1-p)\log\left(\frac{1-p}{2}\right) 
%		- 2s\log(s) - 2\left(\frac{1-p}{2} -s\right)\log\left(\frac{1-p}{2} -s\right) + s\log(ar^2)\\ \nn
%		&\,+ t\log(r) + (s+t)\log(s+t) - t \log(t) + \left(\frac{1+p}{2}-s-t\right)\log\left(\frac{1+p}{2}-s-t\right) \\ \nn
%		&\,-(p-t)\log(p-t) - \left(\frac{1+p}{2}\right) \log\left(\frac{1+p}{2}\right) + p \log(p) \Big]\Big\} \\ 
%		&\, \times\frac{1}{2 \pi N} \frac{\frac{1-p}{2}}{s \left(\frac{1-p}{2}-s\right)} \frac{ 	\sqrt{(\frac{1+p}{2}-s-t)(s+t)p}}{\sqrt{(\frac{1+p}{2})(p-t)t}}\, (1+o(1)).
%	\end{align}
%
%
%and then 
%\be
%B_{s,t}^{\sss (N)} \, = \, \exp \left\{N \left[H(s,t)\right]\right\}.
%\ee
%This proves \eqref{asymptB}.  
then  \eqref{asymptB} follows.
By inspection, $H(s,t)$ and $C(s,t)$ are smooth functions in $D_{p}^{\circ}$ {for $B \neq 0$}. The function $H(s,t)$ can be further extended by continuity to the boundary $\partial D_{p}$ of  $D_{p}$, while $C(s,t)$ cannot be defined in $D_{p}\setminus D_{p}^{\circ}$ since it is unbounded there. In order to prove concavity of $H(s,t)$, we check that its Hessian matrix $\mathbf{Q}(s,t)$ is negative definite on each point of $D_{p}^{\circ}$. For this, we compute the Hessian $\mathbf{Q}(s,t)$ as 
	\begin{small}
	\eqn{\nn
	\left(\begin{array}{cc}
	\displaystyle \frac{1}{\frac{1-p}{2}-s+p-t} + \frac{1}{s+t} - \frac{2}{\frac{1-p}{2}-s} - \frac{2}{s} &  \displaystyle  	\frac{1}{\frac{1-p}{2}-s+p-t} + \frac{1}{s+t}\\
	\displaystyle\frac{1}{\frac{1-p}{2}-s+p-t} + \frac{1}{s+t} &  \displaystyle \frac{1}{\frac{1-p}{2}-s+p-t} + \frac{1}{s+t} - 	\frac{1}{p-t} - \frac{1}{t}\\
	\end{array}\right).
	}
	\end{small}
The eigenvalues $\mu^+$ and $\mu^-$ of  $\mathbf{Q}(s,t)$ are
\begin{small}
	\begin{align*}
	\mu^{\pm} &=\frac{1}{2}\Bigg[\frac{2}{\frac{1+p}{2}-s-t} + \frac{2}{s+t} - 	\frac{2}{\frac{1-p}{2}-s} - \frac{2}{s}- \frac{1}{p-t} - \frac{1}{t}\\
	&\quad \pm \sqrt{\left( 	\frac{2}{\frac{1-p}{2}-s} + \frac{2}{s}- \frac{1}{p-t} - \frac{1}{t}\right)^2 + 4\left(\frac{1}{\frac{1+p}{2}-s-t} + \frac{1}{s+t}\right)^2}\Bigg].
	\end{align*}
%\be
%\lambda^{-} \,=\,\frac{1}{2}\left[\frac{2}{(\frac{1-p}{2}-s+p-t)} + \frac{2}{(s+t)} - \frac{2}{(\frac{1-p}{2}-s)} - \frac{2}{s}- \frac{1}{(p-t)} - \frac{1}{t} - \sqrt{\left( \frac{2}{(\frac{1-p}{2}-s)} + \frac{2}{s}- \frac{1}{(p-t)} - \frac{1}{t}\right)^2+4\left(\frac{1}{(\frac{1-p}{2}-s+p-t)} + \frac{1}{(s+t)}\right)^2}\right].
%\ee
\end{small}
\noindent
We can easily see that $\mu^{-} < 0$, and in order to show that also $\mu^{+}$ is negative, we observe that the determinant of the Hessian matrix is positive.
%\begin{scriptsize}
%\be
%det \,=\left(\frac{2}{(\frac{1-p}{2}-s)} + \frac{2}{s}\right)\left(\frac{1}{(p-t)} + \frac{1}{t}\right) - \left(\frac{2}{(\frac{1-p}{2}-s)} + \frac{2}{s}\right)\left(\frac{1}{\left(\frac{1-p}{2}-s+p-t\right)} + \frac{1}{\left(s+t\right)}\right)-\left(\frac{1}{(p-t)} + \frac{1}{t}\right)\left(\frac{1}{\left(\frac{1-p}{2}-s+p-t\right)} + \frac{1}{\left(s+t\right)}\right) \, >\, 0,
%\ee
%\end{scriptsize}
Therefore, $H(s,t)$ is strictly concave in  $D_{p}^{\circ}$ and, by continuity, concave in $D_{p}$.
%($det=0$ only for $(s=0,t=0)$ and $(s=\frac{1-p}{2}, t=p )$).
%\vspace{0.3cm}
\noindent
Concavity implies that $(s,t)\mapsto H(s,t)$ has a unique global maximum in $D_{p}$, the uniqueness follows by strict concavity and the fact that the maximizer is not on the boundary. In order to find $(s^*, t^*) \, := \, \mathop{\rm argmax}\limits_{(s,t)\in D_{p}}H(s,t)$,  and to prove that it lies in $D_{p}^{\circ}$, we calculate
	\begin{align}\label{H_diff_su _t_e_s}\nn
		\frac{\partial H(s,t)}{\partial s} \, &= \, 2\log\left(\frac{1-p}{2} -s\right) 
		- \log\left(\frac{1+p}{2}-s-t\right) + \log(s+t) - 2\log(s)\\ 
		& \quad \,+ \log(ar^2), \\ 
		\frac{\partial H(s,t)}{\partial t} \, &= \, \log(s+t) - \log(t) - \log\left(\frac{1+p}{2}-s-t\right) + \log(p-t) + \log(r),
	\end{align}
so that $(s^*, t^*)$ is a solution of the system
%To obtain the critical points, we compute the solutions of the system
	\be\label{syst_crit_point}
	%\left\{\begin{array}{ll} \displaystyle \frac{\partial H(s,t)}{\partial s} \, = \, 0\\ 
	%\displaystyle \frac{\partial H(s,t)}{\partial t} \, = \, 0
	%\end{array}\right. \quad \quad \Longrightarrow \quad \quad
	\left\{\begin{array}{ll} \frac{\left(\frac{1-p}{2} -s\right)^2}{s^2} \frac{(s+t)}{\left(\frac{1+p}{2}-s-t\right)} \, = \, 	\frac{1}{ar^2}, \\
	\frac{(p-t)}{t}\frac{(s+t)}{\left(\frac{1+p}{2}-s-t\right)} \, = \, \frac{1}{r}.
	\end{array}\right.
	\ee
{Since $B\neq 0$, and then both $ar^2$ and $r$ are finite and larger than zero, it easy to see from (\ref{syst_crit_point}) that the maximum point cannot be attained on the boundary.}
%By observing (\ref{syst_crit_point}), because $B\neq 0$ and then both $ar^2$ and $r$ are finite and larger than zero, it easy to see that the maximum point cannot be attained on the boundary.
%\be\label{syst_crit_point_2}
%\left\{\begin{array}{ll} s^3(ar^2 + 1) - s^2(1+p)ar^2 + s^2 t ar^2 - s^2\big(\frac{1+p}{2}\big) + s^2t +s\big(\frac{1+p}{2}\big)^2ar^2 -s(1+p)tar^2 + \big(\frac{1+p}{2}\big)^2 tar^2   =  0\\ 
%t^2 -t^2r + tpr - tsr +ts -t\big(\frac{1+p}{2}\big) +spr  =  0.
%\end{array}\right.  
%\ee

The proof of \eqref{unif-bd-BN} follows similarly, now using that $\e^{-n}n^n \sqrt{2 \pi n}\leq n!\leq $\\$ \e^{-n}n^n \sqrt{2 \pi n}(1+\frac{1}{12n})$ for every $n\geq 1.$ The power of $N$ is needed to make the estimate uniform, e.g., by bounding $s((1-p)/2-s)\geq c/N$ uniformly for $s\geq 1/N$.
\end{proof}

\paragraph{Asymptotics by Laplace method.}
In the next Lemma we compute the asymptotic behavior of \eqref{gen_fun_numb_lin} by using the discrete analogue
of Laplace method.
%In the following proposition, we compute the asymptotic behavior of \eqref{gen_fun_numb_lin} both for $m=0$ and for $m>0$.
%A heuristic argument for the asymptotic can be obtained by applying the Laplace method to the function $N^{-1}C(s,t)\exp\{NH(s,t)\}$ in a domain $S$ containing the maximum point $(s^{*},t^{*})$ and contained in $D_{p}^{\circ}$. Laplace's method gives (and noting that a factor $N^2$ arises from changing the sum into an integral)
%	\be
%	\iint_{S} NC(s,t) \exp\{NH(s,t)\} ds\,dt = 2\pi C(s^{*},t^{*}) (\det \mathbf{Q}(s^{*},t^{*}))^{-\frac 1 2} 	\exp\{NH(s^{*},t^{*})\} (1+o(1)).
%	\ee
%The powers of $N$ cancel, since two factors of $N$ arise from rescaling the sums over $\ell$ and $k$ (which are both of the order $N$ with $s=\ell/N$ and $t=k/N$), and one factor of $1/N$ arises due to the limiting two-dimensional Gaussian integral. As the next proposition shows, this is indeed the correct result:
%
\begin{lemma}[Asymptotics of {$Q_{\sss N}[\prod_{l=2}^{\infty} c_l(B)^{N_l}\mid M_{\sss N}=m]$}]
\label{asy_M_N_equal_zero}
For every $m\geq0$ fixed,
	\begin{align}
	\label{asymptotics_m_>0}
	&Q_{\sss N}[\prod_{l=2}^{\infty} c_l(B)^{N_l}\mid M_{\sss N}=m] \\
	& = \, 2\pi\, C(s^{*},t^{*})\, (\det \mathbf{Q}(s^{*},t^{*}))^{-\frac 1 2} \left(b^*\right)^{m} 
	\exp\{NH(s^{*},t^{*})\} (1+o(1)),\nn
	\end{align}
where $(s,t)\mapsto H(s,t)$ and $(s,t)\mapsto C(s,t)$ are defined in Lemma  \ref{lemma_asymptB}, 
$(s^{*},t^{*})$ is the maximum point of $H(s,t)$, $ \mathbf{Q}(s,t)$ is the Hessian matrix of $H$ and
\begin{equation*}
b^* = \left (\frac{1+p}{2p}\right)\left (\frac{p-t^{*}}{\frac{1+p}{2}-s^{*}-t^{*}}\right).
\end{equation*}
\end{lemma}

\proof 
We start by proving (\ref{asymptotics_m_>0}) for $m=0$. Due to Lemma \ref{lemma_asymptB},
we may estimate the asymptotic behavior of the double sum 
	\be\label{sumB1}
	K(N):= \sum_{\ell=0}^{n_1/2}\sum_{k=0}^{n_2} B_{\ell,k}^{\sss (N)}(n_2),
	\ee
by making use of the function $f_{\sss N}(s,t)=N^{-1}C(s,t)\exp\{NH(s,t)\}$ 
%defined in $D_{p}$ where we use 
that appeared in \eqref{asymptB}. 
The correspondence between the two sets of variables $(\ell,k)$ and  $(s,t)$ is given by the simple transformation $s=\ell/N$ and $t=k/N$.  We denote this transformation by $T_{N}$. 

%We denote this transformation by $T_{\sss N}$.

Let us define $\ell^{*}_{\sss N}:= \lfloor s^{*} N\rfloor$, $k^{*}_{\sss N}:= \lfloor t^{*} N\rfloor$  and  introduce $0<\delta < \min\{p, (1-p)/2\}$. The precise value of $\delta$ will be chosen later on. 
We partition the domain of the summation appearing in the sum of $B_{\ell,k}^{\sss (N)}(n_2)$
	\begin{equation*}
	\Lambda_{\sss N}=\{(\ell,k)\colon \ell=1,\dots, n_{1}/2, k=1,\ldots n_{2} \}, 
	\end{equation*}
into two subsets
	\begin{equation*}
	U_{\delta,{\sss N}}=\{(\ell,k)\in \Lambda_{\sss N} \colon |\ell - \ell^{*}_{\sss N}|\le \delta N+1,  
	|k - k^{*}_{\sss N}|\le \delta N+1\},\quad  U_{\delta,{\sss N}}^{c}=\Lambda_{\sss N} 
	\backslash U_{\delta,{\sss N}}.
	\end{equation*}
The set $U_{\delta,{\sss N}}$ is to be considered as a neighborhood of $( \ell^{*}_{\sss N},  k^{*}_{\sss N})$, the ``maximum'' point of $B_{\ell,k}^{\sss (N)}(n_2)$.
%\footnote{ $( \ell^{*}_{\sss N},  k^{*}_{\sss N})$ are approximations to the maximum point in the sense that for any $(\ell, k)$ fixed $B_{\ell,k}^{\sss (N)}(n_2)/%B_{\ell^{*}_{\sss N},k^{*}_{\sss N}}^{\sss (N)}(n_2)\to 0$ as $N\to\infty$, because of \eqref{asymptB}. In the same way, one can see that for any sequence $
%(\ell(N),k(N))\in \Lambda_{\sss N}$ such that $(\ell(N)N,k(N)/N)$ does not converge to $(s^{*},t^{*})$, it is  $B_{\ell(N),k(N)}^{\sss (N)}(n_2)/B_{\ell^{*}_{\sss N},k^{*}%_{\sss N}}^{\sss (N)}(n_2)\to 0$.}
%
%
We observe that $T_{\sss N}( U_{\delta,{\sss N}})$ is contained in the neighborhood of $(s^{*},t^{*})$ in $D_{p}$, i.e.,
	\begin{equation*}
	W_{\delta+\frac 1 N}=\{ (s,t)\in D_{p}\colon |s-s^{*}|\le \delta+\frac 1 N,\,  |t-t^{*}|\le \delta+\frac 1 N \},
	\end{equation*}
while $T_{\sss N}( U_{\delta,{\sss N}}^{c})$ is contained in its complement $W_{\delta+\frac 1 N}^{c}:=D_{p}\backslash W_{\delta+\frac 1 N}$. 
We rewrite \eqref{sumB1} as $K(N) = K_{1}(\delta,N) + K_{2}(\delta,N)$
%	\be\label{II1I2}
%	I(N):=\sum_{\ell=0}^{n_1/2}\sum_{k=0}^{n_2} B_{\ell,k}^{\sss (N)}(n_2)
%	= \sum_{(\ell,k)\in U_{\delta,{\sss N}}}B_{\ell,k}^{\sss (N)}(n_2) 
%	+\sum_{(\ell,k)\in U_{\delta,{\sss N}}^{c}} B_{\ell,k}^{\sss (N)}(n_2),
%	\ee
%and denote 
where
	\be \label{defI1I2}
	K_{1}(\delta,N):= \sum_{(\ell,k)\in U_{\delta,{\sss N}}}B_{\ell,k}^{\sss (N)}(n_2), 
	\quad \quad K_{2}(\delta,N):= \sum_{(\ell,k)\in U_{\delta,{\sss N}}^{c}}B_{\ell,k}^{\sss (N)}(n_2).
	\ee
We aim to prove that the asymptotic behavior of $K(N)$ is given by $K_{1}(\delta,N)$, while $ K_{2}(\delta,N)$ gives a sub-dominant contribution. We start by proving the latter statement.\\

\paragraph{Bound on $\mathbf{K_{2}(\delta,N)}$.} 
Making use of \eqref{unif-bd-BN}, we upper bound
	\be\label{I2}
 	K_{2}(\delta,N)\leq CN^{1/2} \sum_{(\ell,k)\in U_{\delta,{\sss N}}^{c}}\exp \left\{NH (\ell/N,k/N) \right\}.
	\ee
Defining 
	\begin{equation*}
	M(\delta):=\sup_{\substack{|s-s^{*}|>\delta \\ |t-t^{*}|>\delta}} H(s,t)
	\ge \sup_{(s,t)\in U_{\delta,{\sss N}}^{c}} H(s,t)\;,
	\end{equation*}
since the values $(\ell/N,k/N)$ in \eqref{I2} belong to $W_{\delta+\frac 1 N}^{c}$ %$U_{\delta,{\sss N}}^{c}$, 
we can  bound $H(\ell/N,k/N)\leq M(\delta)$. We conclude that
	\be\label{I2a}
 	K_{2}(\delta,N)\le CN^{1/2} \exp\{NM(\delta)\}|U_{\delta,{\sss N}}^{c}|\leq CN^{5/2} \exp\{NM(\delta)\},
	\ee
which, together with $M(\delta)<H(s^{*},t^{*})$, implies that 
	\be\label{I2subdominant}
	\exp\{-NH(s^{*},t^{*})\} \, K_2(\delta, N) \, \rightarrow 0, \quad \; \mbox{as} \; N \rightarrow \infty.
	\ee
Let us remark that, besides the condition  $0<\delta < \min\{p, (1-p)/2\}$ (which guarantees that $(\ell,k)$ and $(s,t)$ are contained in the domains of $B_{\ell,k}^{\sss (N)}(n_2)$ and $f_{\sss N}(s,t)$), in the previous argument no further condition has  been imposed on $\delta$.\\

\paragraph{Asymptotics of $\mathbf{ K_{1}(\delta,N)}$.} 
Here we consider the sum $K_{1}(\delta,N)$ defined in \eqref{defI1I2}. 
%As already observed, the indices in this sum correspond to $(s,t)$ points belonging to the neighborhood $U_{\delta, {\sss N}}$ of $(s^{*},t^{*})$.
Choose $\varepsilon>0$ arbitrary and small. By continuity of $C(s,t)$, we can choose $\delta>0$ small enough so that, for large $N$,
	\begin{equation*}
	C(s^{*},t^{*})-\varepsilon\le C(s,t)\le C(s^{*},t^{*})+\varepsilon,
	\quad \mbox{for all}\quad (s,t)\in W_{\delta+\frac 1 N} % U_{\delta,{\sss N}}.
	\end{equation*}
Then using  \eqref{asymptB}, we obtain
	\begin{align}\label{boundI1}
		K_{1}(\delta,N)\le
		\frac{C(s^{*},t^{*})+\varepsilon}{N}  
		\sum_{(\ell,k)\in U_{\delta,N}}\exp \left\{NH\left (\frac{\ell}{N},\frac{k}{N}\right)\right\}(1+o(1)),
	\end{align}
and a similar lower bound with $C(s^{*},t^{*})+\varepsilon$ replaced by $C(s^{*},t^{*})-\varepsilon$.
Recalling that $\mathbf{Q}(s,t)$ is the Hessian matrix of $H(s,t)$, by Taylor expanding up to second order and using that $(s^*,t^*)$ is the maximum, we can write
	\begin{equation*}
	H(s,t)-H(s^{*},t^{*})\leq \frac 1 2 {\mathbf x}\cdot \mathbf{Q}(s^{*},t^{*}) {\mathbf x} +c\delta \|{\mathbf x}\|^2/2, 
	\quad \text{for all }(s,t)\in  W_{\delta+\frac 1 N} %U_{\delta,{\sss N}},
	\end{equation*}
where ${\mathbf x}=(s-s^{*},t-t^{*})$, and a similar lower bound with $c\delta$ replaced by $-c\delta$. 

%Relying on the general properties of the quadratic forms, we know that for any $(s^{'},t^{'})\in  W_{\delta+\frac 1 N} $ there exist $0<\alpha(s^{'},t^{'})<\beta(s^{'},t^{'})$ (the negative of the eigenvalues) such that
%	\be
%	-\beta(s^{'},t^{'}) ||{\mathbf x}||^{2} \le \frac 1 2 {\mathbf x}\cdot \mathbf{Q}(s^{'},t^{'}) {\mathbf x} \le 	-\alpha(s^{'},t^{'}) ||{\mathbf x}||^{2},\quad {\mathbf x}\in {\mathbb R}^{2}.
%	\ee
%By smoothness of $\mathbf{Q}(s^{'},t^{'})$ and compactness  we can further restrict the choice of $\delta$ such that, denoting  $\alpha^{*}=\alpha(s^{*},t^{*})$ and $\beta^{*}=\beta(s^{*},t^{*})$ we have that
%\be
%0<\alpha^{*}-\epsilon \le \alpha(s^{'},t^{'}) \le \alpha^{*}+\epsilon,\quad \mbox{and}\quad 0<\beta^{*}-\epsilon \le \beta(s^{'},t^{'}) \le \beta^{*}+\epsilon
%\ee
%and thus for $(s,t)\in  W_{\delta+\frac 1 N}$
%\be
%-(\beta^{*}+\epsilon)[(s-s^{*})^{2}+(t-t^{*})^{2}]\le H(s,t)-H(s^{*},t^{*})\le -(\alpha^{*}-\epsilon)[(s-s^{*})^{2}+(t-t^{*})^{2}]\, .
%\ee
By multiplying \eqref{boundI1} by $\exp[-NH(s^{*},t^{*})]$ and applying the previous inequality, we obtain
%\be\label{lowerbI1}
%\frac{(C(s^{*},t^{*})-\epsilon)}{N} \hspace{-0.4cm} \sum_{(\ell,k)\in U_{\delta,N}}\hspace{-0.4cm} \exp \left\{ -N\,(\beta^{*}+\epsilon)[(\frac{\ell}{N}-s^{*})^{2}+(\frac{k}{N}-t^{*})^{2}] \right\}(1+o(1))\le  
%\exp[-NH(s^{*},t^{*})] I_{1}(\delta,N)
%\ee
%and
	\begin{align}\label{upperbI1}\nn
	&\exp[-NH(s^{*},t^{*})] K_{1}(\delta,N)\\
	&\le \frac{C(s^{*},t^{*})+\varepsilon}{N}
	\sum_{(\ell,k)\in U_{\delta,N}}%\hspace{-0.4cm} 
	\exp \left\{\frac N2 {\mathbf x}^T\cdot \mathbf{Q}(s^{*},t^{*}) {\mathbf x} 
	+c\delta N\|{\mathbf x}\|^2/2\right\},
	\end{align}
(where ${\mathbf x}$ is computed with $s=\ell/N$ and $t=k/N$) and a similar lower bound with $C(s^{*},t^{*})+\varepsilon$ replaced with $C(s^{*},t^{*})-\varepsilon$ and 
$+c\delta N\|{\mathbf x}\|^2$ replaced with $-c\delta N\|{\mathbf x}\|^2$. The last step is bounding the sum
	\begin{equation*}
	\tilde{K}_{1}(\delta,N):=\hspace {-0.4cm} \sum_{(\ell,k)\in U_{\delta,N}}\hspace{-0.4cm} 
	\exp \left\{\frac N2 {\mathbf x}^T\cdot (\mathbf{Q}(s^{*},t^{*})\pm c\delta\mathbf{I}) {\mathbf x} 
	\right\},
	\end{equation*}
Now we can substitute the finite sum in the previous display with the infinite one, since the difference is exponentially small \cite{dB}. It is known that (as can be seen by extending \cite[(3.9.4)]{dB} to two-dimensional sums) that 
	\begin{equation*}
	\sum_{\mathbf{j}\in \mathbb{Z}^2}\e^{-\mathbf{j}^T \mathbf{A} \mathbf{j}/(2N)} 
	= \frac{2\pi N}{\det(\mathbf{A})^{1/2}}(1+o(1)),
	\end{equation*}
Therefore,
	\begin{equation*}
	\tilde{K}_{1}(\delta,N) = \frac{2\pi}{\det(\mathbf{Q}(s^{*},t^{*})\pm c\delta \mathbf{I})^{1/2}}\,N(1+o(1)).
	\end{equation*}
From the previous equation, recalling \eqref{upperbI1}, we obtain
	\begin{equation*}
	\exp[-NH(s^{*},t^{*})] K_{1}(\delta,N)\le 
	\frac{2\pi (C(s^{*},t^{*})+\varepsilon)}{\det(\mathbf{Q}(s^{*},t^{*})-c\delta \mathbf{I})^{1/2}}\,(1+o(1)),
	\end{equation*}
and a similar lower bound with $C(s^{*},t^{*})+\varepsilon$ replaced with $C(s^{*},t^{*})-\varepsilon$
and $\mathbf{Q}(s^{*},t^{*})-c\delta \mathbf{I}$ with $\mathbf{Q}(s^{*},t^{*})+c\delta \mathbf{I}$.
%Thus, from \eqref{II1I2} and the previous inequality, we get the following estimate:
%\begin{align}\label{boundI}
%&\frac{\pi (C(s^{*},t^{*})-\epsilon)}{\sqrt{(\alpha^{*}-\epsilon)(\beta^{*}+\epsilon) }}\,(1+o(1))+\exp\{ -NH(s^{*},t^{*})\} {I}_{2}(\delta,N)  \le \exp[-NH(s^{*},t^{*})] I(N)\le  \nn\\
% &\frac{\pi (C(s^{*},t^{*})+\epsilon)}{\sqrt{(\alpha^{*}-\epsilon)(\beta^{*}+\epsilon) }}\,(1+o(1)) +\exp\{ -NH(s^{*},t^{*})\} {I}_{2}(\delta,N).
%\end{align}
%Now taking the limit as $N\to \infty$ and recalling \eqref{I2subdominant}, we have:
%\begin{align}\label{boundIbis}
%&\frac{\pi (C(s^{*},t^{*})-\epsilon)}{\sqrt{(\alpha^{*}-\epsilon)(\beta^{*}+\epsilon) }}\  \le \liminf_{N\to \infty}\exp[-NH(s^{*},t^{*})] I(N)\le  \limsup_{N\to \infty}\exp[-NH(s^{*},t^{*})] I(N)\le   \nn\\
%&\frac{\pi (C(s^{*},t^{*})+\epsilon)}{\sqrt{(\alpha^{*}-\epsilon)(\beta^{*}+\epsilon) }}.
%\end{align}
Since $\varepsilon$ is arbitrary, the previous inequality implies
	\begin{equation*}
	\lim_{N\to \infty}\exp[-NH(s^{*},t^{*})] K(N)=\frac{2\pi\, C(s^{*},t^{*})}{\det(\mathbf{Q}(s^{*},t^{*}))^{1/2}},
	\end{equation*}
which %, recalling that $\alpha^{*}$ and $\beta^{*}$ are the eigenvalues of the Hessian $\mathbf{Q}(s^{*},t^{*})$, 
proves the claim. 
%which, being $\epsilon$ arbitrary, implies that 
%\be
% I_{1}(\delta,N) = {\pi\, C(s^{*},t^{*})}\, (\det \mathbf{Q}(s^{*},t^{*}))^{-\frac 1 2}\,\exp[NH(s^{*},t^{*})]  (1+o(1)),
%\ee
%since $\alpha^{*}$ and $\beta^{*}$ are the eignevalues of the Hessian $\mathbf{Q}(s^{*},t^{*})$.
%\Gib{For the last bound I see two possible ways. 1) Considering the previous sum as a Riemann sum. Difficulty: integrand function depends on the parameter $N$. 2) By exploiting the invariance of $||x||^{2}$ under rotations one realizes that the sum is essentially one-dimensional (change of variables) Then the sum is reduced to  
%$\sum_{i}\exp{(-NAx_{i}^{2})}$, whose asymptotic behavior is known \cite{dB}. } 

%\vspace{0.4cm}
Next, we want to generalize the previous result by computing the asymptotic of  \eqref{gen_fun_numb_lin} in the case $m\ne 0$. We start by rewriting  \eqref{gen_fun_numb_lin}  in the following fashion:
	\begin{equation*}
	Q_{\sss N}[\prod_{l=2}^{\infty} c_l(B)^{N_l}\mid M_{\sss N}=m]
	=\sum_{\ell=0}^{n_1/2}\sum_{k=0}^{n_2-m}\,	G^{\sss (N)}_{\ell,k}(m;n_{2}) 	B_{\ell,k}^{\sss (N)}(n_2)\, ,
	\end{equation*}
where
	\begin{equation*}
	G^{\sss (N)}_{\ell,k}(m;n_{2}):=\frac{B_{\ell,k}^{\sss (N)}(n_2-m)}{B_{\ell,k}^{\sss (N)}(n_2)} 
	\,= \, \frac{\prod_{j=1}^{m} \left ( \frac{n_{1}}{2}+n_{2}-j\right)\prod_{j=0}^{m-1}(n_{2}-k-j)}
 	{\prod_{j=0}^{m-1} \left (n_{2}-j\right)\prod_{j=1}^{m} 
 	\left ( \frac{n_{1}}{2}+n_{2}-\ell-k-j\right)},
	\end{equation*}
for $m=0,1,\ldots, n_{2}$. By defining the function $F(s,t;m)$ on $D_{p}^{\circ}$ given by
	\begin{equation*}
	F(s,t;m)=\left (\frac{1+p}{2p}\right)^{m}\left (\frac{p-t}{\frac{1+p}{2}-s-t} \right)^{m},
	\end{equation*}
we obtain that 
	\begin{equation*}
	G^{\sss (N)}_{\ell,k}(m;n_{2})=F\left(\frac{\ell}{N},\frac{k}{N};m\right)(1+o(1))
	\end{equation*}
as $N\to \infty$. Then the proof is obtained from that for $m=0$ by replacing $C(s,t)$ by $C(s,t)F(s,t)$.
\qed

\begin{remark}[Bound on $b^*$]\label{bstar_less_bc}
Since $(s^{*},t^{*})\in D_{p}^{\circ} $,
	\be\nn
	\frac{p-t^{*}}{\frac{1+p}{2}-s^{*}-t^{*}}=\frac{p-t^{*}}{(\frac{1-p}{2}-s^{*})+(p-t^{*})}<1,
	\qquad
	\text{so that}
	\qquad
	b^*<\frac{1+p}{2p}.
	\ee
This will allow us to use in the following the moment generating function $\q[(b^*)^M]$ defined in \eqref{mom_gen_func_b}.
\end{remark}
\medskip

\paragraph{Boundary contribution.} Lemma \ref{asy_M_N_equal_zero} proves, for any {\em fixed} $0\le m < \infty$,  the asymptotic exponential growth of $Q_{\sss N}[\prod_{l=2}^{\infty} c_l(B)^{N_l}\mid M_{\sss N}=m]$ as $N\to\infty$.  However in formula \eqref{prob_tot_prop4.1} we need to sum over a range of values of $m$ that increases with the 
volume $N$. In order to overcome this problem, in the proof of Proposition \ref{prop-exp-part-funct} we  introduce a cut-off in the sum over $m$ 
(and then send the cut-off to infinity at the end). In doing so we need to exclude the contribution arising from $B_{\ell,k}^{(N)}$ for $\ell$ close to the 
boundary $n_1/2$. This is achieved in the following Lemma.
\begin{lemma}[Boundary contribution]\label{contribution_l_magg}
For every  $\varepsilon >0$ sufficiently small, as $N\to\infty$
	\begin{equation*}
	\sum_{m=0}^{n_2}\expec_m[\bar{Z}_{m}^{\sss(2)}(B)]
	\sum_{\ell>(1-\varepsilon)\frac{n_1}{2}}^{n_1/2}\sum_{k=0}^{n_2-m}\,B_{\ell,k}^{\sss (N)}(n_2-m) \q(M_{\sss N}=m) 
	\, =\, o(\e^{N H(s^{*},t^{*})}).
	\end{equation*}
\end{lemma}
\begin{proof}
By Lemma \ref{lemma_asymptB}, we know that $s^* < \frac{1-p}{2}$. Defining
	\begin{equation*}
	\underline{H} (s^*,t^*) \, := \, \sup_{\substack{(s,t) \\ s > (1-\varepsilon)\frac{1-p}{2}}}H(s,t),
	\end{equation*}
it follows that $\, \underline{H} (s^*,t^*) \, < \, H(s^*,t^*)$. 
%Now, we fix $\ell>(1-\varepsilon)\frac{n_1}{2}\,$ and we prove that $\,Q_{\sss N}\big[\bar{Z}_{\sss M_{\sss N}}^{\sss (2)}(B)\prod_{l=2}^{\infty} c_l(B)^{N_l}\big] \, = \, o(\e^{IN})$.\\ 
Further, we define 
	\begin{equation*}
	D_{\ell,k}^{\sss (N)}(n_1,n_2,m)=B_{\ell,k}^{\sss (N)}(n_2-m)\q(M_{\sss N}=m),
	\end{equation*}
and using \eqref{law-MN} and the following bound
	\begin{equation*}
	2^{-2(m+1)}{{2(m+1)}\choose{m+1}}\leq 2^{-2m}{{2m}\choose{m}},
	\end{equation*}
we obtain
	\begin{align}\label{d_m1_su_d_m}
	\frac{D_{\ell,k}^{\sss (N)}(n_1,n_2,m+1)}{D_{\ell,k}^{\sss (N)}(n_1,n_2,m)}
	&\leq \frac{n_2-m-k}{n_1/2-\ell + n_2-m-k-1}\nn \\
	&\leq \frac{n_2}{n_1/2-\ell+n_2-1} \leq \frac{n_2}{n_2-1}.
	\end{align}
As a consequence, using \eqref{unif-bd-BN}, 
	\begin{align*}
	B_{\ell,k}^{\sss (N)}(n_2-m)\q(M_{\sss N}=m)\, &\leq \, B_{\ell,k}^{\sss (N)}(n_2)\q(M_{\sss N}=0) 	\left(\frac{n_2}{n_2-1}\right)^m\\
	&\leq a N^{1/2} \exp\{H(\ell/N,k/N)\},
	\end{align*}
since, $\left(\frac{n_2}{n_2-1}\right)^m\leq \left(\frac{n_2}{n_2-1}\right)^{n_2}\leq a$ for $m\leq n_2$ and some $a>\e$.
Therefore, using this inequality together with Lemma \ref{lem-part-function-tori}, we obtain that 
%there exists $A$ such that
	\begin{align*}
		&\,\e^{-NH(s^{*},t^{*})} \, \sum_{m=0}^{n_2}\expec_m[\bar{Z}_{m}^{\sss(2)}(B)]
		\sum_{\ell>(1-\varepsilon)\frac{n_1}{2}}^{n_1/2}\sum_{k=0}^{n_2-m}\,B_{\ell,k}^{\sss (N)}(n_2-m) \q(M_{\sss N}=m)
		\\  
		&\qquad\leq a A n_2N^{1/2} \,\e^{-NH(s^{*},t^{*})} \, \sum_{\ell>(1-\varepsilon)
		\frac{n_1}{2}}^{n_1/2}\sum_{k=0}^{n_2} \exp\{H(\ell/N,k/N)\} \\ 
		&\qquad \leq a A N^{7/2} \e^{N(\underline{H}(s^{*},t^{*})-H(s^{*},t^{*}))}\stackrel{N \rightarrow \infty}{\longrightarrow} \, 0.
\end{align*}
%because $I=H(s^{*},t^{*})$ and $\underline{H}(s^{*},t^{*})-H(s^{*},t^{*}) \,<\,0$.
\end{proof}

\noindent
Now we are finally ready for the proof of Proposition \ref{prop-exp-part-funct}. We treat first the case in the presence of an external
field $B$ and then the case without field.
\begin{proof}[Proof of Proposition \ref{prop-exp-part-funct} (a).]  We fix $\mu \in \left\{0,\ldots,n_2\right\}$ and $\varepsilon > 0$ sufficiently small.
Using  (\ref{prob_tot_prop4.1}) and Lemma \ref{lem-ZN-condM=m} we write
%and we split the sum over $m$ and $\ell$ in (\ref{prob_tot_prop4.1}) as
	\begin{equation*}
	Q_{\sss N}\big[\bar{Z}_{\sss M_{\sss N}}^{\sss (2)}(B)\prod_{l=2}^{\infty} c_l(B)^{N_l}\big] 
	\, = \, X^{\sss (1)}_{\sss N, \ell \leq (1-\varepsilon)\frac{n_1}{2}}(\mu) \, + \, X^{\sss (2)}_{\sss N,  \ell \leq 	(1-\varepsilon)\frac{n_1}{2}}(\mu) \, + \, X^{\sss (3)}_{\sss N,  \ell > (1-\varepsilon)\frac{n_1}{2}}
	\end{equation*}
%with, rewriting according Lemma \ref{lem-ZN-condM=m},
where
	\begin{align}
		X^{\sss (1)}_{\sss N, \ell \leq (1-\varepsilon)\frac{n_1}{2}}(\mu) \, = \, &\sum_{m=0}^{\mu} 	\expec_m[\bar{Z}_{m}^{\sss(2)}(B)] \sum_{\ell = 0}^{(1-\varepsilon)\frac{n_1}{2}}\sum_{k=0}^{n_2-m}\,B_{\ell,k}^{\sss 	(N)}(n_2-m)\q(M_{\sss N}=m),\\ \label{secondo_pezzo}
		X^{\sss (2)}_{\sss N, \ell \leq (1-\varepsilon)\frac{n_1}{2}}(\mu) \, = \, &\sum_{m=\mu +1}^{n_2} 	\expec_m[\bar{Z}_{m}^{\sss(2)}(B)] \sum_{\ell = 0}^{(1-\varepsilon)\frac{n_1}{2}}\sum_{k=0}^{n_2-m}\,B_{\ell,k}^{\sss 	(N)}(n_2-m)\q(M_{\sss N}=m),\\ 
		X^{\sss (3)}_{\sss N,  \ell > (1-\varepsilon)\frac{n_1}{2}} \, = \, 	&\sum_{m=0}^{n_2}\expec_m[\bar{Z}_{m}^{\sss(2)}(B)]\sum_{\ell>(1-\varepsilon)	
	\frac{n_1}{2}}^{n_1/2}\sum_{k=0}^{n_2-m}\,B_{\ell,k}^{\sss (N)}(n_2-m) \q(M_{\sss N}=m).
	\end{align}
We analyze the three pieces separately, showing that only the first of them contributes to the exponential
growth of $Q_{\sss N}\big[\bar{Z}_{\sss M_{\sss N}}^{\sss (2)}(B)\prod_{l=2}^{\infty} c_l(B)^{N_l}\big]$.
By Lemma \ref{lem-vertices-tori} and Lemma \ref{asy_M_N_equal_zero} 
	\begin{align}\label{first}
	X^{\sss (1)}_{\sss N, \ell \leq (1-\varepsilon)\frac{n_1}{2}}(\mu)  =  &{2\pi\, C(s^{*},t^{*})}\, (\det 	\mathbf{Q}(s^{*},t^{*}))^{-\frac 1 2} \exp\{NH(s^{*},t^{*})\} \nn \\
	&\times \sum_{m=0}^\mu\expec_m[\bar{Z}_{m}^{\sss(2)}(B)]
	\q(M=m)(1+o(1)).
	\end{align}
The expression in (\ref{secondo_pezzo}) can be rewritten as
	\begin{align*}
	X^{\sss (2)}_{\sss N, \ell \leq (1-\varepsilon)\frac{n_1}{2}}(\mu)  
	\, &= \, \sum_{m=\mu +1}^{n_2}\expec_m[\bar{Z}_{m}^{\sss(2)}(B)]
	\sum_{\ell = 0}^{(1-\varepsilon)\frac{n_1}{2}}\sum_{k=0}^{n_2-m}\,\frac{B_{\ell,k}^{\sss (N)}(n_2-m) 
	\q(M_{\sss N}=m)}{B_{\ell,k}^{\sss (N)}(n_2)\q(M_{\sss N}=0)}\nn\\
	&\qquad \times B_{\ell,k}^{\sss (N)}(n_2)\q(M_{\sss N}=0).
	\end{align*}
Now, by \eqref{d_m1_su_d_m}, $\frac{B_{\ell,k}^{\sss (N)}(n_2-m) \q(M_{\sss N}=m)}{B_{\ell,k}^{\sss (N)}(n_2-(m-1))\q(M_{\sss N}=m-1)}$ is uniformly bounded by $1-\delta$ for $\delta>0$ sufficiently small, because $\ell \leq (1-\vep)n_1/2$. Using this bound together with Lemma \ref{lem-part-function-tori} yields
	\begin{equation*}
	X^{\sss (2)}_{\sss N, \ell \leq (1-\varepsilon)\frac{n_1}{2}}(\mu) \,  
	\leq \, A \,  \sum_{\ell=0}^{n_1/2}\sum_{k=0}^{n_2} B_{\ell,k}^{\sss (N)}(n_2)\sum_{m=\mu +1}^{n_2} (1-\delta)^m.
	\end{equation*}
Thus there exist $\varepsilon(\mu)$ (with $\varepsilon(\mu)\to 0$ as $\mu\to\infty$) such that, uniformly in N,
	\begin{equation*}
	X^{\sss (2)}_{\sss N, \ell \leq (1-\varepsilon)\frac{n_1}{2}}(\mu) \,  
	\leq \, \varepsilon(\mu )\,\exp\{NH(s^{*},t^{*})\} 	(1+o(1)).
	\end{equation*}
Finally, from Lemma \ref{contribution_l_magg} it results that $X^{\sss (3)}_{\sss N,  \ell > (1-\varepsilon)\frac{n_1}{2}} \, = \, o(\exp\{NH(s^{*},t^{*})\})$. Thus,  {from \eqref{first}} we can identify $I \, = \, H(s^{*},t^{*})$ and
	\begin{equation*}
	J\, = \,  {2 \pi\, C(s^{*},t^{*})}\, (\det \mathbf{Q}(s^{*},t^{*}))^{-\frac 1 2} 
	\sum_{m=0}^{\infty} \expec_m[\bar{Z}_{m}^{\sss(2)}(B)](b^*)^{m}\prob(M=m).
	\end{equation*}
{We remark the previous expression is well-defined since}, from Lemma \ref{lem-part-function-tori}, 
	$$
	\sum_{m=0}^{\infty} \expec_m[\bar{Z}_{m}^{\sss(2)}(B)](b^*)^{m}\prob(M=m) \, \leq \, A \, \expec\left[\right(b^*)^{M}],
	$$ 
which is finite because $b^* < \frac{1+p}{2p}$ (see \eqref{mom_gen_func_b} in Lemma \ref{lem-vertices-tori} and Remark \ref{bstar_less_bc}).
\medskip

%\vspace{0.3cm}
%\MLP{Check carefully!!}
\noindent
{\em Proof of Proposition \ref{prop-exp-part-funct} (b)}.
In this case we work with a vanishing external field. Defining   $t_{\sss N}:=\frac{t}{\sqrt{N}}$
%$$have $B=0$ and $B_{\sss N} = \frac{t}{\sqrt{N}}$. 
and by Taylor expanding (\ref{a-def}) around $0$, we have
	\begin{equation*}
	a\left(t_{\sss N}\right) 
	\,:=\, \frac{A_-\left(t_{\sss N}\right)}{A_+\left(t_{\sss N}\right)} \,=\, C \frac{t^2}{N}(1+o(1)), \qquad \text{as} \quad N\to\infty,
	\end{equation*}
where $C$ is a constant (whose value actually depends on $\beta$).
% and to shorten notation we use $t_{\sss N}:=\frac{t}{\sqrt{N}}$. 
Thus,
	\begin{align*}
		Q_{\sss N}[\prod_{l=2}^{\infty} c_l(t_{\sss N})^{N_l}\mid M_{\sss N}=m]
		&= Q_{\sss N}[\prod_{l=2}^{\infty} \left(1+a(t_{\sss N})r^l(t_{\sss N})\right)^{N_l}\mid M_{\sss N}=m]\\
		&= Q_{\sss N}\left[\e^{\sum_{l=2}^{\infty}a(t_{\sss N})r^l(t_{\sss N}) N_l} (1+o(1)) \mid M_{\sss N}=m\right]\\ 
		&= Q_{\sss N}\left[\e^{C t^2\sum_{l=2}^{\infty}r^l(t_{\sss N}) 
		p_l^{\sss{(N)}}}\mid M_{\sss N}=m\right](1+o(1)).
	\end{align*}
By writing 
		\eqan{
		\label{split-expec}
		&Q_{\sss N}\left[\e^{C t^2\sum_{l=2}^{\infty}r^l(t_{\sss N}) 
		p_l^{\sss{(N)}}}\mid M_{\sss N}=m\right]\\
		&\qquad=\e^{C t^2\sum_{l=2}^{\infty}r^l(t_{\sss N}) 
		Q_{\sss N}(p_l^{\sss{(N)}})}\nn\\
		&\qquad \quad +\e^{C t^2\sum_{l=2}^{\infty}r^l(t_{\sss N}) 
		Q_{\sss N}(p_l^{\sss{(N)}})}
		Q_{\sss N}\left[\e^{C t^2\sum_{l=2}^{\infty}r^l(t_{\sss N}) 
		(p_l^{\sss{(N)}}-Q_{\sss N}(p_l^{\sss{(N)}}))}-1\mid M_{\sss N}=m\right],\nn
		}
and using formula  (\ref{prob_tot_prop4.1}), we can rewrite $Q_{\sss N}\big[\bar{Z}_{\sss M_{\sss N}}^{\sss (2)}(t_{\sss N})\prod_{l=2}^{\infty} c_l(t_{\sss N})^{N_l}\big]=S_1(N)+S_2(N)$ as the sum of two contributions, due to the two terms in \eqref{split-expec}. 
%	\begin{align}\nn
%		&\sum_{m=0}^{n_2} \expec_m[\bar{Z}_{m}^{\sss(t)}(t_{\sss N})] 
%		Q_{\sss N}\left[\e^{C(\beta) t^2\sum_{l=2}^{\infty}r^l(t_{\sss N}) p_l^{\sss{(N)}}}\mid M_{\sss N}=m\right] 
%		\q(M_{\sss N}=m)(1+o(1))\\ \nn
% 	   	&\quad=\sum_{m=0}^{n_2} \expec_m[\bar{Z}_{m}^{\sss(t)}(t_{\sss N})] 
%	   	Q_{\sss N}\left[\e^{C(\beta) t^2\sum_{l=2}^{\infty}r^l(t_{\sss N}) 
%	   	\left(p_l^{\sss{(N)}} - \q(p_l^{\sss{(N)}})\right)}\mid M_{\sss N}=m\right] 
%	   	\q(M_{\sss N}=m) \e^{C(\beta) t^2\sum_{l=2}^{\infty}r^l(t_{\sss N})\q(p_l^{\sss{(N)}})}(1+o(1))\\ \nn
% 	  	&\quad=\e^{C(\beta) t^2\sum_{l=2}^{\infty}r^l(t_{\sss N})
%		\q(p_l^{\sss{(N)}})}\left[\sum_{m=0}^{n_2} \expec_m[\bar{Z}_{m}^{\sss(t)}(t_{\sss N})] 
%		\q(M_{\sss N}=m) (1+o(1))\right.\\ \nn
%		& \left. +\sum_{m=0}^{n_2} \expec_m[\bar{Z}_{m}^{\sss(t)}(t_{\sss N})] 
%		Q_{\sss N}\left[\e^{C(\beta) t^2\sum_{l=2}^{\infty}r^l(t_{\sss N}) 
%		\left(p_l^{\sss{(N)}} - \q(p_l^{\sss{(N)}})\right)} -1 \mid M_{\sss N}=m\right] \q(M_{\sss N}=m) (1+o(1))\right ].\nn
%	\end{align}
Now we analyze $S_1(N)$ and $S_2(N)$ as $N\to \infty$.  First, we remark that the sum in the exponential factor converges in this limit.  This can be shown by observing that $r(B)<1$ (see \eqref{defalpha}). Therefore, calling $r^{*}=r(0)$ and given any $\epsilon>0$ such that $r^{*}+\varepsilon<1$,  thanks to the convergence of $r(t_{N})$ to $r^{*}$, we have that for all $N$ sufficiently large,
	$$
	\sum_{l=2}^{\infty}r^l(t_{\sss N})\q(p_l^{\sss{(N)}})\equiv \sum_{l=2}^{N}r^l(t_{\sss N})\q(p_l^{\sss{(N)}})\le 	\sum_{l=2}^{N}(r^{*}+\varepsilon)^{l}
	$$  
where we used the fact that $p_l^{\sss{(N)}}\le 1$  for $l\le N$ and $p_l^{\sss{(N)}}=0$ for all $l>N$. Since the geometric sum in the r.h.s.\ of the previous display is convergent,  the positive series in the l.h.s.  is also convergent to
some positive value $\bar{I}_{0}$.
{Thus, by inserting the first term of the r.h.s. of \eqref{split-expec} in \eqref{prob_tot_prop4.1} and applying bounded convergence, we obtain \eqref{IeJBN0} with} 
	\begin{equation*}
	\bar{I}\left(t \right) = \bar{I}_{0} C t^{2}. 
	\end{equation*}
and	
	\begin{equation*}
  \bar{J} = \sum_{m=0}^{\infty} \expec_m[\bar{Z}_{m}^{\sss(2)}(0)] {\mathbb P}(M=m).
	\end{equation*}
%while the second is vanishing.
%From Lemma 4.1 of \cite{GGPvdH1}, we know that as $N\rightarrow \infty$, $\q(p_l^{\sss{(N)}})\rightarrow p_l^*$. Using this, together with the fact that $r(t_{\sss N}) <1$ and then $r^l(t_{\sss N})$ is exponentially small, we obtain that the first addend behaves asymptotically as:
%\be
%%\e^{C(\beta) t^2\sum_{l=2}^{\infty}r^l(t_{\sss N})\q(p_l^{\sss{(N)}})} \sum_{m=0}^{n_2} \expec_m[\bar{Z}_{m}^{\sss(t)}(t_{\sss N})] \q(M_{\sss N}=m) \rightarrow 
%(1+o(1))\,\e^{N C(\beta) \frac{t^2}{N}\sum_{l=2}^{\infty}r^l(t_{\sss N})p_l^*} \sum_{m=0}^{\infty} \expec_m[\bar{Z}_{m}^{\sss(t)}(t_{\sss N})] \q(M=m),
%\ee
%and then we can identify
%\be
%\bar{I}\left(\beta, t_{\sss N} \right) = C(\beta) {t_{\sss N}}^2\sum_{l=2}^{\infty}r^l(t_{\sss N})p_l^*, \quad \quad \quad \bar{J}\left(\beta, t_{\sss N} \right) = \sum_{m=0}^{\infty} \expec_m[\bar{Z}_{m}^{\sss(t)}(t_{\sss N})] \q(M=m).
%\ee
Further, by Lemma \ref{lem-part-function-tori} and the law of total expectation, 
	\begin{equation*}
	S_2(N)\leq A \e^{C t^2\sum_{l=2}^{\infty}r^l(t_{\sss N})
	\q(p_l^{\sss{(N)}})} 
	Q_{\sss N}\Big[\Big|\e^{C t^2\sum_{l=2}^{\infty}r^l(t_{\sss N}) \left(p_l^{\sss{(N)}} - \q(p_l^{\sss{(N)}})\right)} 
	-1\Big|\Big].
	\end{equation*}
We use the Cauchy-Schwarz inequality to bound
	\begin{align}
	\label{CS-bd}
	&Q_{\sss N}\Big[\Big|\e^{C t^2\sum_{l=2}^{\infty}r^l(t_{\sss N}) \left(p_l^{\sss{(N)}} - \q(p_l^{\sss{(N)}})\right)} 
	-1\Big|\Big]\\
	&\leq
	Q_{\sss N}\Big[\Big(\e^{C t^2\sum_{l=2}^{\infty}r^l(t_{\sss N}) \left(p_l^{\sss{(N)}} - \q(p_l^{\sss{(N)}})\right)} 
	-1\Big)^2\Big]^{1/2},\nn
	\end{align}
and, by Jensen's inequality and  H\"{o}lder inequality,
\begin{align*}
	1 &\leq Q_{\sss N}\left[
	\e^{C t^2 \sum_{l=2}^{\infty}r^l(t_{\sss N}) \left(p_l^{\sss{(N)}} - \q(p_l^{\sss{(N)}})\right)} \right]\\
	&\leq \left(Q_{\sss N}\left[
	\e^{C t^2 \sqrt{N}\sum_{l=2}^{\infty}r^l(t_{\sss N}) \left(p_l^{\sss{(N)}} - \q(p_l^{\sss{(N)}})\right)}
	\right] \right)^{\frac{1}{\sqrt{N}}} 
	\leq 1+o(1),
	\end{align*}
due to  the existence of the finite limit of $Q_{\sss N}\left[\e^{C t^2\sqrt{N} \sum_{l=2}^{\infty}r^l(t_{\sss N}) \left(p_l^{\sss{(N)}} - \q(p_l^{\sss{(N)}})\right)}  \right]$ by \cite[Lemma 4.3]{GGPvdH1}.
Therefore, also the term in \eqref{CS-bd} converges to 0, showing that $S_2(N)$ gives a vanishing contribution.
This completes the proof.
\end{proof}

\subsection{Annealed thermodynamic limits and SLLN: Proof of Theorems \ref{ann_press2_CM12} and \ref{slln_ann-CM12}}
\noindent Finally, we prove the existence of the thermodynamic limits.
\begin{proof}[Proof of Theorem \ref{ann_press2_CM12}.]
The thermodynamic limit of the annealed pressure is given by
	\begin{equation*}
	\widetilde{\psi} (\beta, B) 
	\, = \, \lim_{N \rightarrow \infty} \, \widetilde{\psi}_{\sss N} (\beta, B) 
	\, =  \, \lim_{N \rightarrow \infty} \, \frac{1}{N} \log \left(\q \left(Z_{\sss N} (\beta, B)\right)\right).
	\end{equation*}
From \eqref{rewrite-part-func-12} we can rewrite
	\begin{equation*}
	\q \left(Z_{\sss N} (\beta, B)\right)
	\, = \, \lambda_+^N(B)A_+^{n_1/2}(B) Q_{\sss N}\Big[\bar{Z}_{\sss M_{\sss N}}^{\sss (2)}(B)
	\prod_{l=2}^{\infty} c_l(B)^{N_l}\Big],
	\end{equation*}
and then
	\begin{equation*}
	\widetilde{\psi} (\beta, B) \, = \, \log{\lambda_+} + \frac{1-p}{2}\log{A_+} 
	+  \lim_{N \rightarrow \infty} \, 
	\frac{1}{N} \log\Big\{Q_{\sss N}\Big[\bar{Z}_{\sss M_{\sss N}}^{\sss (2)}(B)\prod_{l=2}^{\infty} c_l(B)^{N_l}\Big]\Big\}.
	\end{equation*}
%Using the previous Lemmas and bounding from above $\expec[\bar{Z}_{m}^{\sss(t)}(B)] \leq A$, we obtain
%\be
%\frac{1}{N} \log\left\{Q_{\sss N}\left[\bar{Z}_{\sss M_{\sss N}}^{\sss (2)}(B)\prod_{l=2}^{\infty} c_l(B)^{N_l}\right]\right\} \, \leq \, \frac{1}{N} \log\left\{A \exp\left\{NH(s^{*},t^{*})\right\}\frac{\pi\, C(s^{*},t^{*})}{\sqrt{(\det \mathbf{Q}(s^{*},t^{*}))}}(1+o(1))\expec[b^{M_{\sss N}}]\right\}.
%\ee
%The lower bound is achieved similarly, using $\expec[\bar{Z}_{m}^{\sss(t)}(B)] \geq 1$.\\ 
Using Proposition \ref{prop-exp-part-funct} 
%and taking the limit $N$ to infinity, we arrive at the expression for the thermodynamic limit of the annealed pressure as
we find
	\begin{equation*}
	\widetilde{\psi} (\beta, B) \, = \, \log{\lambda_+(\beta,B)} + \frac{1-p}{2}\log{A_+(\beta,B}) + H(s^*,t^*).
	\end{equation*}
To prove the existence of the thermodynamic limit of the magnetization, we use Lemma \ref{limit_deriv} and the existence of the pressure in the thermodynamic limit. Then, remembering that $(s^*,t^*)$ is the maximum point of the function $H(s,t)$, we compute
	\begin{align}\label{magn_12_expl}\nn
		\widetilde{M} (\beta, B) 
		= \frac{\partial}{\partial B} \widetilde{\psi} (\beta, B)  
		=  &\frac{\partial}{\partial B} \log{\lambda_+(\beta, B)} 
		+ \frac{1-p}{2} \frac{\partial}{\partial B}\log{A_+}(\beta, B)\nn\\
		&+ s^*(\beta, B) \frac{\partial}{\partial B} \log{\left[a(\beta, B)r^2(\beta, B)\right]}\nn\\ 
		& + t^*(\beta, B) \frac{\partial}{\partial B} \log{r(\beta, B)} 
	\end{align}
In the limit of small external field $B$ by Taylor expanding \eqref{a-def} one has $a(\beta,B) = O(B^2)$. Also,
from the fixed point equations \eqref{syst_crit_point} one can check that $s^*(\beta,B) = O(B^2)$.  
As a consequence  $\lim_{B\rightarrow 0^+} \widetilde{M} (\beta, B) = 0$ for all $\beta >0$, 
and therefore, by the definition in (\ref{def_beta_c}), we conclude that there is no phase transition for $\CMNonetwo$.
\end{proof}

\medskip

\noindent
{\it Proof of Theorem \ref{slln_ann-CM12}:} Again, the SLLN follows immediately from the existence of the annealed pressure in the thermodynamic limit and its differentiability with respect to $B$. See the proof of  Theorem \ref{slln_ann} and \cite[Section 2.2]{GGPvdH1}.
\qed
\bigskip

\vspace{0.2cm}
\noindent
{\small
{\bfseries Acknowledgments.}
We are grateful to Aernout van Enter for lively discussions about the interpretation of the annealed measure and its relation to spin and graph dynamics, as well as for his reference to the Hubbard-Stratonovich identity. We thank Institute Henri Poincar\'e for the hospitality during the trimester "Disordered systems, random spatial processes and their applications''. 
We acknowledge financial support from  the Italian Research Funding Agency (MIUR) through FIRB project 
%``Stochastic processes in interacting particle systems: duality, metastability and their applications'', 
grant n.\ RBFR10N90W.
%\MLP{Check Acknowledgments and add acknowledgments to de Panafieu and Broutin}
The work of RvdH is supported in part by the Netherlands
Organisation for Scientific Research (NWO) through VICI grant 639.033.806 and the Gravitation {\sc Networks} grant 024.002.003.}

%\listoftodos

%%%%%%%%%%%%%%%%%%%%%%%

\end{document}